\documentclass[10pt]{amsart}

\usepackage[T1]{fontenc}
\usepackage[lf]{Baskervaldx} 
\usepackage[bigdelims,vvarbb]{newtxmath} 
\usepackage[cal=boondoxo]{mathalfa} 

\usepackage[normalem]{ulem}
\usepackage{latexsym} 

\usepackage{lipsum} 

\usepackage{amsmath}			
\usepackage{amsfonts}			
\usepackage{amsthm}
\usepackage{thmtools}			
\usepackage{mathtools}

\usepackage{shuffle}		
\usepackage{mathrsfs}

\usepackage{graphicx}
\usepackage{float}	
\usepackage{comment} 
\usepackage{layout}
\usepackage{color}			
\usepackage{listings} 

\usepackage{tikz}				
\usepackage{tikz-cd}
\usetikzlibrary{decorations.pathmorphing,arrows,arrows.meta,calc,shapes,shapes.geometric,decorations.pathreplacing,knots}
\tikzcdset{arrow style=tikz, diagrams={>=stealth}}

\usepackage{pgfplots}
\usepackage{tabularx}
\usepackage{subfig}

\usepackage{enumitem}		
\setitemize[0]{font=\textbf, label=--}

\usepackage{color}
\definecolor{BleuTresFonce}{rgb}{0.215, 0.215, 0.36}

\usepackage[colorlinks,final,backref=page,hyperindex]{hyperref}
\hypersetup{citecolor=BleuTresFonce, linkcolor=black}

\usepackage[nameinlink]{cleveref}

\usepackage{fullpage}
\usepackage{setspace}			
\usepackage{epigraph}
\usepackage{wrapfig}

\usepackage{stmaryrd}

\theoremstyle{plain}
\newtheorem{theorem}{Theorem}

\newtheorem{maintheorem}{Theorem}

\newtheorem{lemma}[theorem]{Lemma}
\newtheorem{proposition}[theorem]{Proposition}
\newtheorem*{proposition*}{Proposition}
\newtheorem{corollary}[theorem]{Corollary}

\theoremstyle{definition}
\newtheorem{definition}[theorem]{Definition}
\newtheorem*{definition*}{Definition}
\newtheorem{defiprop}[theorem]{Definition/Proposition}
\newtheorem{example}[theorem]{Example}
\newtheorem{notation}[theorem]{Notation}
\newtheorem{remark}[theorem]{Remark}
\newtheorem*{remark*}{Remark}

\newcommand{\R}{\mathbb{R}}
\newcommand{\Z}{\mathbb{Z}}
\newcommand{\N}{\mathbb{N}}
\newcommand{\Q}{\mathbb{Q}}

\renewcommand{\H}{\mathrm{H}}

\newcommand{\id}{\mathrm{id}}

\newcommand{\bbC}{\mathbb{C}}
\newcommand{\bbF}{\mathbb{F}}
\newcommand{\bbN}{\mathbb{N}}

\newcommand{\bbR}{\mathbb{R}}

\newcommand{\sfC}{\mathsf{C}}
\newcommand{\sfD}{\mathsf{D}}

\newcommand{\sfR}{\mathsf{I\!\!\:R}} 

\newcommand{\rmB}{\mathrm{B}}
\newcommand{\rmC}{\mathrm{C}}

\newcommand{\rmH}{\mathrm{H}}
\newcommand{\rmN}{\mathrm{N}}
\newcommand{\rmP}{\mathrm{P}}
\newcommand{\rmT}{\mathrm{T}}

\newcommand{\calA}{\mathcal{A}}
\newcommand{\calC}{\mathcal{C}}
\newcommand{\calE}{\mathcal{E}}
\newcommand{\calO}{\mathcal{O}}
\newcommand{\calR}{\mathcal{R}} 
\newcommand{\calS}{\mathcal{S}} 
\newcommand{\calT}{\mathcal{T}}

\newcommand{\Sq}{\mathrm{Sq}}

\renewcommand{\phi}{\varphi}
\newcommand{\eps}{\varepsilon}

\newcommand{\As}{\mathcal{A}\mathrm{s}}
\newcommand{\Com}{\mathcal{C}\mathrm{om}}

\newcommand{\Cech}{\check{\calC}}
\newcommand{\Rips}{\calR}

\newcommand{\Func}{\mathrm{Func}}
\newcommand{\Hom}{\mathrm{Hom}}

\newcommand{\Sing}{\mathrm{Sing}}
\newcommand{\DCpx}{\Delta \mathsf{Cpx}}
\newcommand{\Cpx}{\mathsf{Cpx}}
\newcommand{\sSet}{\mathsf{sSet}}

\newcommand{\grVect}{\mathsf{gr}\mbox{-}\mathsf{Vect}}
\newcommand{\Ch}{\mathsf{Ch}}
\newcommand{\coCh}{\mathsf{coCh}}

\renewcommand{\Top}{\mathsf{Top}}

\newcommand{\Alg}{\mathsf{Alg}}
\newcommand{\AsAlg}{\Alg_{\As}}
\newcommand{\ComAlg}{\Alg_{\Com}}
\newcommand{\AinfAlg}{\Alg_{\calA_\infty}}

\newcommand{\Trans}{\mathsf{Trans}}
\newcommand{\fTrans}{\mathsf{fTrans}}

\newcommand{\ho}{\mathsf{ho}}

\newcommand{\op}{\mathrm{op}}

\newcommand{\For}{\mathrm{forget}}

\newcommand{\APL}{\mathrm{A}_{\mathrm{PL}}}
\newcommand{\Alpha}{\mathcal{A}}

\renewcommand{\d}{\mathrm{d}}

\newcommand{\dGH}{\mathrm{d}_{\mathcal{GH}}}

\renewcommand{\inf}{\mathrm{inf}}
\renewcommand{\min}{\mathrm{min}}
\renewcommand{\max}{\mathrm{max}}
\renewcommand{\sup}{\mathrm{sup}}
\newcommand{\dist}{\mathrm{dist}}

\newcommand{\RP}{\mathbb{R}\mathrm{P}^2}

\title{Multiplicative persistent distances}
\address{LAGA, Universit\'e Paris 13, 99 Avenue Jean Baptiste Cl\'ement 93430, Villetaneuse, France}

\author{Gr\'egory \textsc{Ginot}} 
\author{Johan \textsc{Leray}} 
\email{ginot@math.univ-paris13.fr, leray@math.univ-paris13.fr}

\date{\today}

\dedicatory{"Bats-toi, signe et persiste" --  France Gall}

\thanks{The authors acknowledge support from the project ANR-16-CE40-0003 ChroK. The second author is financed by a postdoctoral allocation of the DIM Math Innov - R\'egion \^Ile de France.}

\begin{document}
	
\begin{abstract}
	We define and study several new interleaving distances for persistent cohomology which take into account the algebraic structures of the cohomology of a space, for instance the cup product or the action of the Steenrod algebra. In particular, we prove that there exists a persistent  $\calA_{\infty}$-structure associated to data sets and  and we define the associated distance. We prove the stability of these new distances for \v{C}ech or Vietoris Rips complexes with respect to the Gromov-Hausdorff distance, and we compare these new distances with each other and the classical one, building some examples which prove that they are not equal in general and refine effectively the classical bottleneck distance. 
\end{abstract}
	
\maketitle

\section*{Introduction}

Persistent homology arised as a successful attempt to make invariants of algebraic topology computable in practice in various contexts. A prominent example being to  study data sets and their topology, which have become increasingly important in many area of sciences. 
In particular, to be able to discriminate and compare large data sets, it is natural  to associate invariants to each of them in order to be able to say if they  are similar and  describe similar phenomenon or not.  The latter operation is obtained by considering a metric on the invariants associated to the data which,  classically, is the interleaving or bottleneck distance on the  persistent homology of the data. The interested reader may consult \cite{Oudo15,Edel10} for an extended discussion of the theory and of its many applications.
Our goal is to study and compare several refinements of those distances obtained by considering more structure, inspired by homotopical algebra, on the persistent cohomology which discriminate more data sets. 
 
\subsection*{Topological data analysis}
Associating algebraic invariants to \emph{shapes} is a main apparatus of algebraic topology. \emph{Topological data analysis} (TDA for short) associates and studies the topology of data sets through the help of algebraic topology invariants characterizing as finely as possible the data. Roughly, a main idea of TDA is to associate to, a  potentially large, set ${X}$ of $N$ points a family of spaces  $X_\eps$ given by the union of balls centered on each point  with radius given by the parameter $\eps$.  Now we can consider the invariants of each space  but, even better, we can study the set $\{X_\eps\}$ as a continuous family of spaces, called a  \emph{persistent space}, and considering the evolution of these invariants when $\eps$ grows. The more accessible topological invariant is the homology of these spaces also known as \emph{persistent homology}.

\subsection*{Persistent homology }
The homology of  a persistent space gives us a parametrized family of graded vector space. To such object, one associates a \emph{barcode}, which represents the evolution  of the dimension of each homology group when the parameter varies. For instance, a $i^{\mathrm{th}}$-homology class can be born at the time $\eps_1$ in the $i$-th group and dies at  time $\eps_2$. This class is associated to a bar of length $\eps_2-\eps_1$ and the collection of those is the \emph{barcode} of the persistent homology groups.   
This barcode defines a invariant of the persistent space $\{X_\eps\}$. To compare two barcodes $B_X$ and $B_Y$ associated to datasets ${X}$ and ${Y}$, Cohen \emph{et al.} defined in \cite{cohen2007stability} the \emph{bottleneck distance} which is a (pseudo-)distance. A more intrinsic notion of distance, directly defined on the persistent (co)homology, is given by the  \emph{interleaving} distance introduced by Chazal \emph{et al.} in~\cite{chazal2014stability}. 

\subsection*{Applications}
These techniques of TDA have been applied in many areas: for example in reconstruction of shapes (see \cite{Towards,Shape}), gene expression analysis (see for example \cite{jeitziner2017two}), or in neurosciences (see \cite{kanari2018topological}). It should be noted that the bottleneck distance is only slightly or not sensitive to data noise.

\subsection*{Content of the paper}
Algebraic topologists have constructed several other invariants finer than homology of the space: for instance the cohomology has a natural graded \emph{algebra} structure induced by the cup product. This structure is itself a shadow of the differential graded algebra structure carried by the \emph{cochains} which is a better invariant, see \Cref{ex:borromean}. These refined invariants thus encode in a much more effective way the homotopy type of a space compared to mere homology and in fact, the homotopy of a (nilpotent finite type) space is completely encoded by this dg-algebra structure together with higher homotopies for its commutativity, i.e., its $\calE_\infty$-algebra structure~\cite{mandell2006cochains}. 
  
In this paper, we give the theoretical framework to construct and compare \emph{new interleaving distances} (on data sets) which take into account these extra algebraic structures in the persistence setting, in a \emph{systematic way}. In particular we exhibit a \emph{hierarchy} of such refined distances, see theorem~\ref{thm:maintheoremB}. We follow here the general principle of interleaving distance associated to persistent objects in any category $\sfC$ as defined by Bubenik \emph{et al.} in \cite{bubenik2014categorification}. Our distances refine the classical interleaving distance $\d_{\grVect}$ but are of different computational difficulties. In fact, we have the following commutative diagrams of functors: 
\begin{itemize}
	\item for $p$ a prime, we have
	\[
	\begin{tikzcd}
		& \Top^\op
		\ar[dl,"{\rmH^*(-,\bbF_p)}" description, bend right]
		\ar[d] \ar[dr, bend left]
		& \\
		\calA_p\mbox{-}\Alg \ar[r,"\mathrm{forget}"'] & \AsAlg \ar[r,"\mathrm{forget}"'] & \grVect
	\end{tikzcd}
	\]
	where $\calA_p\mbox{-}\Alg$ is the category of algebras over the Steenrod algebra $\calA_p$~see \Cref{sec:Steenrod_distance};
	\item denoting $\ho(\AsAlg)$ the derived category of dg-algebras (\ref{nota:hocat}),  we have
	\[
	\begin{tikzcd}
		& \Top^\op 
		\ar[dl,"{\rmC^*(-,k)}" description, bend right]
		\ar[d] \ar[dr, bend left]
		& \\
		\ho(\AsAlg) \ar[r,"\rmH^*(-)"'] & \AsAlg \ar[r,"\mathrm{forget}"'] & \grVect .
	\end{tikzcd}
	\]
\end{itemize}
\emph{Each} of the categories in the lines of the diagram \emph{gives rise to an interleaving distance}.

Also, several of the more refined interleaving distances introduced above are computed in the homotopy categories of cochains  with some extra structure and not a cohomological level for which we have the barcode decomposition. The homotopy category of such cochain algebra is hard to study and cochain algebras are too big to be easily used on a computer at the moment.

We bypass this problem by using the \emph{homotopy transfer theorem} for $\calA_{\infty}$-algebras  which allows to encode the cochain algebra on the cohomology groups \emph{without losing information}. Indeed, an $\calA_\infty$-algebra is an associative algebra up to homotopy (see \Cref{def:Ainfty}) and, for all topological spaces $X$, the singular cochain complex of $X$ is equivalent of the cohomology of $X$ as $\calA_\infty$-algebras (see \Cref{thm::HTT}). Further, unlike quasi-isomorphisms of dg-algebras, $\calA_\infty$-quasi-isomorphisms have inverses which simplifies greatly the study of interleaving in these category. However transfer theorems are not very functorial and therefore it is unlikely that they can be applied to abstract persistence spaces in general (see \Cref{R:notransfer}).  

Nevertheless, for persistent simplicial sets $X\colon \sfR \to \sSet$ satisfying  some  mild finiteness assumption (that we call \emph{finite filtered data}, see \Cref{def:ffdata}, this finiteness assumption being satisfied by Vietoris-Rips, alpha and  \v{C}ech complexes associated to any finite set of points), we prove the following theorem.

\begin{maintheorem}[$\calA_\infty$ interleaving distance (see  \Cref{prop:functor_to_Ainfty})]
There exists a persistent $\calA_\infty$-structure on the persistent cohomology of finite filtered data and an interleaving distance $\d_{\calA_\infty}$, which refine the cohomology algebra and takes into  account higher Massey products of singular cohomology.
\end{maintheorem}

Some approaches to use $\calA_\infty$-structures for persistence have already been considered in the literature, notably in the work of F. Belch\'i \emph{et al.} (see \cite{belchi2015,belchi2015,belchi2017,belchi2019}) and Herscovich (cf. \cite{Herscovich2018}). In both cases, they consider transferred structures but do not consider the full either $\calA_\infty$ or persistent structure  and in particular do not define an associated interleaving distance.

The homology of a space is sensitive to the characteristic of its coefficient and this reflects on the additional algebraic structure of the cochains. Therefore, we define (see \Cref{sec:Steenrod_distance}) two new distances: the first $\d_{p_\infty}$ given by the maximum between the distance defined by the structure of Steenrod module and the $A_\infty$-structure of the cohomology with coefficient in $\bbF_p$ (\Cref{sec:Steenrod_distance}) where $p$ is a fixed prime or $0$ (with the notation $\bbF_0=\Q$) and where there is no Steenrod structure; the second $\d_{\mathbb{P}}$ is given by the supremum of $\d_{\calA_{p_\infty}}$ over the set $\mathbb{P}=\{0,p \mbox{ prime}\}$~. Our second main contribution is the comparison of these distances. 

\begin{maintheorem}[see \Cref{sec::resume}]\label{thm:maintheoremB}
	All distances defined for finite filtered data in this paper satisfy the following inequalities:
	\[
	\begin{tikzcd}[row sep = small,  column sep = large]
	&&&
	\d_{\calA_{p}-\As}
	\ar[rd,phantom, "\geqslant",sloped, description]
	\ar[rd,phantom, "_{(\ref{item:4})}", shift right=2.2ex]
	&& \\
	\d_{\mathbb{P}}
	\ar[r,phantom, "\geqslant" sloped, description] 
	&
	\d_{p_\infty,q_\infty}
	\ar[r,phantom, "\geqslant" sloped, description] 
	&
	\d_{\calA_{p_\infty}} 
	\ar[rd,phantom, "\geqslant" sloped, description] 
	\ar[rd,phantom, "_{(\ref{item:7})}", shift right=2ex]
	\ar[ru,phantom, "\geqslant" sloped, description] 
	\ar[ru,phantom, "_{(\ref{item:5})}", shift right=2ex]
	&~&
	\d_{\As,\bbF_p} 
	\ar[r,phantom, "\geqslant" sloped, description] 
	\ar[r,phantom, "_{(\ref{item:1})}", shift right=2ex]
	& \d_{\grVect,\bbF_p} \\
	&&&
	\d_{\calA_\infty,\bbF_p} 
	\ar[ru,phantom, "\geqslant" sloped, description] 
	\ar[ru,phantom, "_{(\ref{item:5})}", shift right=2ex]
	&&
	\end{tikzcd} \ .
	\]
	which are not equalities in general.
\end{maintheorem}

We believe that these refined distances are reasonable approximation of the most refined of all, that is the interleaving distance associated to the $\calE_\infty$-structure of cochains, see~\Cref{sec:Einfty} or the even finer (but not algebraic) interleaving distance in the homotopy category of spaces studied by Blumberg and Lesnick~\cite{BL17}. Unlike this latter one, they are defined on the underlying  persistent \emph{cohomology groups} and therefore they seem much more \lq\lq{}computerizable\rq\rq{}. In particular, there are algorithms to compute the ${\calA_{2_\infty}}$-structure. 
In the category $\grVect$ of graded vector spaces,  the classical interleaving distance is easily computable because this distance is equal to the bottleneck distance by the \emph{isometry theorem} (see \cite[Theorem 3.1]{Oudo15}) which can be computed using matchings.
This is not yet the case for the new distances we define in this paper, 
but for some of them, one can use similar idea to compute explicit effective bounds for them (see \Cref{remark:computation_multiplicative_distance} and \Cref{remark:computation_Ainfty_distance}); see  \Cref{ex:torus_sphere}, \Cref{ex:torus_sphere_comput}, \Cref{ex:Borromeandiscret} and \Cref{ex:borromean_compute} for explicit or numerical examples.

Further, we prove refined stability theorems for those distances. Indeed all these distances satisfy a stability property for \v{C}ech or Vietoris Rips complexes (see \Cref{sec:CechRips}) with respect to the Gromov-Hausdorff distance.

\begin{maintheorem}[Stability results (see \Cref{thm:stabilite}, \Cref{thm:stability_Einfty_version})]
	Let $X$ and $Y$ be two finite set of points of $\R^n$. We have the following inequality:
	\begin{align*}
	\d_{\dagger}(\calR(X),\calR(Y)) \leqslant 2\dGH\left(X,Y\right) \ ; \\
	\d_{\dagger}(\Cech(X),\Cech(Y)) \leqslant 2\dGH\left(X,Y\right) \ ,
	\end{align*}
	where $\d_\dagger$ is one of the distances of \Cref{thm:maintheoremB}.
\end{maintheorem}

Such a theorem is very important for applications since, for set of points $X$ and $Y$ representing the same data up to some noise, this theorem implies that if the noise if small then the distance $	\d_{\dagger}(\calR(X),\calR(Y))$ is also small. 

\begin{remark*} 
	This work is a first theoretical step towards new distances taking into account more topological information in  Topological Data Analysis. For the moment, many of these distances are hard to compute. 
	But, one can use it to distinguish some data which are close with respect to the standard interleaving distance. See~\Cref{ex:torus_sphere}, \Cref{Prop:torusvswedge}, \Cref{ex:Borromeandiscret} for instance. 
	In future work we plan to study how to define bottleneck distances taking the multiplicative structure(s) into account, in order to get more computer friendly distances. The existence of a bottleneck distance in the derived category of sheaves~\cite{Berk18} is an evidence for those.  
\end{remark*}

\setcounter{tocdepth}{1}
\tableofcontents

\subsection*{Notations}
We introduce some notations used throughout the paper:

\begin{itemize}
	\item $k$ is a field unless otherwise specified; 
	
	\item If $\sfC$, $\sfD$ are categories, $\sfD^\sfC$ stands for the category of functors $\sfC\to \sfD$\ ;
	
	\item $(\sfR,\leqslant)$ is the poset of real numbers $(\R,\leqslant)$,  viewed as a category. For all $r<s$ in $\R$, the unique corresponding morphism in the category $\sfR$ is denoted by $(r\leqslant s)$ ; $\sfR^\op$ corresponds to the poset $(\R,\geqslant)$ viewed as a category: we denote by $(r\geqslant s)$ its morphisms. Functors $X :\sfR\to  \sfC$ and $Y :\sfR^\op \to  \sfC$ are denoted by $X_\bullet$  and $Y^\bullet$ respectively;
	
	\item $\grVect$ stands for the category of graded vector space;  $\Ch_k$ the category of $\Z$-graded chain complexes and  $\coCh_k$ the category of $\Z$-graded cochain complexes. When we needed, we will denote the degree of objects in $\Ch_k$  (resp. $\coCh_k$)  by a lower index $C_*$ (resp. an upper index $C^*$).
	
	\item $\Alg_\calO$ is the  category of $\calO$-algebras in the category of cochain complex over $k$, with $\calO$ an operad (see \cite[Chapter 5]{LV2012}) (and strict morphisms);
	
	\item  when we define an interleaving distance $\d_{\dagger}$ by using a functor of cohomology $\rmH^*:\Top^\op \to \sfC $, we denote it by $\d_{\sfC}(-_1,-_2)\coloneqq \d_{\sfC}(\rmH^*(-_1),\rmH^*(-_2))$.
	
	\item we denote by $\For$, "the" forgetfull functor (which occurs in several  contexts).
\end{itemize}

\section{Multiplicative distances}\label{sect:distance_As}

\subsection{Persistent objects and interleaving}
In this section, we recall the notion of persistent objects and interleaving distances in generic categories. 

\subsubsection{Definitions}

\begin{definition}[Persistent object -- Shifting]
	Consider a category $\sfC$. A \emph{persistent object} in $\sfC$ is a functor $F_\bullet:\sfR\rightarrow  \sfC$. The image $F_{r\leqslant s}$ of the morphism $(r\leqslant s)$ by the functor $F$ is called \emph{a structural morphism of $F$}. Let $\eps$ be an object of $\sfR$, we have the natural  shifting operation $(-)[\eps]$ which sends a persistent object $F_\bullet$ to $F[\eps]_\bullet$ which is defined, for all $r$ in $\sfR$, by $F[\eps]_r\coloneqq F_{r+\eps}$. For each $\eps$ in $\sfR$, we have the natural transformation $\eta_\eps: F  \to F[\eps]$, given, for all $r$ in $\sfR$, by the structural morphism   $(r\leqslant r+\eps)$.
\end{definition}

\begin{definition}[Copersistent object -- Shifting]
	A \emph{copersistent object} in $\sfC$ is a functor $F^\bullet:\sfR^\op\to \sfC$. Let $\eps$ be an object of $\sfR$, we have the natural shifting $(-)[\eps]$ defined, for all copersistent object $F^\bullet$ and all $r$ in $\sfR$, by $F[\eps]^r=F^{r-\eps}$. We also have the natural transformation $\eta_\eps: F  \to F[\eps]$, given, for all $r$ in $\sfR$, by the structural morphism  $(r\geqslant r-\eps)$.
\end{definition}

\begin{definition}[$\eps$-morphism]
	Let $F$ and $G$ be two (co)persistent object in the category $\sfC$, and $\eps$ in $\R$. A morphism $F\to G$ is a natural transformation and a \emph{$\eps$-morphism from $F$ to $G$} is a natural transformation $F\to G[\eps]$.
\end{definition}

\begin{definition}[Interleaving pseudo-distance]\label{def:interleaving}
	Let $F$ and $G$ be two functor in $\sfC^\sfR$ and let $\eps$ be in $\sfR$. The persistent objects $F$ and $G$ are $\eps$\emph{-interleaved} if there exists two $\eps$-morphisms
	\[
		\mu : F \longrightarrow G[\eps] 
		\qquad \mbox{and} \qquad
		\nu : G \longrightarrow F[\eps]
	\]
	such that, the following diagram  commutes
	\[
		\begin{tikzcd}
			F \ar[rr,"\eta_{2\eps}"] \ar[rd,"\mu"'] & & F[2\eps] \ar[rd,"\mu{[2\eps]}"] &\\
			& G[\eps] \ar[ru,"\nu{[\eps]}" description] \ar[rr,"\eta_{2\eps}{[\eps]}"'] & & G[3\eps]
		\end{tikzcd}
	\]
	
	We define the \emph{interleaving (pseudo-)distance} by:
	\[
		\d_{\sfC}(F,G)\coloneqq \inf\left\{ \eps \geqslant 0 ~|~ F \mbox{ and } G \mbox{ are $\eps$-interleaved}\right\}.
	\]
\end{definition}

\begin{remark}
	The definitions of $\eps$-interleaving and distance of interleaving between copersistent objects are completely similar.
\end{remark}

\begin{remark}
	Given an interleaving distance $\d_\sfD$ for persistent objects in the category $\sfD$ and a functor $H:\sfC\to\sfD$, then we obtain a (pseudo-)distance defined by  \[\d_\sfD(H(-_1),H(-_2))\]  for persistent objects in $\sfC$.
\end{remark}

\begin{remark}[About our notations for the distances]\label{rem:notation_distance}
	In this paper, we introduce several distances between persistent spaces. We make the following choice of notations: when such a distance is defined by using a functor of cohomology $\rmH^*:\Top^\op \to \sfC $, we decide to denote it as follows
	\[
		\d_{\dagger}(-_1,-_2)\coloneqq \d_{\sfC}(\rmH^*(-_1),\rmH^*(-_2))
	\]
	with a suitable choice of notation to replace the symbol $\dagger$. Furthermore, whenever we use (co)chain complexes of (co)homology functors to compute the distance, the result depends on the base ground field of coefficients. But we will usually \emph{not} include it in the notation. However, when there might be some confusion or we want to put emphasis on the correct coefficient we will write 
	\[
		\d_{\dagger, k}(-_1,-_2) 
	\] 
	for the distance computed with coefficients being $k$.
\end{remark}

We will often use tacitly the following easy lemma.

\begin{lemma}\label{lem::apply_functor}
	Let $F$ and $G$ be two (co)persistent objects in $\sfC$ and let $H:\sfC \to \sfD$ be a functor. We have
	\[
	\d_{\sfD}(H F,H G)\leqslant \d_{\sfC}(F,G).
	\]
\end{lemma}

\begin{proof}
	Let $\eps \geqslant 0$ such that $\d_{\sfC}(F,G) \leqslant \eps$, then there exists an $\eps$-interleaving between $F$ and $G$. As $H$ is a functor, it preserves the diagram of $\eps$-interleaving, so $\d_{\sfD}(H F,H G)\leqslant \eps$.
\end{proof} 

This lemma implies readily stability for persistent spaces constructed out of a Morse-type function stability thanks to the following initial stability result:

\begin{theorem}[Stability]\label{thm:stability1}
	Let $X$ be a topological space and let $f,g:X \to \R$ be two continuous maps. Denote 
	\[
	\begin{array}{rcl}
	X^f_\bullet : \sfR & \longrightarrow & \Top \\
	r & \longmapsto & f^{-1}\left( (-\infty,r] )\right)
	\end{array}
	\quad \mbox{ and } \quad 
	\begin{array}{rcl}
	X^g_\bullet : \sfR & \longrightarrow & \Top \\
	r & \longmapsto & g^{-1}\left( (-\infty,r] )\right)
	\end{array}
	\]
	the two filtered spaces associated to $f$ and $g$. We have
	\[
		\d_{\Top}(X^f_\bullet,X^g_\bullet) \leqslant \|f-g\|_\infty .
	\]
\end{theorem}

\begin{proof}
	 Let $\eps\geqslant 0$ such that $\|f-g\|_\infty\leqslant \eps$ and fix $r$ in $\sfR$. Let $x$ be a point in $X^f_r$: there exists $s \leqslant r$ such that $f(x)=s$. By assumption, we have $|g(x) - f(x)| \leqslant \eps$, then $g(x)\leqslant s+\eps$ so $x$ is in $X^g_{r+\eps}$. So we have $X^f_r \subseteq X^g_{r+\eps}$. By the argument, we have, for all $r$ in $\sfR$, the following inclusions
	\[
	X^f_r \subseteq X^g_{r+\eps} \subseteq X^f_{r+2\eps} 
	\mbox{ and }
	X^g_r \subseteq X^f_{r+\eps} \subseteq X^g_{r+2\eps} \ ,
	\]
	which define an $\eps$-interleaving between $X^f_\bullet$ and $X^g_\bullet$.
\end{proof}

\subsubsection{Two canonical examples of persistent spaces}\label{sec:CechRips}
Many interesting persistent objects arise from simplicial complexes/sets constructions. 
We denote by $\sSet$, the category of simplicial sets. A simplicial set has a canonical associated topological space, called its geometric realization. We denote
	\[
		|-|:\sSet \longrightarrow \Top 
	\]
the associated functor. It is the left adjoint of the \emph{singular simplicial complex functor} $\Sing:\Top \to \sSet$  which is defined, for $X$ a topological space, by $\Sing_n(X)\coloneqq \Hom_{\Top}(\Delta_{\Top}^n,X)$, where $\Delta_{\Top}^n$ is the $n$-th topological standard simplex. An important property is that they have the same homotopy theories (and the functors actually form a Quillen equivalence). In the rest of the paper, we often identify simplicial sets and topological spaces when defining (co)chains and (co)homology type functors/constructions.

\begin{definition}[Vietoris-Rips complex]\label{def:Rips}
	Let $X$ be a metric space. For $a$ in $\sfR$, we define a simplicial set $\calR_a(X,\d_X)$ on the vertex set $X$ by the following condition:
	\[
		\sigma \in \calR_a(X,\d_X) \Longleftrightarrow \d_X(x,y) \leqslant a \mbox{ for all } x,y \mbox{ in } \sigma.
	\]
    The collection of these simplices is the \emph{Vietoris-Rips filtered complex} of $X$ denoted
	\[
		\Rips(X,\d_X):\sfR\longrightarrow \sSet.
	\]
\end{definition}

\begin{definition}[Intrinsic \v{C}ech complex]\label{def:Cech}
	Let $X$ be a metric space. For $a$ in $\sfR$, we define a simplicial set $\Cech_a(X,\d_X)$ on the vertex set $X$ by the following condition:
	\[
		[x_0,x_1,\ldots,x_k] \in \Cech_a(X,\d_X) \Longleftrightarrow \bigcap_{i=0}^k \rmB(x_i,a)\ne \varnothing.
	\]
	We denote the \emph{intrinsic \v{C}ech filtered complex} of $X$ by
	\[
		\Cech(X,\d_X):\sfR\longrightarrow \sSet.
	\]
	For $a$, a real number, and $\sigma=[x_0,\ldots, x_n]$, a simplex of $\Cech_a(X,\d_X)$, an element $\bar{x}$ of $\cap_i \rmB(x_i,a)$ is called a \emph{$a$-center} of $\sigma$.
\end{definition}

\begin{figure}[h!]
	\begin{minipage}[c]{.46\linewidth}
		\centering
	\begin{tikzpicture}[scale=0.6]
	\pgfmathsetmacro{\RAD}{1pt};
	\coordinate (A) at (0.04,-0.35);
	\coordinate (B) at (1.7,0.4);
	\coordinate (C) at (0.1,1.5);
	\coordinate (D) at (2.1,3.7);
	\coordinate (E) at (2.9,2.4) ;
	\draw[fill=gray!20, opacity=0.2, thick] (A) circle (\RAD);
	\draw[fill=gray!20, opacity=0.2, thick] (B) circle (\RAD);
	\draw[fill=gray!20, opacity=0.2, thick] (C) circle (\RAD);
	\draw[fill=gray!20, opacity=0.2, thick] (D) circle (\RAD);
	\draw[fill=gray!20, opacity=0.2, thick] (E) circle (\RAD);
	\draw[thick] 
	(D) node {$\scriptstyle{\bullet}$} node[above] {$\scriptstyle{x_4}$} to 
	(E) node {$\scriptstyle{\bullet}$} node[below] {$\scriptstyle{x_5}$} ;
	\draw[thick, fill=gray!50, opacity=0.7] 
	(A) node {$\scriptstyle{\bullet}$} node[below] {$\scriptstyle{x_1}$} to 
	(B) node {$\scriptstyle{\bullet}$} node[below] {$\scriptstyle{x_2}$} to 
	(C) node {$\scriptstyle{\bullet}$} node[above] {$\scriptstyle{x_3}$} to (A);
	\end{tikzpicture}
	\caption{The Vietoris Rips complex of a set $X$}
	\label{fig:vietoris}
\end{minipage}
\begin{minipage}[c]{.46\linewidth}
	\centering
	\begin{tikzpicture}[scale=0.6]
	\pgfmathsetmacro{\RAD}{1pt};
	\coordinate (A) at (0.04,-0.35);
	\coordinate (B) at (1.7,0.4);
	\coordinate (C) at (0.1,1.5);
	\coordinate (D) at (2.1,3.7);
	\coordinate (E) at (2.9,2.4) ;
	\draw[fill=gray!20, opacity=0.2, thick] (A) circle (\RAD);
	\draw[fill=gray!20, opacity=0.2, thick] (B) circle (\RAD);
	\draw[fill=gray!20, opacity=0.2, thick] (C) circle (\RAD);
	\draw[fill=gray!20, opacity=0.2, thick] (D) circle (\RAD);
	\draw[fill=gray!20, opacity=0.2, thick] (E) circle (\RAD);
	\draw[thick] 
	(D) node {$\scriptstyle{\bullet}$} node[above] {$\scriptstyle{x_4}$} to 
	(E) node {$\scriptstyle{\bullet}$} node[below] {$\scriptstyle{x_5}$} ;
	\draw[thick] 
	(A) node {$\scriptstyle{\bullet}$} node[below] {$\scriptstyle{x_1}$} to 
	(B) node {$\scriptstyle{\bullet}$} node[below] {$\scriptstyle{x_2}$} to 
	(C) node {$\scriptstyle{\bullet}$} node[above] {$\scriptstyle{x_3}$} to (A);
	\end{tikzpicture}
	\caption{The \v{C}ech complex of the same set $X$}
	\label{fig:cech} 
\end{minipage}
\end{figure}

\begin{notation}\label{not:alphacomplex}
	For a finite set $X$ of points in $\R^d$, there exists another filtered complexe associated to $X$, called \emph{alpha complex}, which we denote by 
	\[
		\Alpha(X) \colon \sfR \longrightarrow \sSet \ . 
	\]
	We refer to \cite[Section III.4]{edelsbrunner2010computational} for the definition of this filtered complex, which is usefull for concrete computations with the python librairy  \texttt{Gudhi} (see \cite{gudhiweb} for the documentation). The alpha complex has some nice properties. We can compare the \v{C}ech complex and the alpha complex of $X$ as follow : for all $\eps> 0$, we have $\Alpha(X)_\eps \subseteq \Cech(X)_\eps$. Moreover, the Nerve Theorem implies that
	\[
		|\Alpha(X)_\eps|	 \simeq \bigcup_{x\in X} \rmB(x,\eps) \ .
	\]
\end{notation}

\subsection{Distances for persistent algebras}\label{subsect:distance_As} 
We start by reviewing of some classical objects in algebraic topology. 

\begin{defiprop}[Singular (co)chain functor]\label{def:cupproduct}
	The \emph{singular chain functor}, denoted by $\rmC_*^{\mathrm{Sing}}:\Top \to \Ch_k$,  is defined by $\rmC_*^{\mathrm{Sing}}(X)\coloneqq k[\Sing(X)]$, where the differential is given by the signed sum of the maps induced by the face maps. The \emph{singular cochain functor} $\rmC^*_{\As}:\Top^\op\rightarrow \AsAlg$ is defined, for all topological space $X$, by the cochain complex  $\rmC^*_{\As}(X)\coloneqq \Hom_{k}(\rmC_*^{\mathrm{Sing}}(X),k)$. Equipped with the \emph{cup product} defined as the composite
	\[
		\begin{tikzcd}
		-\cup-:\rmC^*_{\As}(X)\otimes \rmC^*_{\As}(X)
		\ar[r]
		& \rmC^*_{\As}(X\times X) \ar[r,"\Delta^*"]
		& \rmC^*_{\As}(X)
		\end{tikzcd}
	\]
	where the first map is the K\"unneth map and the second is the map induced by the diagonal, the singular cochain complex is a \emph{differential graded associative} algebra (dg-algebra for short).	
\end{defiprop}\

Therefore we have in particular a notion of persistent cochain complex and a notion of persistent \emph{dg-algebra} which are naturally induced by persistent spaces $X\colon \sfR \rightarrow \Top$. The natural notion of equivalences for these are given by persistent quasi-isomorphisms, that is natural transformations 
$A^* \to B^* $ which induced isomorphisms in cohomology; here $A^*, B^*$ are either persistent  dg-algebras $\sfR \to 
\AsAlg$ or persistent (co)chain complexes $\sfR \to 
\Ch$. Both chain complexes and dg-algebras have natural notions of homotopies and we can pass to their homotopy categories. The cohomology functors from dg-algebras (resp. cochain complexes) factors through the respective homotopy categories. 

\begin{definition}\label{nota:hocat} Let
\begin{itemize}
	\item $\ho(\AsAlg) \coloneqq \AsAlg [\mathrm{qiso}^{-1}]$ be the homotopy category of dg-associative algebras where the quasi-isomorphisms are formally inverted;
	
	\item $\ho(\Ch) \coloneqq \Ch  [\mathrm{qiso}^{-1}]$ be the homotopy category of cochain complexes where the quasi-isomorphisms are formally inverted.
\end{itemize}
 We denote $\rmC_{\ho(\As)}: \Top^{\op}\to \ho(\AsAlg)$, the composition of 
$\rmC_{\As}$ with the canonical quotient functor $\AsAlg \to \ho(\AsAlg)$.
\end{definition}

\begin{notation}\label{nota:cochainfunctor}
	We consider the following functors associated to spaces (and simplicial sets):
	 \begin{enumerate}
	 	\item $\rmC_{\ho(\Ch)}^*:\Top^\op\rightarrow \ho(\Ch_k)$ with $\rmC_{\ho(\Ch)}^* = \For\circ \rmC_{\ho(\As)}^*$;
	 	\item $\rmC_{\Ch}^*:\Top^\op\rightarrow \Ch_k$ with $\rmC_{\Ch}^* = \For\circ \rmC_{\As}^*$;
	 	\item  $\rmH^*_{\Com}:\Top^\op \rightarrow \ComAlg$; 
	 	\item $\rmH^*_{\As}:\Top^\op \rightarrow \AsAlg$ with $ \rmH^*_{\As} = \For\circ \rmH^*_{\Com}$ or $\rmH^*_{\As} = \H \circ \rmC_{\As}^*$;
 		\item $\rmH^*_{\grVect}:\Top^\op \rightarrow \grVect$ with $ \rmH^*_{\grVect} = \For\circ \rmH^*_{\As}$\ .
 	 \end{enumerate}
\end{notation}

We can summarize all these functors in the following commutative diagram:
\[
 	\begin{tikzcd}
 		\Top^\op \ar[dd,"\rmH_{\Com}"'] \ar[rd,"\rmC_{\ho(\As)}"] \ar[rrdd, bend right=15, dotted, "\rmH"']  \ar[rrd,bend left, "\rmC_{\ho(\Ch)}"]  \ar[rdd,bend right, "\rmH_{\As}"] & & \\
		&	\ho(\AsAlg) \ar[r,"\For"] \ar[d,"\rmH"] & \ho(\Ch_k) \ar[d,"\rmH"] \\
		\ComAlg \ar[r,"\For"'] &	\AsAlg \ar[r,"\For"'] & \grVect 
 	\end{tikzcd} \ .
\]
Associated to each of these functors,  we can define their interleaving distances 
according to~\Cref{def:interleaving}.

\begin{notation}
	Let $X:\sfR \rightarrow \Top$ and $Y:\sfR \rightarrow \Top$ be two persistent topological spaces. According to \Cref{rem:notation_distance}, some of these distance are denoted as follows:
	\[
		\d_{\grVect}(X,Y)  \coloneqq \d_{\grVect}(\rmH^*(X),\rmH^*(Y)) \ , \ \ 
		\d_{\As}(X,Y)  \coloneqq \d_{\AsAlg}(\rmH^*(X),\rmH^*(Y)) \ .
	\]
	Remark that the distance $\d_{\grVect}$ is the \emph{classical} interleaving distance (see \cite{Oudot}) (or more precisely the supremum of the interleaving distance associated to every degrees).
\end{notation}

\begin{remark}
Considering the homotopy category of persistent cochain complexes, that is objects of the type $\rmC^*_{\ho(\Ch)}(X)$, is an analogue of considering the derived category of sheaves over $\bbR$ which is a recent and promising nice approach to persistent homology, see~\cite{Kash18}. Indeed, the associated interleaving distance $\d_{\ho(\Ch_k)}$ has been considered in~\cite{BerPetit} and seen to agree with the convolution distance of~\cite{Kash18} for sheaves.
\end{remark}

\begin{remark}
We can of course also consider the interleaving distance on the  persistent cochain  complexes $\rmC^*$ of  persistent spaces $X, Y:\sfR \rightarrow \Top$ both taken in  \emph{dg-algebras}. This is however not a very pertinent notion to look at. Indeed, if $T$ is a triangulation of simplicial complex $X$, there is no natural  dg-algebra map $\rmC^*(X) \to \rmC^*(T)$  though they are homotopy equivalent simplicial complexes and, in fact,  have the same underlying space. 

In particular, given a persistent space and a persistent triangulation of it, $\rmC^*(X)$ and $\rmC^*(T)$ are not $\varepsilon$-interleaved in $\AsAlg$ for small $\varepsilon$ and therefore $\d_{\AsAlg}(\rmC^*(X), \rmC^*(T)) \gg 0$ in general event though they represent the same topological space. In particular this distance will not satisfy stability (as in  \Cref{thm:stabilite}). This problem of course disappear precisely in the homotopy category of algebras.

Note that, unlike for dg-algebras, over a field, one can always find a highly non-natural inverse for a cochain complex map  $\rmC^*(X) \to \rmC^*(T)$ for fixed spaces so that the situation looks better. But however, in general, one can not find such an inverse in a persistent way. The problem is similar to \Cref{R:notransfer}.
\end{remark}

These distances satisfy the following inequalities.

\begin{proposition}\label{prop:inequalities}
	Let $X:\sfR \rightarrow \Top$ and $Y:\sfR \rightarrow \Top$ be two persistent topological spaces. We have the following inequalities:
	\begin{align*} 
		\d_{\grVect}(X,Y) &
			 \leqslant  \d_{\As}(\rmH^*(X),\rmH^*(Y))
			 = \d_{\ComAlg}(\rmH^*(X),\rmH^*(Y)) \ , \\
		\d_{\grVect}(X,Y) &
			\leqslant  \d_{\ho(\Ch)}(\rmC^*(X),\rmC^*(Y)) 
			\leqslant  \d_{\ho(\AsAlg)}(\rmC^*(X),\rmC^*(Y)) \ , \\
		\d_{\grVect}(X,Y) &
			\leqslant \d_{\As}(X,Y)
			\leqslant  \d_{\ho(\AsAlg)}(\rmC^*(X),\rmC^*(Y)) \ .
	\end{align*} 
\end{proposition}

\begin{proof}
    It follows from \Cref{lem::apply_functor}.
\end{proof}
\begin{notation}
 The interleaving distance $\d_{\As}$ in the category of associative graded algebra will be called the graded algebra interleaving distance or sometimes the cup-product interleaving distance because of how it appears in our examples. 
\end{notation}

\begin{example}\label{ex:torus_sphere} In this example we study examples (already with low dimensional spaces) for which the standard bottleneck distance can be made as small as one wishes while their copersistent graded algebras interleaving distance remains as close as we wished to a fixed number.  

More precisely we exhibit two data sets $D_X$ and $D_Y$ in $\R^3$ such that the interleaving distance between the cohomology of their \v{C}ech complex as graded vector space is strictly smaller than their interleaving distance as copersistent graded algebras. Consider the following two topological spaces $X$ and $Y$ living in  $\R^3$:
\begin{align*}
	X & = \{(x,y,0) \ | \ x^2+(y-2)^2=1\} \cup \{x^2+y^2+z^2=1\} \cup \{(x,y,0) \ |\ x^2+(y+2)^2=1\} \ ; \\
	Y & = \{(x,y,z) \ | \ (x^2+y^2+z^2+3)^2 =16(x^2+y^2)\} \ ,
\end{align*}
which are represented in \Cref{fig:sphere_and_cirles} and \Cref{fig:torus} respectively.
\begin{figure}[h!]
	\begin{minipage}[c]{.46\linewidth}
	\centering
	\includegraphics[scale=1]{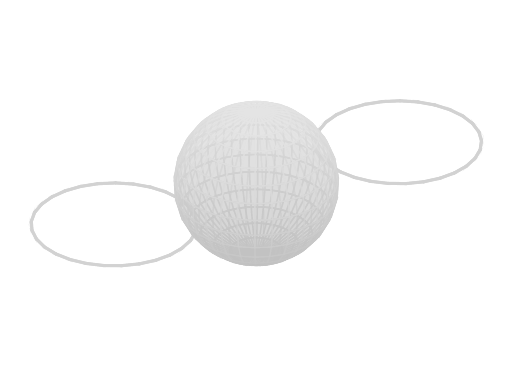}
	\caption{The space $X$}
	\label{fig:sphere_and_cirles}
	\end{minipage} 
	\hfill 
	\begin{minipage}[c]{.46\linewidth}
	\centering 
	\includegraphics[scale=0.7]{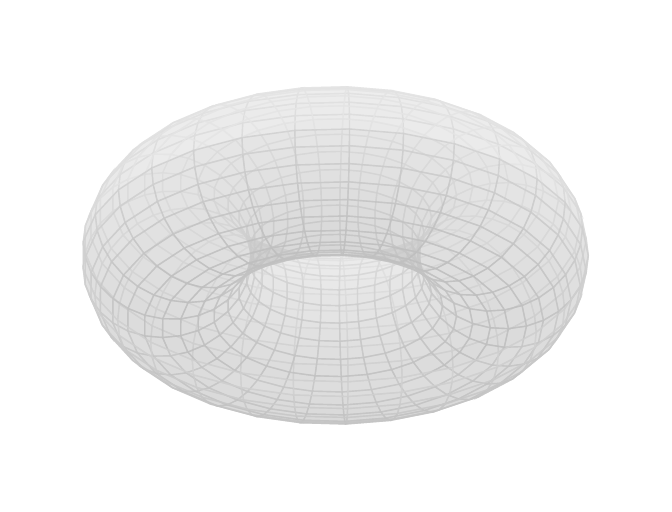}
	\caption{The space $Y$}
	\label{fig:torus}
	\end{minipage}
\end{figure} 

One can easily describe the cohomology algebra (with coefficient in the field $k$) of these two spaces (see \cite[Section 3.2]{Hatcher} for the description of computation of the cup product of the sphere and the torus). For the space $X$, we have 
\[
	\rmH^0(X) = k\cdot e, \quad \rmH^1(X)= k\cdot \alpha \oplus k\cdot \beta \quad \mbox{ and } \quad \rmH^2(X) = k\cdot \gamma
\]
where $\alpha$ and $\beta$ are the two classes induced by the two circles and $\gamma$ is the class induced by the $2$-sphere. It is clear that, by definition of the cup product, we have $\alpha \cup \beta = 0 = \beta \cup \alpha$, $e$ is the unit of the product and the other product are trivial by a degree argument. For the space $Y$, we also have 
\[
	\rmH^0(Y) = k\cdot e, \quad \rmH^1(Y)= k\cdot \alpha \oplus k\cdot \beta \quad \mbox{ and } \quad \rmH^2(Y) = k\cdot \gamma
\]
but the classes $\alpha$ and $\beta$ respectively representing meridian and parallel circles of the torus, and we have $\alpha \cup \beta = \gamma$
and $e$ is the unit of the cup product. Note that the graded vector spaces $\rmH^*(X)$ and $ \rmH^*(Y)$ are trivially isomorphic but graded commutative algebras $(\rmH^*(X),\cup)$ and $(\rmH^*(Y),\cup)$ are not because there is no injective graded algebra morphism from $\rmH^*(Y)$ to $\rmH^*(X)$.

For $\eps \geqslant 0$, we denote
\[
	X_\eps \coloneqq \bigcup_{x \in  X} B(x,\eps)
	\qquad \mbox{and} \qquad 
	Y_\eps \coloneqq \bigcup_{y \in  Y} B(y,\eps) \ .	
\] 
Now remark that, for $\eps < 1$, we have homotopy equivalences $X_\eps \simeq X$ and $Y_\eps \simeq Y$ and, for $\eps \geqslant 1$, $X_\eps$ and $Y_\eps$ are contractile. We fix $0< \alpha \ll   \frac{1}{2}$ and we choose a finite cover by open balls of radius $\alpha$ of $X_\alpha$ and $Y_\alpha$:
\[
	X_\alpha \subset \bigcup_{i \in  I} B(x_i,\alpha)
	\qquad \mbox{and} \qquad 
	Y_\alpha \subset \bigcup_{j \in  J} B(y_j,\alpha) 	
\]
where $I$ and $J$ are finite. We can then define the discrete spaces
\[
	D_X^\alpha = \bigcup_{i \in  I} \{x_i\}
	\qquad \mbox{and} \qquad 
	D_Y^\alpha = \bigcup_{j \in  J} \{y_j\}.
\]
We can think these spaces as noisy discretisations of our spaces $X$ and $Y$. Remark that, for $\alpha \leqslant \eps' < 1-\alpha$, we have 
\[
	\Cech(D_X^\alpha)_{\eps'} \simeq 	X_\alpha \simeq X
	\qquad \mbox{and} \qquad 
	\Cech(D_Y^\alpha)_{\eps'} \simeq 	Y_\alpha \simeq Y.
\]

\begin{figure}[h!]
	\begin{tikzpicture}
	\draw (0.7,0.4) -- (7.3,0.4);
	\draw (1,0.4) node {|} node[below=1.5] {$\scriptscriptstyle{0}$};
	\draw (2,0.4) node {|} node[below=1.5] {$\scriptscriptstyle{\alpha}$};
	\draw[fill,color=gray!10] (1,0.8) rectangle (2,2.3);
	\draw (1.5,1.5) node {\tiny{"noise"}};
	\draw (5,0.4) node {|} node[below=1.5] {$\scriptscriptstyle{1-\alpha}$};
	\draw (6,0.4) node {|} node[below=1.5] {$\scriptscriptstyle{1+\alpha}$};
	\draw[blue] (2,1.5) -- (5,1.5) ;
	\draw[fill,color=gray!10] (5,1.3) rectangle (6,2.3);
	\draw (5.5,1.8) node {\tiny{"noise"}};
	\draw[blue] (2,1.6) -- (5,1.6) node[above left] {$\scriptscriptstyle{\rmH^1}$};
	\draw[red] (2,2.1) -- (5,2.1) node[above left] {$\scriptscriptstyle{\rmH^2}$};
	\draw (2,1) -- (5,1) node[above left] {$\scriptscriptstyle{\rmH^0}$} (5,1) -- (7,1);
	\end{tikzpicture}
	\caption{A part of the barcode of $\rmH^*(\Cech(D_X^\alpha)_\bullet)$ and $\rmH^*(\Cech(D_Y^\alpha)_\bullet)$}
	\label{fig::cpxcercle}
\end{figure}

\begin{proposition}\label{Prop:torusvswedge}For any $0<\alpha <\frac{1}{2}$, one has
 \[
	\d_{\grVect}\left(\Cech(D_X^\alpha),\Cech(D_Y^\alpha)\right) \leqslant 2\alpha \ .
\]
while \[
	\frac{1-2\alpha}{2} \leqslant \d_{\As}\left( \Cech(D_X^\alpha),\Cech(D_Y^\alpha)\right)\leqslant \frac{1}{2} \ .
\]
\end{proposition}
In particular, as soon as $\alpha< \frac{1}{4}$, we have
\[
	\d_{\grVect}\left(\Cech(D_X^\alpha),\Cech(D_Y^\alpha)\right)
	<\d_{\As}\left( \Cech(D_X^\alpha),\Cech(D_Y^\alpha)\right).
\]
and the smaller is $\alpha$, the bigger is the difference between the two distances.
Indeed one can make the usual bottleneck distance goes to $0$ while the associative distance goes to $\frac{1}{2}$. 

\begin{proof}[Proof of Proposition~\ref{Prop:torusvswedge}]
 For the first inequality $
	\d_{\grVect}\left(\Cech(D_X^\alpha),\Cech(D_Y^\alpha)\right) \leqslant 2\alpha 
$, note that for  all $\eps \in [\alpha, 1-\alpha[ \cup ]1+\alpha,+\infty [$,   $\rmH^*(\Cech(D_X^\alpha)_\eps) $ and $ \rmH^*(\Cech(D_Y^\alpha)_\eps)$ are isomorphic as graded vector spaces. Therefore, for all $\eta>0$, one can exhib a linear $(2\alpha+\eta)$-interleaving between $\rmH^*(\Cech(D_X^\alpha)) $ and $ \rmH^*(\Cech(D_Y^\alpha))$ and the result on the distance $\d_{\grVect}$ follows.

Now, suppose that there is $\eps < \frac{1-2\alpha}{2}$ such that there exists a $\eps$-interleaving in the category $\AsAlg$ between $\rmH^*(\Cech(D_\alpha^X))$ and $\rmH^*(\Cech(D_\alpha^X))$. Then, by definition, we will have the 
the following diagram
\[
	\begin{tikzcd}
		& \rmH^*(\Cech(D^\alpha_X))_{\alpha+\eps} \arrow[rd,"\nu"]& \\
		\rmH^*(\Cech(D^\alpha_Y))_\alpha  \arrow[rr,"\cong"'] \arrow[ru,"\mu"] && \rmH^*(\Cech(D^\alpha_Y))_{\alpha+2\eps}
	\end{tikzcd} \ ,
\]
that implies that $\mu$ is injective. But there is no injective graded algebra homomorphism $\rmH^*(\Cech(D^\alpha_Y))_\alpha \cong \rmH^*(Y) \to \rmH^*(X) \cong 
\rmH^*(\Cech(D^\alpha_X))_{\alpha+\eps}$ as previously mentioned (since $\alpha\cup \beta =\gamma$ in $\rmH^*(Y)$ while the cup product is zero in $\rmH^*(X)$). Therefore there is no such $\eps$-interlevaing and we have the inequality
\[
	\frac{1-2\alpha}{2} \leqslant \d_{\As}\left( \Cech(D_X^\alpha),\Cech(D_Y^\alpha)\right)
\] 
as a consequence. 
On the other hand, since the structure maps $\rmH^*(\Cech(D^\alpha_Y))_\eps \to (\Cech(D^\alpha_Y))_{1+\eps}$ and $\rmH^*(\Cech(D^\alpha_X))_\eps \to (\Cech(D^\alpha_X))_{1+\eps}$ are zero, it is obvious that there is an  $\frac{1}{2}$-multiplicative interleaving. We thus get  the missing inequality 
\[
	\d_{\As}\left( \Cech(D_X^\alpha),\Cech(D_Y^\alpha)\right)\leqslant \frac{1}{2} \ .
\]
\end{proof}
\end{example}

\begin{remark} 
	The previous example can be done for the Rips complex of the same data sets  $D_X^\alpha$, $D_Y^\alpha$ by considering similar balls but for the euclidean norm replaced by the $\|-\|_\infty$ norm.
\end{remark}

\begin{remark}
	It is easy to construct several examples along the line of~\Cref{ex:torus_sphere}. We follow this idea to make \Cref{ex:borromean}.
\end{remark}

\begin{example}[Concrete computations for \Cref{ex:torus_sphere}] \label{ex:torus_sphere_comput}

One can consider a concrete example of this situation,  considering a specific discretisation of these spaces. For the rest of this example, we fix an integer $n$, which suppose even for convenience of computation. We consider the following discrete space

\[
	\begin{aligned}
	\widetilde{X} \coloneqq &
	\left\{
		\left(
			\cos(-\frac{\pi}{2}+i\frac{\pi}{n})\cos(j\frac{\pi}{n}),
			\cos(-\frac{\pi}{2}+i\frac{\pi}{n})\sin(j\frac{\pi}{n}),
			\sin(-\frac{\pi}{2}+i\frac{\pi}{n}
		\right) ~|~
		i\in \llbracket0,n\rrbracket, \quad j\in \llbracket-n,n\rrbracket
	\right\} \\
	& \cup 
	\left\{
		(\cos(j\frac{\pi}{n}),
		\sin(j\frac{\pi}{n})-2,
		0) ~|~
		 j\in \llbracket-n,n\rrbracket
	\right\}
	\cup 
	\left\{
		(\cos(j\frac{\pi}{n}),
		\sin(j\frac{\pi}{n})+2,
		0) ~|~
		 j\in \llbracket-n,n\rrbracket
	\right\}
	\end{aligned}
\]
which is clearly a discretisation of the space $X$. One can describe the persistent homology of this discrete space. For $0\leqslant \eps < 2\sin(\frac{\pi}{n})$, we have 
\[
	\rmH^2(\Rips(\widetilde{X}))_\eps = 0 \ ,
\]
with several birth and death of classes in degree $1$ ; 
and, for all $\eps\in [2\sin(\frac{\pi}{n}), 1[$, 
\[
	\rmH^*(\Rips(\widetilde{X}))_\eps = \rmH^*(X) \ ,
\]
and, for $\eps \geqslant 1$, on a 
\[
	\rmH^*(\Rips(\widetilde{X}))_\eps = k
\]
concentrated in degree $0$. One also consider the discrete space 
\[
	\widetilde{Y} \coloneqq 
	\left\{
		\left(
		(2+\cos(i\frac{\pi}{n}))\cos(j\frac{\pi}{3n}),
		(2+\cos(i\frac{\pi}{n}))\sin(j\frac{\pi}{3n}),
		\sin(i\frac{\pi}{n})
		\right)
		~|~
		i\in \llbracket-n,n\rrbracket, \quad j\in \llbracket-3n,3n\rrbracket
	\right\}
\]
which is a discretisation of the space $Y$. For $0\leqslant \eps < B_n$, we have 
\[
	\rmH^2(\Rips(\widetilde{Y}))_\eps = 0 \ ,
\]
with several birth and death of classes in degree $1$, where 
\[
	B_n  \coloneqq \sqrt{(3\cos(\frac{\pi}{3n})-2 - \cos(\frac{\pi}{n}))^2  + 9\sin^2(\frac{\pi}{3n}) +\sin^2(\frac{\pi}{n})} \ ;	
\] 
and, for all $\eps\in [B_n, 1[$, 
\[
	\rmH^*(\Rips(\widetilde{Y}))_\eps = \rmH^*(Y) \ ,
\]
and, for $\eps \geqslant 1$, on a 
\[
	\rmH^*(\Rips(\widetilde{Y}))_\eps = k
\]
concentrated' in degree $0$. Remark that, for all $n\in \N^*$, $B_n < 2\sin(\frac{\pi}{n})$, then, for all $\eps > 2\sin(\frac{\pi}{n})$, we have 
\[
	\rmH^*(\Rips(\widetilde{X}))_\eps \underset{\grVect}{\cong} \rmH^*(\Rips(\widetilde{Y}))_\eps
\]
so 
\[
	\d_{\grVect}\left(\Rips(\widetilde{X})),\Rips(\widetilde{Y}))\right) \leqslant \sin(\frac{\pi}{n}) \ ,
\]
and
\[
	\frac{1-2\sin(\frac{\pi}{n})}{2} \leqslant \d_{\As}\left(\Rips(\widetilde{X})),\Rips(\widetilde{Y}))\right) \leqslant  \frac{1}{2} \ .
\]
One can use the python librairy \texttt{Gudhi} to make some computations of the bottleneck distance between $\widetilde{X}$ and $\widetilde{Y}$ using alpha complexes (see \Cref{not:alphacomplex}). For example, if $n=20$, then the persistent diagrams and the barcodes (where one just print the fifteenth bigger bars) are given in \Cref{fig:barcode_alpha_X}, \ref{fig:barcode_alpha_Y}, \ref{fig:pers_alpha_X} and \ref{fig:pers_alpha_Y}. One also compute (using again \texttt{Gudhi}) the bottleneck distance between persistent diagrams in each degree, which gives us :
\[
	\begin{aligned}
	\mathrm{d}_{\mathrm{bottle}}(\rmH^0(\Alpha(\widetilde{X})), \rmH^0(\Alpha(\widetilde{Y}))) & \approx 0.0030 \\
	\mathrm{d}_{\mathrm{bottle}}(\rmH^1(\Alpha(\widetilde{X})), \rmH^1(\Alpha(\widetilde{Y}))) & \approx 0.0055 \\
	\mathrm{d}_{\mathrm{bottle}}(\rmH^2(\Alpha(\widetilde{X})), \rmH^2(\Alpha(\widetilde{Y}))) & \approx 0.0027
	\end{aligned} 
\] 
Therefore the graded bottleneck distance is $\max_{j}(\mathrm{d}_{\mathrm{bottle}}(\rmH^j(\Alpha(\widetilde{X})), \rmH^j(\Alpha(\widetilde{Y})))\simeq 0,0055$.
\begin{figure}[h!]
	\begin{minipage}[c]{.46\linewidth}
	\includegraphics[scale=0.45]{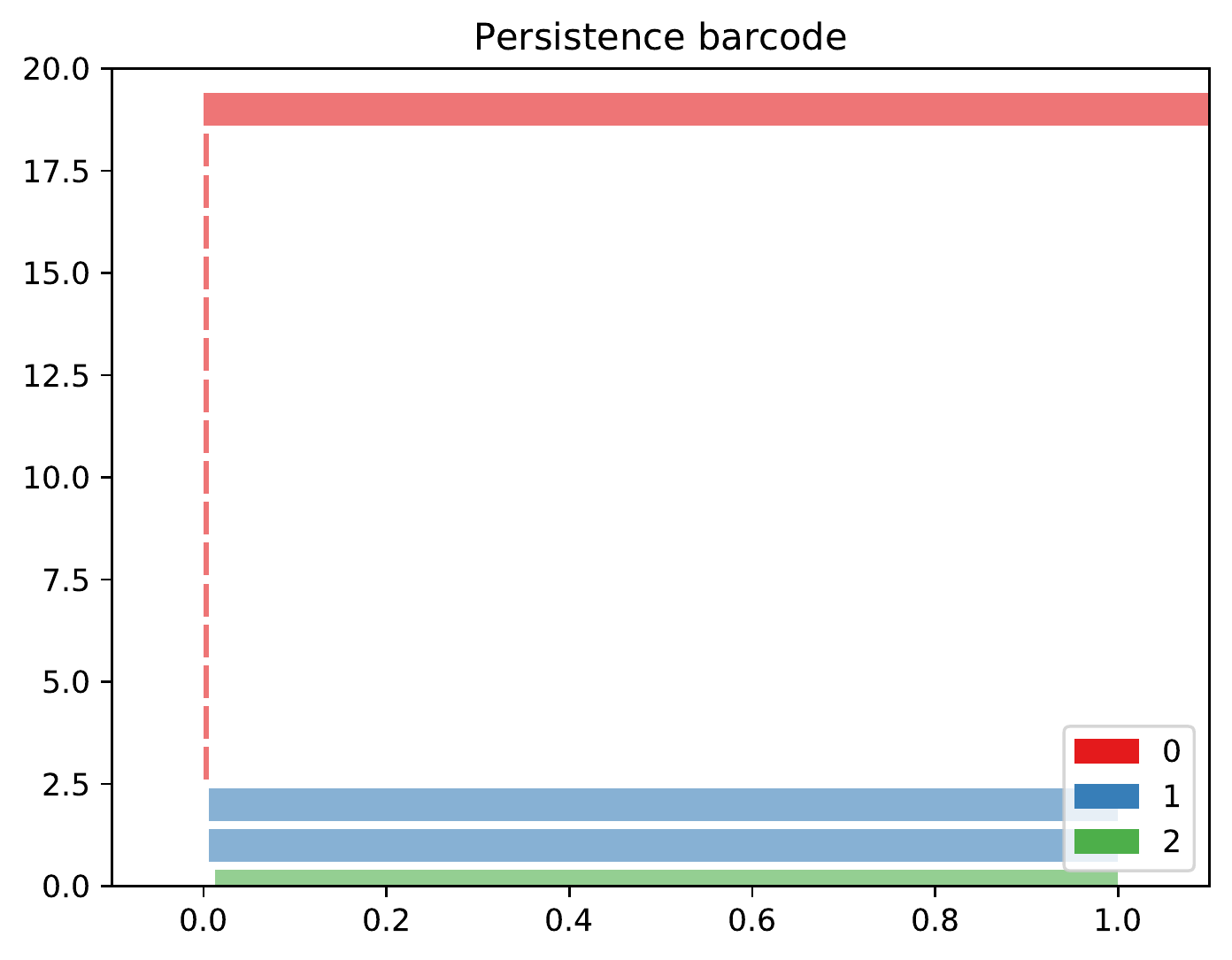}
	\caption{Persistent barcode of $\Alpha(\widetilde{X})$}
	\label{fig:barcode_alpha_X}
\end{minipage}
\begin{minipage}[c]{.46\linewidth}
	\includegraphics[scale=0.45]{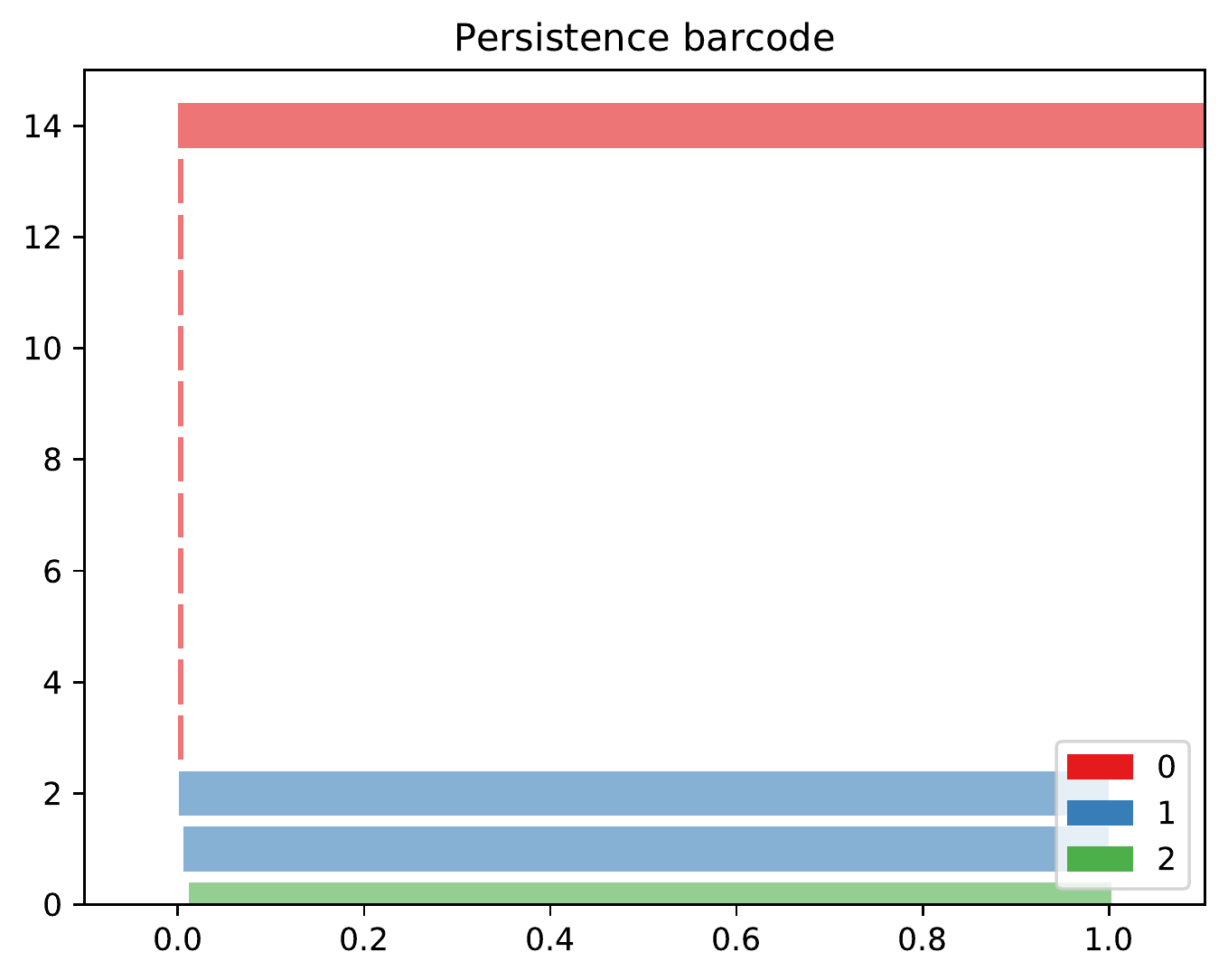}
	\caption{Persistent barcode of $\Alpha(\widetilde{Y})$}
	\label{fig:barcode_alpha_Y} 
\end{minipage}
\end{figure}

\begin{figure}[h!]
	\begin{minipage}[c]{.46\linewidth}
	\includegraphics[scale=0.45]{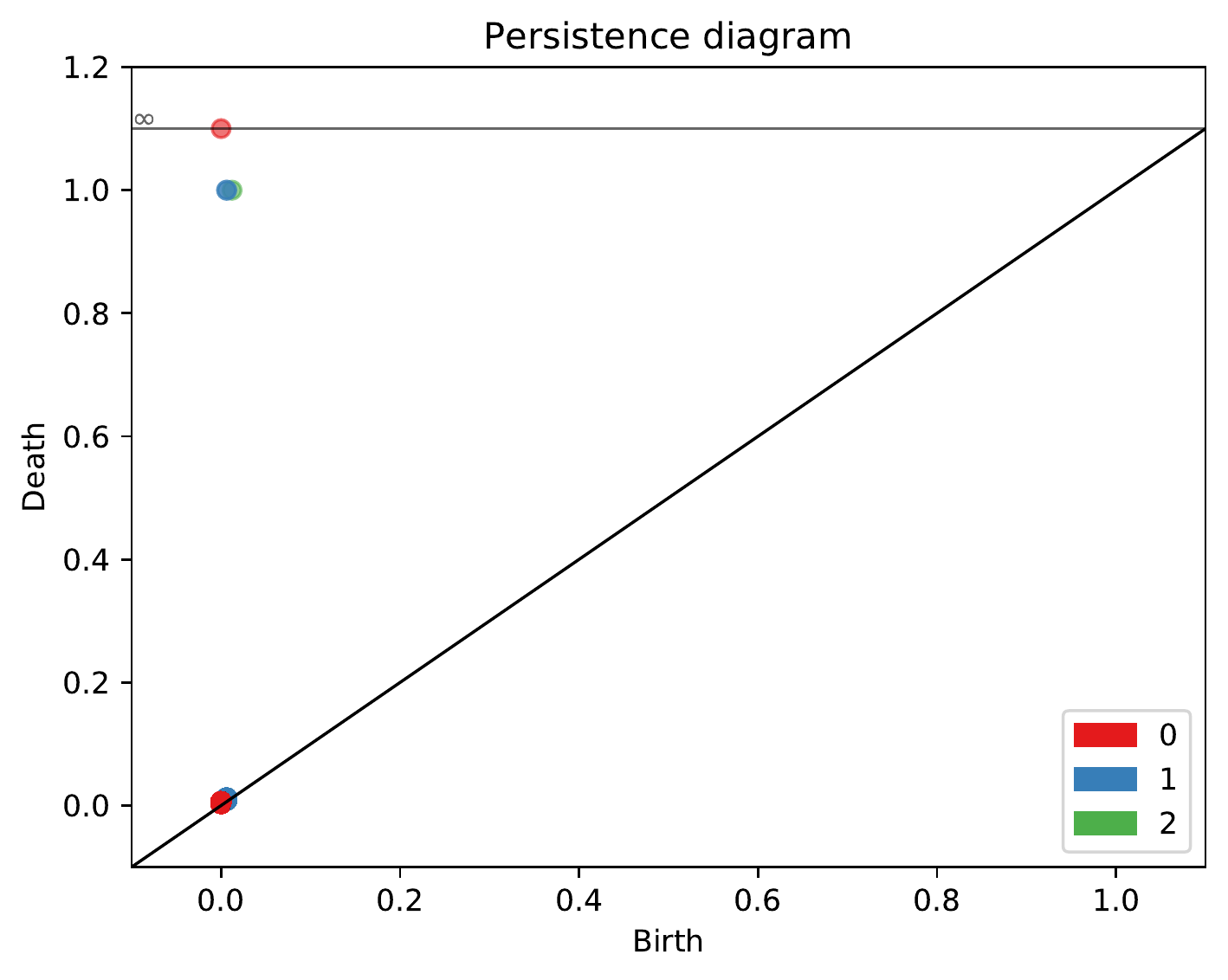}
	\caption{Persistent diagram of $\Alpha(\widetilde{X})$}
	\label{fig:pers_alpha_X}
\end{minipage}
\begin{minipage}[c]{.46\linewidth}
	\includegraphics[scale=0.45]{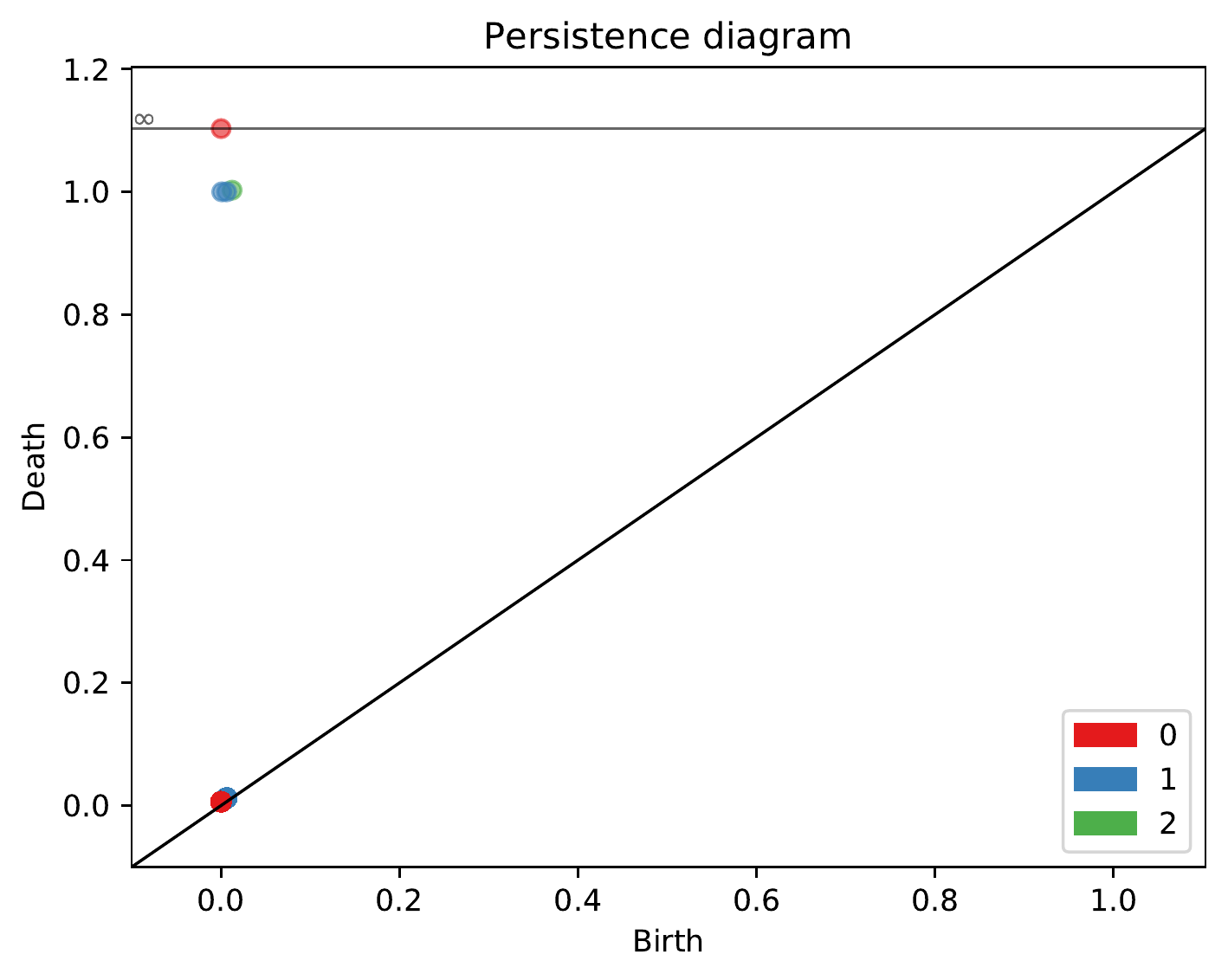}
	\caption{Persistent diagram of $\Alpha(\widetilde{Y})$}
	\label{fig:pers_alpha_Y} 
\end{minipage}
\end{figure}
We can estimate the associative distance as follows. We work over the field $\mathbb{F}_{11}$. 
The two long bars of the barcode of $\rmH^1(\Alpha(\tilde{Y}))$ corresponds to classes whose cup-product has a non-trivial component along the generator of the long bar in $\rmH^2(\Alpha(\tilde{Y}))$. On the other hand the two long bars of the barcode of 
of $\rmH^1(\Alpha(\tilde{X}))$ corresponds to classes whose cup-product is zero. Write $ (t_0^i, t_1^i]$ for those  respective 6 intervals ($i\in \{1\dots, 6\}$).  

\smallskip

Let $\delta =\min_{i}(t^i_1) -\max_{i}(t^i_0)$ be the difference  between the higher starting point of the dimension $1$ and $2$ long bars of $\rmH^*(\Alpha(\tilde{X}))$ and $\rmH^*(\Alpha(\tilde{Y}))$ and  the smaller endpoint corresponding to those long bars in dimension $1$ and $2$. Then there could be  no algebra $\frac{\delta}{2}$-interleaving in between $\rmH^*(\Alpha(\tilde{X}))$ and $\rmH^*(\Alpha(\tilde{Y}))$ because we would have a commutative diagram
\begin{equation}\label{eq:torus-spherecomp}
	\begin{tikzcd}
		& \rmH^1(\Alpha(\tilde{X}))_{\max_{i}(t_0^i)}^{\otimes 2} \arrow[r,"\cup"]& \rmH^2(\Alpha(\tilde{Y}))_{\min_{i}(t_1^i)} \arrow[rd, "\nu"]&\\
		\rmH^1(\Alpha(\tilde{Y}))_{\max_{i}(t_0^i)}^{\otimes 2}  \arrow[rr,"\cong"] \arrow[ru,"\mu"] && \rmH^1(\Alpha(\tilde{Y}))_{\min_{i}(t_1^i)}^{\otimes 2} \arrow[r, "\cup"] & \rmH^2(\Alpha(\tilde{Y}))_{\max_{i}(t_0^i)}
	\end{tikzcd} \ ,
\end{equation}
 for which the lower line will be non-zero while the top compositon is zero. The numerical computation gives us $\delta= 0.987702$ therefore we find that 
\[  0.49385\leqslant \d_{\As}\left(\rmH^*(\Alpha(\widetilde{X})), \rmH^*(\Alpha(\widetilde{Y}))\right)  \leqslant 0.5.\]
\end{example}

\begin{remark}[About computation of the graded algebra distance] \label{remark:computation_multiplicative_distance}
In the above~\Cref{ex:torus_sphere_comput}, it was fairly easy to give a good approximation of the graded associative algebra interleaving distance, since there was essentially very few long bars and one of the algebra structure was trivial, which was giving right away strong constraint on the possible graded algebras interleavings. Similar situations can be exploited for most other kind of distances we discuss in this paper, see~\Cref{ex:Borromeandiscret} for instance. 


We believe this kind of argument can be generalized in practice to give an algorithm to give a lower bound for computing the graded associative algebra distance, though we do not have such as the moment in general. The idea is the following. 

Let $A^*,\, B^*$ be pointwise finite dimensional copersistent graded algebras over the field $k=\mathbb{F}_2$, that is we assume that persistent vector space $\bigoplus_i A^i$ is pointwise finite dimensional. A graded algebra interleaving $A^* \to B^*$ gives an interleaving distance in the usual vector space sense in all degrees at once, but which satisfy the additional constraint of being compatible with the cup-product (such as in diagram~\eqref{eq:torus-spherecomp}). There exists an $\eps$-interleaving in between $A^i$ and $B^i$ if and only if there exists a $\eps$-matching in between their barcode~\cite{bauer2013induced}. Every bar in the barcode corresponds to (co)homology classes, and given two of them, say $a_1\in A^i$, $a_2\in A^j$, for which we can compute their product (in practice  our copersistent algebras arised as cohomology of finite simplicial complexes) which is a sum $\sum_{k\in I} c_k \in A^{i+j}$ where the $c_k$ are cohomology classes associated to bars of the graded barcode of $A^*$. Therefore, if one has an $\eps$-graded algeba interleaving such that  $a_1, a_2$ are $\eps$-matched to two bars $b_1\in B^i$, $b_2\in B^j$, then, denoting $b_1\cup b_2= \sum_{\ell\in J} d_j$,  from diagrams similar to~\eqref{eq:torus-spherecomp}, we get that   
the bars $(c_k)_{k\in I}$ are $\eps$-interleaved (with the cohomology classes representing the) bars $(d_\ell)_{\ell\in J}$ as well as $\eps$-matched with them.

Thus, in order to give an effective lower bound on the graded algebra distance, on can look for an $\eps$-matching between the degree $1$ bars of $A^1$ and $B^1$ such that for any matched pairs $(a_1, b_1)$, $(a_2, b_2)$, there exists an $\eps$-matching between the degree $2$ bars which restricts to $\eps$-matching in between  the corresponding non zero bars $(c_k)$ and $(d_k)$ appearing in  $a_1\cup a_2= \sum c_k$, and $b_1\cup b_2= \sum d_k$. And then further restricts to those matching such that there are similar conditions for cup product between higher degrees bars. Then the lower such $\eps$ is a lower bound for the associative distance. By finiteness, this is only a finite number of steps. Also it is not needed to consider degree 0 bars since the one appearing in the cohomology of a space are units of the corresponding cohomology of their connected component.

Note that this is precisely a generalization of what we computed  in~\Cref{ex:torus_sphere_comput} above.

\end{remark}

\begin{example}\label{ex:twoloops}There are several other types of examples of spaces with similar invariants that will be easily distinguished by their cup-product algebras like $\mathbb{P}^3(\mathbb{C})$ and $S^2\times S^4$.  

Another interesting example is to consider the complement $Y=S^3 \setminus E$ and $X=S^3\setminus T$ of two entangled circles $E$ and a trivial link $T$ of two circles. Both spaces have isomorphic cohomology vector spaces $\rmH^1(X)=k^2=\rmH^1(Y)$, $\rmH^2(X)=k=\rmH^*(Y)$ but they have different algebras. Indeed the cup-product of the generators corresponding to loops around each string of $E$ have a non-trivial cup-product while this is not the case in the trivial case. In that case the cup-product is the Poincar\'e dual of intersection pairing and describe whether or not two loops can be unknotted. 
 
In order to generalize this example to case where there are $n\geq 3$ loops involved, one needs to use higher homotopy structures as such coming from $\calA_\infty$-algebras. See~\Cref{ex:borromean}. 

On can extend this example into a data sets example in the same way as in~\Cref{ex:torus_sphere_comput} and \Cref{ex:Borromeandiscret}.
\end{example}

\section{The \texorpdfstring{$\calA_{\infty}$}-interleaving distance} \label{sect:Ainfty}
One big drawback of working with the homotopy category $\ho(\AsAlg)$ of dg-associative algebras is that it is hard to study, namely because we cannot, in general, invert a quasi-isomorphisms of dg-algebras (that is find a morphism in the opposite direction inducing the inverse in cohomology). One is forced to work with zigzags of morphisms (in other words to consider maps from $X$ to $Y$, one has to consider to pass through any other object $Z$) and to put a complicated equivalence relation on zigzags. One classical way to avoid that in algebraic topology, which will also  relate the structure to the barcode of a persistence space, is to replace dg-algebras by $\calA_\infty$-algebras. We  investigate the associated distance in this section. 

\subsection{Review on associative algebras up to homotopy}
This part is a rapid overview of the theory of $A_\infty$-algebras: the interested reader can consult \cite[Chapter 9]{LV2012} or \cite[Chapter 1]{lefevre2002categories}. 

\begin{definition}[$\calA_\infty$-algebra]\label{def:Ainfty}
	An $\calA_\infty$-algebra is a graded vector space $A=\{A_k\}_{k\in \Z}$,  equipped with an $n$-ary operation
	\[
		m_n:A^{\otimes n} \longrightarrow A \quad \mbox{ of degree $n-2$ for all $n \geqslant  1$,}
	\]
	which satisfy, for all $n\geqslant 1$, the following relation
	\begin{equation}\tag{$\mathrm{rel}_n$}
		\sum_{p+q+r=n} (-1)^{p+qr} m_{p+r+1} \circ ( \id_A ^{\otimes p}\otimes  m_q \otimes \id_A^{\otimes r} )=0 \ .
	\end{equation}
\end{definition}

\begin{remark}
	\begin{enumerate}
		\item By the relation $(\mathrm{rel}_1)$, the map $m_1$ is a differential, and the relation $(\mathrm{rel}_2)$ implies that the binary product $m_2$ is a chain complex map.
		
		\item A differential graded associative algebra $(A,\mu,d_A)$ is an $\calA_\infty$-algebra where $m_1=d_A$, $m_2=\mu$ and, for $n\geqslant 3$, $m_n=0$.
	\end{enumerate}
\end{remark}

The notion of morphisms of $\calA_\infty$-algebras is too rigid  to encode the homotopy theory of $\calA_\infty$-algebras in a practical way, so we use a more flexible notion of morphisms, called $\infty$-morphisms. This  notion of $\infty$-morphism allows to define the category $\infty$-$\AinfAlg$ (see \Cref{def:AinftyAlg_category}): this category has a notion of $\infty$-quasi-isomorphism which we use to define its homotopy category. The point is that each $\infty$-quasi-isomorphism has an inverse in homology in the strict category $\infty$-$\AinfAlg$. So, we do \emph{not} need to consider zigzags of morphisms to understand what is a morphism in the homotopy category of $\calA_\infty$-algebras. This makes the notion of interleaving completely straightforward.

\begin{definition}[$\infty$-morphisms between $\calA_\infty$-algebras]\label{def:Ainftymaps}
	Let $(A,m^A_*)$ and $(B,m^B_*)$ be two $\calA_\infty$-algebras. An $\infty$-morphism $f: A \rightsquigarrow B$ of $\calA_\infty$-algebras is a family of maps $\{f_n :A^{\otimes n} \to B \}_{n\geqslant 1}$ of degree $n-1$ such that $f_1$ is a morphism of chain complexes and, for $n\geqslant 2$, $f_n$ satisfies the relation
	\begin{align*}
		\partial(f_n) & = \
		\sum_{{\substack{p+q+r=n}}} (-1)^{p+qr}f_{p+r+1} \circ (\id_A,\ldots,\id_A,m_q^A,\id_A,\ldots,\id_A) 
		\\
		& \qquad - \sum_{{\substack{ k \geqslant 2 \\ i_1+\ldots+i_k=n}}} 
		(-1)^{\sum_{j=i}^{k-1}(k-j)(i_j-1)} m^B_k \circ (f_{i_1},\ldots , f_{i_k}) \ .
	\end{align*}
	The set of $\infty$-morphism from $A$ to $B$ is denoted by $\Hom_{\infty-\calA_\infty}(A,B)$.
\end{definition}

\begin{remark}
	A (strict) morphism $f:A\to B$ of $\calA_\infty$-algebras, i.e. which satisfies
	\[
		f\circ m^A_n = m_n^B\circ f^{\otimes n}
	\]
	for all $n\geqslant 1$, is canonically a $\infty$-morphism given by the family of maps $(f,0,0\ldots )$.
\end{remark}

\begin{defiprop}[The category $\infty\textsc{-}\Alg_{\calA_\infty}$]
	\label{def:AinftyAlg_category}
	There exists an associative composition of $\infty$-morphisms:
	\[
		\Hom_{\infty-\calA_\infty}(B,C) \times \Hom_{\infty-\calA_\infty}(A,B) \longrightarrow \Hom_{\infty-\calA_\infty}(A,C)
	\] 
	such that, for all $\calA_\infty$-algebra $A$, the identity $A \rightsquigarrow A$ is given by the strict classical one. The $\calA_\infty$-algebras and the $\infty$-morphisms form a category denoted by $\infty\textsc{-}\Alg_{\calA_\infty}$.
\end{defiprop}

An important tool we will use is the fact that a chain complex which is quasi-isomorphic to an $\calA_\infty$-algebra automatically inherits a transferred quasi-isomorphic $\calA_\infty$-structure. In particular, we have the following theorem, which gives us a canonical $\calA_\infty$-structure on the homology of an $\calA_\infty$-algebra.

\begin{theorem}[Homotopy Transfer Theorem  {\cite[Theorem 10.3.7 and 10.3.10]{LV2012}}]
	\label{thm::HTT}
	Let $A$ be an $\calA_\infty$-algebra, $i,p$ two morphisms of chain complexes such that $i$ is a quasi-isomorphism and $h$ a map of degree $1$:
	\[
		\begin{tikzcd}
			(A,d_A)\arrow[r, shift left, "p"] \arrow[loop left, distance=1.5em,, "h"] & (\rmH(A),0) \arrow[l, shift left, "i"]
		\end{tikzcd} \ ,
	\]
	such that $ip - \id_A =d_Ah+hd_A$.
	\begin{enumerate}
		\item There is a $\calA_\infty$-structure on the homology $\rmH_*(A)$ of the underlying chain complex of $A$, which extends its associative algebra structure.
		
		\item The embedding $i$ and the projection $p$, associated to the choice of sections for the homology, extend to $\infty$-quasi-isomorphisms of $\calA_\infty$-algebras.
		
		\item The $\calA_\infty$-structure on the homology $\rmH_*(A)$ is independent of the choice of sections of $\rmH_*(A)$ into $A$ in the following sense: any two such transferred structures are related by an $\infty$-isomorphism, whose first map is the identity on $\rmH_*(A)$.
	\end{enumerate}
\end{theorem}

\begin{remark}[Higher multiplications and Massey products]
	In following examples, we will exhibit how we can use the $\calA_\infty$-structure (and later on $\calA_\infty$-interleaving distance) to differentiate spaces that are not distinguished by the bottleneck distance nor the graded algebra interleaving distance from \Cref{sect:distance_As}. 

	Let $X$ be a topological space, $p,i,h$ be a choice of contraction between $\rmC^*(X)$ and $ \rmH^*(X)$ as in \Cref{thm::HTT} above:
	\[
		\begin{tikzcd}
			(\rmC^*(X),\cup,\partial)\arrow[r, shift left, "p"] \arrow[loop left, distance=2em, "h"] & (\rmH^*(X),0) \arrow[l, shift left, "i"]
		\end{tikzcd} \ ,
	\]
	where $\cup$ is the associative cup product on $\rmC^*(X)$. Then, the $n$-ary operations of  the transferred $\calA_\infty$-structure on $\rmH^*(X)$ gives us some generalized higher Massey products which  we just call \emph{higher Massey products} for short.
	This denomination is not exaclty the classical one, but the two notions are closely related. For more details on classical Massey product and the higher operations $m_n$, the reader can refer to \cite{BMM20}. 
	
	This transferred structure depends of the choice of contraction, but two different choices of contractions gives us two $\mathrm{A}_\infty$-structures on $\rmH^*(X)$ which are $\infty$-isomorphic. 
	Particularly, if for a specific choice of contraction, the $\mathrm{A}_\infty$-structure is trivial (meaning the higher operations $m_n$ are null for $n\geq 3$), then it is true for every choice of contraction. See also \cite[Theorem C]{BMM20}. 
	These higher Massey products  allow to make the difference between some spaces which have the same cohomology as graded algebra but whose  $m_i$'s for $i\geq 3$ defines non-isomorphic $\calA_\infty$-structure. 
\end{remark}

\begin{example}[Borromean rings and trivial entanglement]\label{ex:borromean}



Consider the complement in the $3$-sphere $S^3\subset \R^4$ (of radius $1$) of the Borromean rings $B_3$ (see \Cref{fig:borromean}) and the trivial entanglement of three rings $U_3$ (see \Cref{fig:trivial_entanglement}). These two spaces have the same cohomology as graded algebra  : one has, for any field $k$, 
\[
	(\rmH^*(S^3 \setminus U_3),\cup) \cong (\rmH^*(S^3 \setminus B_3),\cup)
\]
where 
\[
	\rmH^0(S^3 \setminus U_3) = k\ , \quad \rmH^1(S^3 \setminus U_3) =k^3 \quad \mbox{ and } \rmH^2(S^3 \setminus U_3) = k^2 \ .
\] The cup-product are trivial since any two circles one can be entangled  if one forget the third in both examples (see~\Cref{ex:twoloops}).
However, these two spaces are not isomorphic as $\calA_\infty$-algebras, because $\rmH^*(S^3 \setminus U_3)$ has a trivial $m_3$ Massey product while $\rmH^*(S^3 \setminus B_3)$ has a non trivial one (see \cite[Section 9.4.5]{LV2012} or \cite[Proposition 3.5]{belchi2017} for a nice study of this example) for any field $k$.
\begin{figure}[h!]
	\begin{minipage}[c]{.46\linewidth}
		\centering
	\begin{tikzpicture}[scale=0.7]
	\begin{knot}[
		clip width=7,
		]
		\strand [thick] (0,0) circle (1.0cm);
		\strand [thick] (1,0) circle (1.0cm);
		\strand [thick] (0.5,1) circle (1.0cm);
		\flipcrossings{1, 2, 5, 6}
	\end{knot}
	\end{tikzpicture}
	\caption{Borromean rings} 
	\label{fig:borromean}
	\end{minipage}
	\begin{minipage}[c]{.46\linewidth}
		\centering
	\begin{tikzpicture}[scale=0.7]
		\begin{knot}[clip width=7]
		\strand [thick] (0,0) circle (1.0cm);
		\strand [thick] (1,0) circle (1.0cm);
		\strand [thick] (0.5,1) circle (1.0cm);
		\end{knot}
	\end{tikzpicture} 
	\caption{Trivial entanglement}
	\label{fig:trivial_entanglement}
	\end{minipage}
\end{figure} 
We are going below to use this example and a pattern similar to~\Cref{ex:torus_sphere} to give a concrete example of  discretisation which are distinguished by the $\calA_\infty$-analogue of interleaving distance but not by the graded algebra (therefore nor bottleneck either)-ones. See~\Cref{ex:Borromeandiscret}. 

Note that there are similar examples of link that can be proved to be non homotopic to trivial ones using similar Massey products arguments. 
\end{example}

\begin{example}
 The Borromean ring example~\ref{ex:borromean} above has obvious generalisation in higher dimension, allowing to prove that embeddings of higher spheres can not be untangled using higher Massey products. 
\end{example}

\begin{example}[Heisenberg manifold]
Other standard examples where non-trivial higher  Massey products arised and can be studied similarly.  
For instance, one can consider the  Heisenberg group, i.e. the group $\mathrm{Heis}$ of $3\times 3$ matrices of the form 
\[
	\begin{pmatrix}
		1 & a & b \\ 0 & 1 & c \\ 0 & 0 & 1
	\end{pmatrix}
\]
with $a,b,c \in \R$, and let $\Gamma$ be the subgroup of $\mathrm{Heis}$ with integer entries. The Heisenberg manifold is the quotient space $H=(\mathrm{Heis} / \Gamma)$;  it is a 3 dimensional compact manifold. Its cohomology algebra is $\rmH^1(H)= k \alpha \oplus k\beta$,  $\rmH^2(H)= k\gamma  \oplus k\tau$ and $\rmH^3(H)=k \omega$. The only non-trivial (that is the ones not imposed by degrees and the fact that odd degrees elements squared to zero) algebra relations are given by \[ \alpha\cup \beta =0=\alpha\cup \gamma = \beta\cup \tau,\quad \alpha \tau= \omega =-\beta \cup \gamma.\]
This is the same associative graded algebra structure as the connected sum of $S^1\times S^2$ with itself.
However, they are distinguised by the fact that the fact that the Heisenberg manifold has non-trivial Massey products; for instance $m_3(\alpha, \alpha, \beta)= \pm \gamma$. This can be seen since the cochain algebra $\rmC^*(H)$ of the Heisenberg manifold is quasi-isomorphic as a differential graded associative algebra with the graded symmetric algebra
$S(x,y, z)$, where $x,y, z$ are in degrees $1$, equipped with the diffential given by $d(x)=d(y)=0$, $dz= xy$. Note that the classes $\alpha=[x]$, $\beta= [y]$, $\gamma = [xz]$, $\tau=[yz]$ and $\omega =[xyz]$. See \cite[Chapter~2]{tralle2006symplectic} for more details.
One can build numerical examples similar to~\Cref{ex:torus_sphere_comput} and~\Cref{ex:Borromeandiscret} for which the $\calA_\infty$-distance allows to differentiate discretisation of the Heisenberg manifold from discretisation of the connected sum of two copies of $S^1\times S^2$ while the bottleneck or the graded algebra distance can not.  

\medskip  

The Heisenberg manifold example has many higher dimensional analogues with interesting non-trivial and useful higher Massey-products. For instance the famous Kodaira-Thurston manifold which is 
$
	K \coloneqq H  \times S^1 $ whose cohomology of $K$ with rationnal coefficients is  
\[
	\rmH^1(K;\Q) = \Q^3 \ , \ \rmH^2(K;\Q) = \Q^4 \ , \mbox{ and } \ \rmH^3(K;\Q) = \Q^3 \ .
\]
There are many symplectic non-K\"ahler manifolds that can be detected using non-triviality of Massey products and constructed along the same lines; see \cite{rudyak2000thom}.
\end{example}

The notion of $\calA_\infty$-algebra and $\calA_\infty$-morphisms gives a nice and practical model for the homotopy category of dg-algebras. 

\begin{theorem}[Equivalence of homotopy category (see {\cite{lefevre2002categories}} or  {\cite[Theorem 11.4.8]{LV2012}})]\label{thm:equivalence_homotopy_category}
	The homotopy category of differential graded associative algebras and the homotopy category of $\calA_\infty$-algebras with the $\infty$-morphisms are equivalent:
	\[
	\ho\left(\AsAlg\right) \cong \ho\left(\infty\textsf{-}\AsAlg\right) \cong
	\ho\left(\infty\textsf{-}\Alg_{\calA_\infty}\right).
	\]
\end{theorem}
In particular, one can use $\calA_\infty$-algebra to study the homotopy category of algebras. We will in the next section explain under which condition we can use this structure in the persistance context.

\subsection{The problem of persistent transfer datum}\label{sec::filtered_data}
For a persistent topological space $X:\sfR\rightarrow \Top$ and $a>b$, we want to associate canonically contractions datum:
\[
	\begin{tikzcd}
	\rmC^*(X_a) \arrow[r, shift left, "p_a"] \arrow[loop left, distance=1.5em, "h_a"] \arrow[d,"\rmC^*(\phi_{ab})"'] 
	& \rmH^*(X_a)  \arrow[l, shift left, "i_a "] \\
	\rmC^*(X_b) \arrow[r, shift left, "p_b"] \arrow[loop left, distance=1.5em, "h_b"] 
	& \rmH^*(X_b)  \arrow[l, shift left, "i_b "] 
	\end{tikzcd} \ ;
\]
This is not possible in general, because we have to make choices, which have no reason to be compatible with the persistence structure maps (see~\Cref{R:notransfer}). Therefore we restrict our category of persistent topological spaces to a subcategory of objects satisfying some  finiteness conditions. These conditions are automatically satisfied for a set of data such as those arising in applications. 

Let us denote by $\DCpx$, the category of delta complexes as in~\cite{Hatcher}. 
Note that for \v{C}ech and Rips complexes arising from a data set, it is sufficient to consider simplicial complexes and not general delta complexes.

\begin{definition}[Finite filtered data]\label{def:ffdata}
We consider the full subcategory of the category of persistent delta-complex  $\Func(\sfR,\DCpx)$ with objects $X :\sfR \to \DCpx$ satisfying:
\begin{enumerate}
	\item for all $r$ in $\sfR$, the delta complex $X_r$ is finite;
	\item for all $a<b$ in $\sfR$, the morphism $X_a \hookrightarrow X_b$ is an injection;
	\item the set $\left\{X_r ~|~ r\in\sfR \right\}/_\sim$, where $X_a \sim X_b$ if the structural morphism $X_{a<b}$ is an isomorphism, is finite ;  
	\item for each $a$ in $\sfR$, we have a total order $\leqslant_a$ on $X_a$ such that
	\begin{itemize}
		\item for $\alpha$ and $\beta$ in $X_a$ such that $\alpha \subset \beta$ then $\alpha \leqslant_a \beta$ ;
		\item for all $a<b$ in $\sfR$, for all $\beta$ in $X_b \backslash X_a$, and all $\alpha$ in $X_a$, then $\alpha \leqslant_a \beta$.
	\end{itemize}
\end{enumerate}
We call this category  \emph{the category of finite filtered data} and we denote it by $\mathsf{fData}^\sfR$. 
\end{definition} 

Our motivating example for this definition is the following.
\begin{example}\label{ex:finiteRips}
	Let $X$ be a finite set of points in a metric space. Then, $\Cech(X)_\bullet$ and $\calR(X)_\bullet$ are objects  in $\mathsf{fData}^\sfR$ (see  \Cref{def:Rips} and \Cref{def:Cech}).
\end{example}
\begin{example}\label{ex:finiteMorse}
 Let $X$ be a compact smooth (or piecewise linear) manifold and $f: X\to \R$ be a smooth (or PL) Morse function. Then  the sublevel sets $f^{-1}((-\infty, t])$  are a family compact  manifolds with boundary  which can be given finite $\Delta$-complexes structures. As long as $ t \in [s, s')$ where $s$ and $s'$ are two critical values of $f$, the $\Delta$-structures can be taken to be homotopic and therefore, since they are finitely many critical values, we can get a  persistent $\Delta$-complex $(\rmC_*(f^{-1}((-\infty, t)))_{t\in \R}$ (where $C_*$ is the chain complex computing simplicial homology) which is a finite filtered data.  Its homology is the standard sublevel set $(\rmH_*(f^{-1}((-\infty, t]))_{t\in \R}$. 
\end{example}

The conditions of a finite filtered data ensures that the  interleaving distances are closed on this subcategory:

\begin{lemma}\label{lem:distance_nulle}
	Let $X$ and $Y$ be two finite filtered data and $F:\Top \to \sfC$ be a functor. We have that $\d_{\sfC}(F(X),F(Y))=0$ if, and only if, the persistent objects $F(X)$ and $F(Y)$ are isomorphic in $\sfC^{\sfR}$.
\end{lemma}

\begin{proof}
	We denote by $r_1,\ldots,r_n$, the objects of the category $\sfR$ such that, for all $1\leqslant i \leqslant n$ and for all $\eps>0$ in $\sfR$, $F(X)_{r_i-\eps}\ncong F(X)_{r_i}$ or $F(Y)_{r_i-\eps}\ncong F(Y)_{r_i}$.
	
	Suppose that $\d_{\sfC}(F(X),F(Y))=0$. Let $t<r_n$ be in $\sfR$: there exists $1\leqslant i \leqslant n$ such that $r_i \leqslant t < r_{i+1}$. We consider  $\eps= \frac{r_{i+1}-t}{3} $. By assumption, there exists an $\eps$-interleaving between $F(X)$ and $F(Y)$:
	\[
	\begin{tikzcd}
	F(X)_t \ar[r] & F(Y)_{t+\eps} \ar[r] & F(X)_{t+2\eps}  ;
	\end{tikzcd}
	\]
	as $t+2\eps<r_{i+1}$,  then $F(X)_t \cong  F(X)_{t+2\eps}$, so $F(X)_t \cong F(Y)_{t+\eps} \cong F(Y)_{t}$. By the same argument, for $t\geqslant r_n$, $F(X)_t \cong F(Y)_t$.
\end{proof}

\subsection{Contractions in family}\label{sub:contraction}
In order to apply the homotopy transfer theorem for a persistent space given by a  finite filtered data, we need to encode some part of the data allowing to obtain the transferred structure in an explicit way. This is the role of the following of definition.

\begin{definition}[Category of transfer data]
	The category $\Trans_{\Ch}$ (resp, $\Trans_{\coCh}$) is defined as follows: its objects are $(A,H,i,p,h)$ such that
	\[
	\begin{tikzcd}
	(A,d_A)\arrow[r, shift left, "p"] \arrow[loop left, distance=1.5em,, "h"] & (H,0) \arrow[l, shift left, "i"]
	\end{tikzcd} \ ,
	\]
	where $A$ is a chain complex (resp. cochain complex), $H$ is a graded vector space, $i$ and $p$ are morphisms of chain complexes, $h$ is a linear map of homological degree $1$ (resp. homological degree $-1$) such that:
	\[
	ip -\id_A = d_Ah+hd_A
	\]
	and such that they satisfy the \emph{side conditions}:
	\[
	 \quad pi=\id_H, \quad hi=0, \quad ph=0\quad \mbox{and} \quad h^2=0.
	\]	
	
	The morphisms $(A,H,i,p,h)\to (A',H',i',p',h')$ in the category $\Trans_{\Ch}$ are defined as the morphisms $A \to A'$ of chain complexes. We also define the category $\Trans_{\Ch,\As}$ to be the subcategory of $\Trans_{\Ch}$ whose objects are $(A,H,i,p,h)$ such that $A$ is a differential graded associative algebra, and the morphisms are morphisms of differential graded associative algebras.
\end{definition}

\begin{remark}\label{R:notransfer}
	Let $\phi:(A,H,i,p,h)\to (A',H',i',p',h')$ be a morphism in $\Trans_\Ch$. 
	We do not suppose any compatibility between the morphism $\phi$ and the structural morphisms $i,p,h$ and $i',p',h'$, contrarily to the classical notion of \cite{manetti2017projective}. The reason is that we do not have a persistent inclusion of the homology of a chain complex in it, even in the simplest cases, as the following example shows. Let $X$ be the union of two points $(0,0)$ and $(1,0)$ in $\R^2$, and consider the associate Vietoris-Rips complex $\calR(X)_\bullet$. We have the table of \Cref{tab::example}. 
	
	\begin{figure}[h!] 
		\centering
		\begin{tabularx}{16cm}{c|c|c|c}
			Radius $\eps$ & Picture & $\rmC_*(\calR(X))_\eps$ & $\rmH_0(\calR(X))_\eps$ \\
			$\eps<\frac{1}{2}$ & 
			\begin{tikzpicture}[baseline=(n.base)]
			\node (n) at (0,-0.2) {};
			\pgfmathsetmacro{\RAD}{0.3pt};
			\coordinate (A) at (0,0);
			\coordinate (B) at (1,0);
			\draw[fill=gray!20, opacity=0.2, thick] (A) circle (\RAD);
			\draw[fill=gray!20, opacity=0.2, thick] (B) circle (\RAD);
			\draw[thick] (A) node {$\scriptstyle{\bullet}$} node[below] {$\scriptstyle{x_1}$} ;
			\draw[thick] (B) node {$\scriptstyle{\bullet}$} node[below] {$\scriptstyle{x_2}$} ;
			\end{tikzpicture}
			& 
			$kx_1 \oplus kx_2 \overset{d}{\longleftarrow} 0$
			& 
			$kx_1 \oplus kx_2 $
			\\
			$\eps>\frac{1}{2}$ & 
			\begin{tikzpicture}[baseline=(n.base)]
			\node (n) at (0,-0.2) {};
			\pgfmathsetmacro{\RAD}{0.7pt};
			\coordinate (A) at (0,0);
			\coordinate (B) at (1,0);
			\draw[fill=gray!20, opacity=0.2, thick] (A) circle (\RAD);
			\draw[fill=gray!20, opacity=0.2, thick] (B) circle (\RAD);
			\draw[thick, fill=gray!50, opacity=0.7] 
			(A) node {$\scriptstyle{\bullet}$} node[below] {$\scriptstyle{x_1}$} to 
			(B) node {$\scriptstyle{\bullet}$} node[below] {$\scriptstyle{x_2}$};
			\draw[thick] (0.5,-0.1) node[above]  {$\scriptstyle{x_{12}}$};
			\end{tikzpicture}
			& 
			$kx_1 \oplus kx_2 \overset{d}{\longleftarrow} kx_{12}$
			& 
			$kx_1$
		\end{tabularx}
		\caption{Simplicial chain complex and simplicial homology of the Vietoris-Rips complex associated to $\{(0,0),(1,0)\}$}
		\label{tab::example}
	\end{figure}
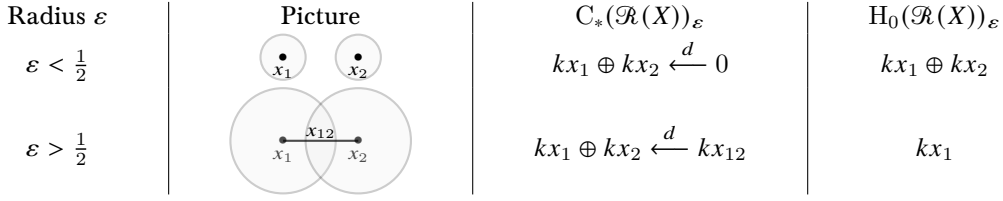 
	
	Suppose that we have a persistent inclusion $i_\bullet:\rmH_*(\calR(X))_\bullet \to \rmC_*(\calR(X))_\bullet$. Then, we obtain, for all $\eps< \frac{1}{2}$, the following diagram
	\[
		\begin{tikzcd}
			kx_1\oplus kx_2 \ar[r,"i_\eps","\cong"'] \ar[d,"(\eps\leqslant \eps+\frac{1}{2})"'] 
				& kx_1\oplus kx_2 \ar[d, "(\eps\leqslant \eps+\frac{1}{2})","\cong"'] \\
			kx_1 \ar[r,"i_{\eps+\frac{1}{2}}"'] & kx_1\oplus kx_2
		\end{tikzcd}
	\]
	which is \emph{not} commutative. Therefore, it cannot exist such a persistent inclusion. By the same argument, it cannot exist a persistent projection $p^\bullet:\rmH^*(\calR(X))^\bullet \to \rmC^*(\calR(X))^\bullet$ from the persistent cohomology to the persistent cochain complex.
\end{remark}

\begin{notation}
	We denote by  $\Trans_{\Ch}^{\sfR}$ the  category of functors $\sfR\to \Trans_{\Ch}$, that is the category of persistent transfer data. We also denote by  $\fTrans_{\Ch}^{\sfR}$, the full subcategory of $\Trans_{\Ch}^{\sfR}$ of objects $A_{\bullet}$ such that the set $\left\{A_r ~|~ r\in\sfR \right\}/_\sim$, where $A_\alpha \sim A_\beta$ if the structural morphism $A_{\alpha<\beta}$ is an isomorphism, is finite.
\end{notation}

The functor of simplicial complex $\rmC_*:\DCpx \rightarrow \Ch_k$ induces a functor
\[
	\rmT\rmC_*:\mathrm{fData}^\sfR \longrightarrow \fTrans_{\Ch}^\sfR 
\]
where the contraction are given by \cite[Algorithm 1]{real2009cell}: consider a $\sfR$-filtered data $X:\sfR\to\DCpx$. For all $r$ in $\sfR$, $X_r$ is totally ordered, and, as $X_a \hookrightarrow X_b$ for all $a<b$ in $\sfR$, we denote
$X_0=\left\{c_0,\ldots,c_{i_0} \right\}, X_1=\left\{c_0,\ldots,c_{i_0},c_{i_0+1} \ldots c_{i_1} \right\}$ and for all $r$ in $\sfR$,
\[
	X_r=\left\{ c_0, c_1,\ldots c_{i_0},c_{i_0+1}, \ldots, c_{i_1}, \ldots, c_{i_r} \right\}.
\]
Fix $r$ in $\sfR$, denote by $m$, the cardinal of $X_r$, we construct  algorithmically a linear map $h_m$ of homological degree one on  $(\bigoplus_{i=1}^m kc_i,\partial_m)$ as follows:
\[
\begin{array}{ll}
	&C_0\coloneqq \{c_0\},\partial_0, h_0(c_0)\coloneqq 0 \\
	&\texttt{For} \ i=1 \ \texttt{to} \ m \ \texttt{do}\\
	&\qquad C_i\coloneqq \left\{C_{i-1}\cup\{c_i\},\partial_i \right\}\\
	&\qquad \texttt{If } (\partial_i -\partial_{i-1}h_{i-1}\partial_i)(c_i)=0,\texttt{ then} \\
	&\qquad \qquad h_i(c_i)\coloneqq 0 \\
	&\qquad \qquad \texttt{For } j=0 \texttt{ to } i-1 \texttt{ do} \\
	&\qquad \qquad \qquad h_i(c_j)=h_{i-1}(c_j) \\
	&\qquad  \texttt{If } (\partial_i -\partial_{i-1}h_{i-1}\partial_i)(c_i)=\sum_{k=1}^n \lambda_k u_k \ne 0 \texttt{ with } u_1<\ldots<u_k<\ldots <u_n\in C_{i-1}, \texttt{ then} \\
	&\qquad \qquad \bar{\phi}(u_1)\coloneqq -\lambda_1^{-1}c_i \texttt{ and } \bar{\phi}(u_k)\coloneqq 0 \texttt{ otherwise} \\
	&\qquad \qquad \texttt{For } j=0 \texttt{ to } i\texttt{ do} \\
	&\qquad \qquad \qquad h_i(c_j)=(h_{i-1}+\bar{\phi}-\bar{\phi}h_{i-1}\partial_{i-1}\bar{\phi}\partial_{i-1}h_{i-1})(c_j)
	\\
	&\texttt{Output:} h\coloneqq h_m
\end{array} 
\]
Then, for each $r$ in $\sfR$, the functor $\rmC_*$ sends $X_r$ on the contraction $((\rmC_*(X_r),\partial), \mathrm{Im}(\pi),\iota,\pi,h)$ where $\pi\coloneqq \id - \partial h-h\partial$, $\iota$ is the inclusion of $\mathrm{Im}(\pi)$ in $\rmC_*(X_r)$, and $h$ given by the previous algorithm. Finally, we have define a functor
\[
	\rmT\rmC_*:\mathrm{fData}^\sfR \longrightarrow \fTrans_{\Ch}^\sfR .
\]
Composition with the functor $\Hom_{k}(-,k)$ gives us a functor
\begin{equation}\label{eqdef:TC}
	\rmT\rmC^*:\mathrm{fData}^\sfR \longrightarrow \fTrans_{\coCh,\As}^{\sfR^\op} \ 
\end{equation}
which we call the \emph{persistent transfer data dg-algebra}  functor.

\subsection{$\calA_\infty$-interleaving distance}\label{subsec:distance_Ainfty}

\begin{theorem}\label{prop:functor_to_Ainfty}
There is a functor 
	\[
	\begin{array}{rcc}
		\rmH_* : \fTrans_{\coCh,\As}^{\sfR^\op} & \longrightarrow & \ho\left(\infty\textsf{-}\Alg_{\calA_\infty}\right)^{\sfR^\op} \\
		(A,H,i,p,h)^\bullet & \longmapsto & (H,\{\mu_i\}_{i\in\N })^\bullet
		\end{array} \ 
	\] whose composition with the forgetful functor 
	$ \ho\left(\infty\textsf{-}\Alg_{\calA_\infty}\right)^{\sfR^\op} \to (\grVect)^{\sfR^{\op}}$  is the underlying persistent cohomology.  We call $\rmH_*$ the \emph{transferred $\calA_\infty$-structure} functor.
\end{theorem}

\begin{proof}
	Let $(A,H,i,p,h)^\bullet$ be an object in $\fTrans_{\coCh,\As}^{\sfR^\op}$, and consider $t_0,\ldots,t_n$ in $\sfR$ such that, for all $1 \leqslant j \leqslant n$, for all $\eps>0$, $A_{t_j-\eps} \ncong A_{t_j}$ as differential graded algebras.
	By the Homotopy Transfer Theorem (see \Cref{thm::HTT}), for each $t_j$, $H_{t_j}$ has an $\calA_\infty$-structure, and for all $t_i\leqslant t < t_{j+1}$, as $H_{t_j}$ and $H_{t}$ are isomorphic as graded vector spaces. We can thus put on $H_{t}$ the same $\calA_{\infty}$-structure as on $H_{t_j}$ and take the identity as the structural morphism between them:
	\[
		(t_j\leqslant t) : H_{t_j} \overset{\id}{\longrightarrow} H_t \ .
	\]
	We need to construct the structural morphism $ \eta_i^H\coloneqq (t_j\leqslant t_{j+1})_H:H_{t_{j}} \to H_{t_{j+1}} $. If we denote by $\eta^A_i$, the structural morphism $A_{j_i} \to A_{t_{j+1}}$, we define $ \eta_i^H$ by the following composition:
	\[
	\begin{tikzcd}
		H_{t_{j}} \arrow[d,  "i_{j}"'] \arrow[r, dotted, "\eta_i^H"] & H_{t_{j+1}} \\
		A_{t_j} \arrow[r,"\eta_i^A"']  & A_{t_{j+1}} .\arrow[u, shift left, "p_{j+1}"']
	\end{tikzcd}
	\]
	Then, by \Cref{thm::HTT}, $i_{j+1}$ is a $\infty$-quasi-isomorphism and $p_{j+1}$ is its inverse, so the following diagram is homotopy commutative
	\[
	\begin{tikzcd}
	H_{t_{j}} \arrow[d,  "i_{j}"'] \arrow[r,  "\eta_i^H"] & H_{t_{j+1}} \arrow[d, "i_{j+1}"] \\
	A_{t_j} \arrow[r,"\eta_i"']  & A_{t_{j+1}} 
	\end{tikzcd} \ .
	\]
So we have proved that the datum of $H_\bullet$ with transferred structure and morphisms defined previously is a persistent object in the category $\ho\left(\infty\textsf{-}\Alg_{\calA_\infty}\right)$.
\end{proof}

Combining the persistent transfer data dg-algebra functor~\eqref{eqdef:TC} and the transferring $\calA_\infty$-structure functor of \Cref{prop:functor_to_Ainfty}, we obtain the following definition.

\begin{definition}\label{def:Ainftycochainalgebra} 
	We define the \emph{$\calA_\infty$-algebra homology functor} as the composition
	\[
	 \rmH_* \circ \rmT\rmC^*: \mathsf{fData}^\sfR \longrightarrow \ho\left(\infty\textsf{-}\Alg_{\calA_\infty}\right)^{\sfR^\op}.
	\]
\end{definition}

It is therefore a functor from finite filtered data to the (homotopy) category fo $\calA_\infty$-algebras. This functor is the $\calA_\infty$ analogue of the cochain algebra functor from persistent spaces to (homotopy classes of) differential graded algebras (see~\Cref{def:cupproduct} and~\Cref{nota:cochainfunctor}).

\begin{definition}[$\calA_\infty$-interleaving distance]
	Let $X$ and $Y$ be two filtered data. We define the \emph{$\calA_\infty$-interleaving distance} by 
	\[
	\d_{\calA_\infty}(X,Y) \coloneqq
	\d_{\ho(\infty\textsf{-}\Alg_{\calA_\infty})}(\rmH_*(\rmT\rmC^*(X)),\rmH_*(\rmT\rmC^*(Y))).
	\] 
	where the functor $\rmH_*\circ \rmT\rmC^*$ is given by~\Cref{def:Ainftycochainalgebra}. 
\end{definition}

The $A_\infty$-interleaving distance realizes the interleaving distance in the homotopy category of differential graded associative algebras, see~\Cref{prop:inequalities_Ainfty} below. The point of the $\calA_\infty$-distance is that we only need to consider actual $\calA_\infty$-morphisms (instead of zigzags passing to arbitrary intermediate dg-algebras) to study interleavings.
\begin{example}
By examples~\ref{ex:finiteRips} and~\ref{ex:finiteMorse}, we can associate a $\calA_\infty$-persistence structure and $\calA_\infty$-interleaving distance to Rips and Mayer-Vietoris complexes of a data set as well as to the sublevel sets of Morse functions on compact smooth manifolds. 
\end{example}

\begin{remark}
	F. Belch\'i \emph{et al.} give a definition of $\calA_\infty$-barcode, and consider the associated bottleneck distance (see \cite{belchi2015,belchi2015,belchi2017,belchi2019}). They  work with the transferred $\calA_\infty$-coalgebra structure $\{\Delta_n\}$ of the homology of space and construct the $\calA_\infty$-barcode using the kernel of the coproducts $\Delta_n$. However, this definition of barcode has a drawback: they only consider the kernel of the first coproduct $\Delta_n$ which is not trivial because the kernel of the higher coproducts depend highly on the transfer data (see \cite[Section 3]{belchi2017}). Therefore this definition lose a large part of the $\calA_\infty$-structure in general and thus some reachable topological information.
\end{remark}

\begin{proposition}\label{prop:inequalities_Ainfty}
	Let $X$ and $Y$ be two filtered data.
	\begin{enumerate} 
	\item We have the inequality
	$
	\d_{\calA_\infty}(X,Y)
	\leqslant
	\d_{\Alg_{\As}}\left(\rmC^*(X),\rmC^*(Y)\right) \ .
	$
	\item There is the equality $
	\d_{\calA_\infty}(X,Y)
	=
	\d_{\ho(\Alg_{\As})}(\rmC^*(X),\rmC^*(Y)) \ .
	$
	\end{enumerate}
\end{proposition}

What we are really interested about is point (2) of the proposition.

\begin{proof}
	\begin{enumerate}
		\item It is an immediate consequence of \Cref{lem::apply_functor} and \Cref{prop:functor_to_Ainfty}.
		\item By the Homotopy Transfer Theorem (see \Cref{thm::HTT}),  for each $r$ in $\sfR$, we have the following two $\infty$-quasi-isomorphisms
		\[
		\begin{tikzcd}
		\rmC^*(X)^r \arrow[r, squiggly, shift left=1.5, "p^r_\infty"] & 
		\rmH^*(X)^r \arrow[l, squiggly, shift left=1.5, "i^r_\infty"]
		\end{tikzcd} \ 
		\] which are quasi-inverse of each others. 
		Therefore, for each $r$ in $\sfR$, $\rmH^*(X)^r$ and $\rmC^*(X)^r$ are isomorphic in the homotopy category of $\calA_\infty$-algebras so that we have
		\[
		\d_{\calA_\infty}(X,Y)
		=
		\d_{\ho(\infty\textsf{-}\Alg_{\calA_\infty})}(\rmC^*(X),\rmC^*(Y)) \ .
		\]
		By the \Cref{thm:equivalence_homotopy_category}, we have $\ho(\infty\textsf{-}\Alg_{\calA_\infty}) \cong \ho(\Alg_{\As})$, so we deduce the result.
	\end{enumerate}
\end{proof}

\begin{remark}
	Recall that to define the distance $\d_{\ho(\Alg_{\As})}(\rmC^*(X),\rmC^*(Y))$, we do not need any finiteness assumption on the functors $X,Y: \sfR \to \Top$.
\end{remark}

\begin{example}[Borromean rings]\label{ex32}\label{ex:Borromeandiscret}
	We now gives a finite filtered data version of \Cref{ex:borromean}, following the strategy of \Cref{ex:torus_sphere}.

 We fix an embedding of the borromean rings $B_3$ and the trivial entanglement $U_3$ in the $3$-sphere $S^3\subset \R^4$ of radius $R>0$, see~\Cref{ex:borromean}.

 Let  $0<\ell< \frac{R}{2}$, such that 
\[
	(B_3)_\ell \coloneqq \bigcup_{x\in B_3}	\rmB(x,\ell) \simeq B_3
	\ \mbox{ and } \
	(U_3)_\ell \coloneqq \bigcup_{x\in B_3}	\rmB(x,\ell) \simeq U_3 \ ,
\]
so the choice of the positive number $\ell$ depends of the choice of the embeddings of $B_3$ and $U_3$ in $S^3$, but can be made to be large enough provided $R$ is. For $\eps \geqslant 0$, we denote
\[
	X^\ell_\eps \coloneqq \bigcup_{x \in  (S^3 \setminus (B_3)_\ell)} \rmB(x,\eps)
	\qquad \mbox{and} \qquad 
	Y^\ell_\eps \coloneqq \bigcup_{y \in  (S^3 \setminus (U_3)_\ell)} \rmB(y,\eps) \ .	
\] 
Remark that, for $\eps < \ell$, we have homotopy equivalences $X^\ell_\eps \simeq S^3 \setminus (B_3)_\ell$ and $Y^\ell_\eps \simeq S^3 \setminus (U_3)_\ell$ ; for $\ell <\eps < 1$, $X^\ell_\eps \simeq S^3$ and $Y^\ell_\eps \simeq S^3$ ; and, for $\eps \geqslant 1$, $X^\ell_\eps$ and $Y^\ell_\eps$ are contractile. Further $S^3 \setminus (B_3)_\ell$ and $S^3 \setminus (U_3)_\ell$ are compact subspaces of $\R^4$. 

\smallskip

Now we construct discretizations of those subspaces of $\R^4$.  
Fix $0< \alpha \ll  \ell$ and we choose a finite cover by open balls (in $\R^4$) of radius $\alpha$ of $X^\ell_\alpha$ and $Y^\ell_\alpha$:
\[
	X^\ell_\alpha \subset \bigcup_{i \in  I} \rmB(x_i,\alpha)
	\qquad \mbox{and} \qquad 
	Y^\ell_\alpha \subset \bigcup_{j \in  J} \rmB(y_j,\alpha) 	
\]
where $I$ and $J$ are finite. We can then define the discrete spaces
\[
	D_X^{\alpha,\ell} = \bigcup_{i \in  I} \{x_i\}
	\qquad \mbox{and} \qquad 
	D_Y^{\alpha,\ell} = \bigcup_{j \in  J} \{y_j\}.
\]
We can think these spaces as noisy discretisations of our spaces $S^3\setminus (B_3)_\ell$ and $S^3\setminus (U_3)_\ell$. Remark that, for $\alpha \leqslant \eps < \ell-\alpha$, we have the following homotopy equivalences 
\begin{equation}\label{eq:htpyequiv}
	\Cech(D_X^{\alpha,\ell})_{\eps} \simeq 	X^\ell_\alpha \simeq S^3\setminus (B_3)_\ell
	\qquad \mbox{and} \qquad 
	\Cech(D_Y^{\alpha,\ell})_{\eps} \simeq 	Y^\ell_\alpha \simeq S^3\setminus (U_3)_\ell.
\end{equation}

\begin{figure}[h!]
	\begin{tikzpicture}
	\draw (0.7,0.4) -- (11.3,0.4);
	\draw (1,0.4) node {|} node[below=1.5] {$\scriptscriptstyle{0}$};
	\draw[fill,color=gray!10] (1,0.8) rectangle (2,3);
	\draw (1.5,1.9) node {\tiny{"noise"}};
	\draw (2,0.4) node {|} node[below=1.5] {$\scriptscriptstyle{\alpha}$};
	\draw (5,0.4) node {|} node[below=1.5] {$\scriptscriptstyle{\ell-\alpha}$};
	\draw (6,0.4) node {|} node[below=1.5] {$\scriptscriptstyle{\ell+\alpha}$};
	\draw (9,0.4) node {|} node[below=1.5] {$\scriptscriptstyle{1-\alpha}$};
	\draw (10,0.4) node {|} node[below=1.5] {$\scriptscriptstyle{1+\alpha}$};
	\draw[blue] (2,1.5) -- (5,1.5) ;
	\draw[blue] (2,1.6) -- (5,1.6) ;
	\draw[blue] (2,1.7)  node[above right] {$\scriptscriptstyle{\rmH^1}$} -- (5,1.7);
	\draw[red] (2,2.2) -- (5,2.2) ;
	\draw[red] (2,2.3)  node[above right] {$\scriptscriptstyle{\rmH^2}$} -- (5,2.3) ;
	\draw[purple] (6,2.6) -- (9,2.6) node[above left] {$\scriptscriptstyle{\rmH^3}$};
	\draw (2,1) -- (11,1) node[above left] {$\scriptscriptstyle{\rmH^0}$};
	\draw[fill,color=gray!10] (5,1.2) rectangle (6,3);
	\draw (5.5,2.1) node {\tiny{"noise"}};
	\draw[fill,color=gray!10] (9,2.2) rectangle (10,3);
	\draw (9.5,2.6) node {\tiny{"noise"}};
	\end{tikzpicture}
	\caption{A part of the barcode of $\rmH^*(\Cech(D_X^{\alpha,\ell})_\bullet)$ and $\rmH^*(\Cech(D_Y^{\alpha,\ell})_\bullet)$}
	\label{fig::cpx_borromean}
\end{figure}

\begin{proposition}\label{Prop:Borromean}
 For any $0<\alpha <\ell$, one has
 \[
	\d_{\grVect}\left(\Cech(D_X^{\alpha,\ell}),\Cech(D_Y^{\alpha,\ell})\right) \leqslant \d_{\As}\left(\Cech(D_X^{\alpha,\ell}),\Cech(D_Y^{\alpha,\ell})\right) \leqslant 2\alpha \ , \ .
\]
while \[
	\frac{\ell-2\alpha}{2} \leqslant \d_{\calA_\infty}\left( \Cech(D_X^{\alpha,\ell}),\Cech(D_Y^{\alpha,\ell})\right)\leqslant \frac{\ell+\alpha}{2} \ .
\]
\end{proposition}
In particular as soon as $\alpha< \frac{\ell}{4}$, we have
\[
	\d_{\grVect}\left(\Cech(D_X^{\alpha,\ell}),\Cech(D_Y^{\alpha,\ell})\right)
	<\d_{\calA_\infty}\left(\Cech(D_X^{\alpha,\ell}),\Cech(D_Y^{\alpha,\ell})\right).
\]
and the smaller is $\alpha$, the bigger is the difference between the two distances: the associative and bottleneck distance converging to zero while the $\calA_\infty$ one goes to $\frac{\ell}{2}$.

\begin{proof} Recall that the complement in the $3$-sphere $S^3\subset \R^4$  of the Borromean rings $B_3$  and the trivial entanglement of three rings $U_3$  have isomorphic (trivial) cohomology algebras  
	where 
	\[
		\rmH^0(S^3 \setminus U_3) = k\ , \quad \rmH^1(S^3 \setminus U_3) =k^3 \quad \mbox{ and } \rmH^2(S^3 \setminus U_3) = k^2 \ .
	\] 
We thus get from the homotopy equivalences~\eqref{eq:htpyequiv} that, 	
 for all $\eps \in [\alpha,\ell-\alpha] \cup [\ell+\alpha, 1]$, we have a graded algebra isomorphism $\rmH^*(\Cech(D_X^{\alpha,\ell})) \underset{\AsAlg}{\cong} \rmH^*(\Cech(D_Y^{\alpha,\ell}))$. Therefore 
\[
	\d_{\As}\left(\Cech(D_X^{\alpha,\ell}),\Cech(D_Y^{\alpha,\ell})\right) \leqslant 2\alpha \ ,
\]
because one can construct an $2\alpha$-interleaving in the category $\AsAlg$ between $\rmH^*(\Cech(D^{\alpha,\ell}_X))$ and $\rmH^*(\Cech(D^{\alpha,\ell}_Y))$.
The first inequalities follows.

\smallskip

Now, note that the homotopy equivalences~\eqref{eq:htpyequiv} imply that the canonical structure maps 
\[\rmH^*(\Cech(D^{\beta,\ell}_X)) \to \rmH^*(\Cech(D^{\tau,\ell}_X)), \quad \mbox{and} \quad
\rmH^*(\Cech(D^{\beta,\ell}_Y)) \to \rmH^*(\Cech(D^{\tau,\ell}_Y))\] are the identity maps for all $\alpha\leqslant \tau\leqslant \beta \leqslant \ell-\alpha$.
Suppose that there is $\eps < \frac{\ell-2\alpha}{2}$ such that there exists a $\eps$-interleaving in the category $\infty\textsc{-}\Alg_{\calA_\infty}$ between $\rmH^*(\Cech(D^{\alpha,\ell}_X))$ and $\rmH^*(\Cech(D^{\alpha,\ell}_Y))$ and write $\mu:\rmH^*(\Cech(D^\alpha_X))_\alpha \to \rmH^*(\Cech(D^\alpha_Y))_{\alpha+\eps}$, 
$u:\rmH^*(\Cech(D^\alpha_Y))_\alpha \to \rmH^*(\Cech(D^\alpha_X))_{\alpha+\eps} $ for the interleaving maps. Since those are persistence modules maps ant the structure maps are the identity in the range $[\alpha, \ell- \alpha]$, we get that for $\eps < \frac{\ell-2\alpha}{2}$, we have that $\mu[\eps] =\mu$ and $\nu[\eps]=\nu$.
 Therefore we  have the following diagrams (of $\calA_\infty$-algebras)
\[
	\begin{tikzcd}
		& \rmH^*(\Cech(D^\alpha_Y))_{\alpha+\eps} \arrow[rd,squiggly,"\nu"]& \\
		\rmH^*(\Cech(D^\alpha_X))_\alpha  \arrow[rr,"="'] \arrow[ru,squiggly, "\mu"] && \rmH^*(\Cech(D^\alpha_X))_{\alpha+2\eps}
	\end{tikzcd} \ ,\]\[ \begin{tikzcd}
		& \rmH^*(\Cech(D^\alpha_X))_{\alpha+\eps} \arrow[rd,squiggly,"\mu"]& \\
		\rmH^*(\Cech(D^\alpha_Y))_\alpha  \arrow[rr,"="'] \arrow[ru,squiggly, "\nu"] && \rmH^*(\Cech(D^\alpha_Y))_{\alpha+2\eps} \ .
	\end{tikzcd}
\]
that imply that $\calA_\infty$-morphisms $\mu$ and $\nu$ are $\infty$-isomorphims (See~\Cref{def:Ainfty} and~\ref{def:Ainftymaps} for the notations). Indeed, this diagram imposes that the underlying linear morphisms $\mu_1$ and $\nu_1$ are isomorphisms of graded vector spaces which implies the result by~\Cref{def:Ainftymaps}.  
This is a contradiction because $	\rmH^*(\Cech(D^\alpha_X))_\alpha \cong \rmH^*(S^3\setminus (B_3)_\ell)$ and $	\rmH^*(\Cech(D^\alpha_Y))_{\alpha+\eps}\rmH^*(S^3\setminus (U_3)_\ell)$ are not isomorphic as $\calA_\infty$-algebras.  Hence
\[
	\d_{\As}\left( \Cech(D_X^\alpha),\Cech(D_Y^\alpha)\right)\geqslant \frac{\ell-2\alpha}{2} \ .
\]
On the other hand for $\eps \geq \frac{\ell +\alpha}{2}$, one has a trivial interleaving since $\rmH^*(\Cech(D^\alpha_Y))_{t+2\eps}=0=\rmH^*(\Cech(D^\alpha_X))_{t+2\eps}$ in degrees $1,2$ for $t\in [0, \ell +\alpha]$. This gives the last inequality.
\end{proof}
\end{example}

\begin{example}[Concrete computations for \Cref{ex:borromean}] \label{ex:borromean_compute}
Following the strategy of \Cref{ex:torus_sphere} and \Cref{ex:torus_sphere_comput}, we use \Cref{ex:Borromeandiscret} to construct two discretes spaces $\widetilde{X}$ and $\widetilde{Y}$ such that $\d_{\As}(\Alpha(\widetilde{X})_\bullet,\Alpha(\widetilde{Y})_\bullet)$ is small compared to $\d_{\calA_\infty}(\Alpha(\widetilde{X})_\bullet,\Alpha(\widetilde{Y})_\bullet)$, where $\Alpha(D)$ denote the alpha complex of the discrete space $D$ (see \Cref{not:alphacomplex}).

For the rest of this example, we fix the radius $R >0$  of our sphere $S^3$ centered in $0$. We denote by 
\[
	\begin{array}{rccc}
		\Phi \colon &  \R^3 & \longrightarrow & \R^4 \\
		& u = (x,y,z) & \longmapsto	& \left( \frac{2Rx}{1+\|u\|_{2}}, \frac{2Ry}{1+\|u\|_{2}}, \frac{2Rz}{1+\|u\|_{2}},\frac{R(\|u\|_{2}-1)}{1+\|u\|_{2}}  \right)
	\end{array}
\]
the stereographic projection from $\R^3$ to $S^3$. Let $n$ be an even integer $n\in 2\N^*$, we consider the following discretisation of the $S^3$ sphere with radius $R$
\[
	\widetilde{S^3} \coloneqq 
	\left\{
		{\scriptstyle
		\left(R\cos(\frac{i\pi}{n})\sin(\frac{k\pi}{n}),
		R\sin(\frac{i\pi}{n})\sin(\frac{k\pi}{n}),
		R\cos(\frac{j\pi}{n})\cos(\frac{k\pi}{n}),
		R\cos(\frac{j\pi}{n})\cos(\frac{k\pi}{n})\right)
		}
		~\big|~ i,j \in \llbracket 0,2n\rrbracket, \ k\in \llbracket -\frac{n}{2},\frac{n}{2}\rrbracket
	\right\} \ .
\]
Now, we fix two real numbers $0<r_1 <r_2$ and an integer $N\in\N^*$ with $n \ll N$. We consider the discretisation of the borromean rings in $\R^3$ given by 
\[
	\widetilde{B_3} \coloneqq
	\left\{
		{\scriptstyle (0,r_1\cos(\frac{i\pi}{N}),r_2\sin(\frac{i\pi}{N}))} 
		|i\in \llbracket 0,2N \rrbracket
	\right\}
	\sqcup 
	\left\{
		{\scriptstyle{}(r_2\cos(\frac{i\pi}{N}),0,r_1\sin(\frac{i\pi}{N}))}
		|i\in \llbracket 0,2N \rrbracket
		\right\}
	\sqcup 
	\left\{
		{\scriptstyle{}(r_1\cos(\frac{i\pi}{N}),r_2\sin(\frac{i\pi}{N}),0)}
		|i\in \llbracket 0,2N \rrbracket
	\right\} \ ,
\]
and the discretisation of trivial entanglement in $\R^3$
\[
	\widetilde{U_3} \coloneqq 
	\left\{
		{\scriptstyle (0,r_1\cos(\frac{i\pi}{N}),r_2\sin(\frac{i\pi}{N}))} 
		|i\in \llbracket 0,2N \rrbracket
	\right\}
	\sqcup 
	\left\{
		{\scriptstyle (r_1\cos(\frac{i\pi}{N}),0,r_1\sin(\frac{i\pi}{N}))}
		|i\in \llbracket 0,2N \rrbracket
		\right\}
	\sqcup 
	\left\{
		{\scriptstyle(r_2\cos(\frac{i\pi}{N}),r_2\sin(\frac{i\pi}{N}),0)}
		|i\in \llbracket 0,2N \rrbracket
	\right\} \ .
\]	
We choose a positive real number $0<L<\frac{R}{2}$ such that 
\[
	\bigcup_{x\in \Phi(\widetilde{B_3})} \rmB(x,L) \simeq 
	S^1 \amalg S^1 \amalg S^1 \simeq 
	\bigcup_{x\in \Phi(\widetilde{U_3})} \rmB(x,L)
\]
which corresponds to the radius of the thickening of the rings embedded in $\R^4$, and depends 	of $n,N,R$.

We define the discrete spaces
\[
	\widetilde{X} \coloneqq \widetilde{S^3} \setminus \left(\bigcup_{x \in \Phi(\widetilde{B_3})} B(x,L)\right) 
	\qquad \mbox{ and } \qquad 
	\widetilde{Y} \coloneqq  \widetilde{S^3} \setminus \left(\bigcup_{x \in \Phi(\widetilde{U_3})} B(x,L)\right) \ .
\]

Using the librairy \texttt{Gudhi}, we compute the barcode and the persistent diagrams of the alpha complexes of $\widetilde{X}$ and $\widetilde{Y}$ ; see \Cref{fig:barcode_B3}, \ref{fig:barcode_U3}, \ref{fig:persistent_B3} and \ref{fig:persistent_U3}, where the computation is maked coefficient in $\mathbb{F}_{11}$ and with the following choice of parameters: $n=40$, $R=20$, $r_1 = 0.7$, $r_2=3.2$, $N=100$ and $L=5$ (for the barcode diagrams, one just show the twenty bigger bars). We also compute the bottleneck distance between persistent diagrams in each degree, which gives us :
\[
	\begin{aligned}
	\mathrm{d}_{\mathrm{bottle}}(\rmH^0(\Alpha(\widetilde{X})), \rmH^0(\Alpha(\widetilde{Y}))) & \approx 0.048 \\
	\mathrm{d}_{\mathrm{bottle}}(\rmH^1(\Alpha(\widetilde{X})), \rmH^1(\Alpha(\widetilde{Y}))) & \approx 1.188 \\
	\mathrm{d}_{\mathrm{bottle}}(\rmH^2(\Alpha(\widetilde{X})), \rmH^2(\Alpha(\widetilde{Y}))) & \approx 0.985 \ .
	\end{aligned} 
\] 
One also compute that, for all $1.187\leqslant \eps \leqslant 23.768$, we have 
\[
	\rmH(\Alpha(\widetilde(X))_\eps) \cong \rmH(S^3\setminus B_3) \ ,
\]
and for all $1.086\leqslant \eps \leqslant 23.889$, we have 
\[
	\rmH(\Alpha(\widetilde(Y))_\eps) \cong \rmH(S^3\setminus U_3) \ .
\]
Now, similarly to~\Cref{ex:torus_sphere_comput} and~\Cref{remark:computation_multiplicative_distance}, we can  
give a rather precise estimate of the $\calA_\infty$-distance. Indeed recall that $\rmH^*(S^3 \setminus U_3)$ and $\rmH^*(S^3 \setminus B_3)$ are not isomorphic as $\calA_\infty$-algebras  since the $m_3$ products (in the $\calA_\infty$-structure between the classes corresponding to the three significative bars of degree $1$ is zero in the case corresponding to the trivial entanglement but is non-zero for the Borromean one and the same holds for their discretisation. 

Letting  $\delta =\min_{i}(t^i_1) -\max_{i}(t^i_0)$ be the difference  between the higher starting point of the dimension $1$ and $2$ long bars of $\rmH^*(\Alpha(\tilde{X}))$ and $\rmH^*(\Alpha(\tilde{Y}))$ and  the smaller endpoint corresponding to those long bars in dimension $1$ and $2$. Then there could be  no algebra $\frac{\delta}{2}$-interleaving in between $\rmH^*(\Alpha(\tilde{X}))$ and $\rmH^*(\Alpha(\tilde{Y}))$ because we would have a commutative diagram
\begin{equation}\label{eq:diagBorrocomp}
	\begin{tikzcd}
		& \rmH^1(\Alpha(\tilde{X}))_{\max_{i}(t_0^i)}^{\otimes 3} \arrow[r,"m_3"]& \rmH^2(\Alpha(\tilde{Y}))_{\min_{i}(t_1^i)} \arrow[rd, "\nu"]&\\
		\rmH^1(\Alpha(\tilde{Y}))_{\max_{i}(t_0^i)}^{\otimes 3}  \arrow[rr,"\cong"] \arrow[ru,"\mu"] && \rmH^1(\Alpha(\tilde{Y}))_{\min_{i}(t_1^i)}^{\otimes 3} \arrow[r, "m_3"] & \rmH^2(\Alpha(\tilde{Y}))_{\max_{i}(t_0^i)}
	\end{tikzcd} \ ,
\end{equation}
 for which the lower line will be non-zero while the top compositon is zero.

On the other hand since the smallest difference  in between the smallest ending point of the three significative degree 1 bars of $\rmH^1(\Alpha(\tilde{X}))$ and the highest starting points of those bars is  $\delta':=23.219$,  we obtain a trivial $\calA_\infty$ $\frac{\delta'}{2}$-interleaving  (since $\rmH^1(\Alpha(\tilde{X}))_{t+\delta'}$ has a trivial $\calA_\infty$-structure).  Therefore we have 
\[
	\frac{23.768-1.187}{2} \approx 11.29 \leqslant \mathrm{d}_{\calA_\infty}(\rmH^*(\calA(\widetilde{X})), \rmH^*(\calA(\widetilde{Y})))  \leqslant \frac{23.219}{2} \approx 11.609 \ .
\]

\begin{figure}[h!]
	\begin{minipage}[c]{.46\linewidth}
	\includegraphics[scale=0.45]{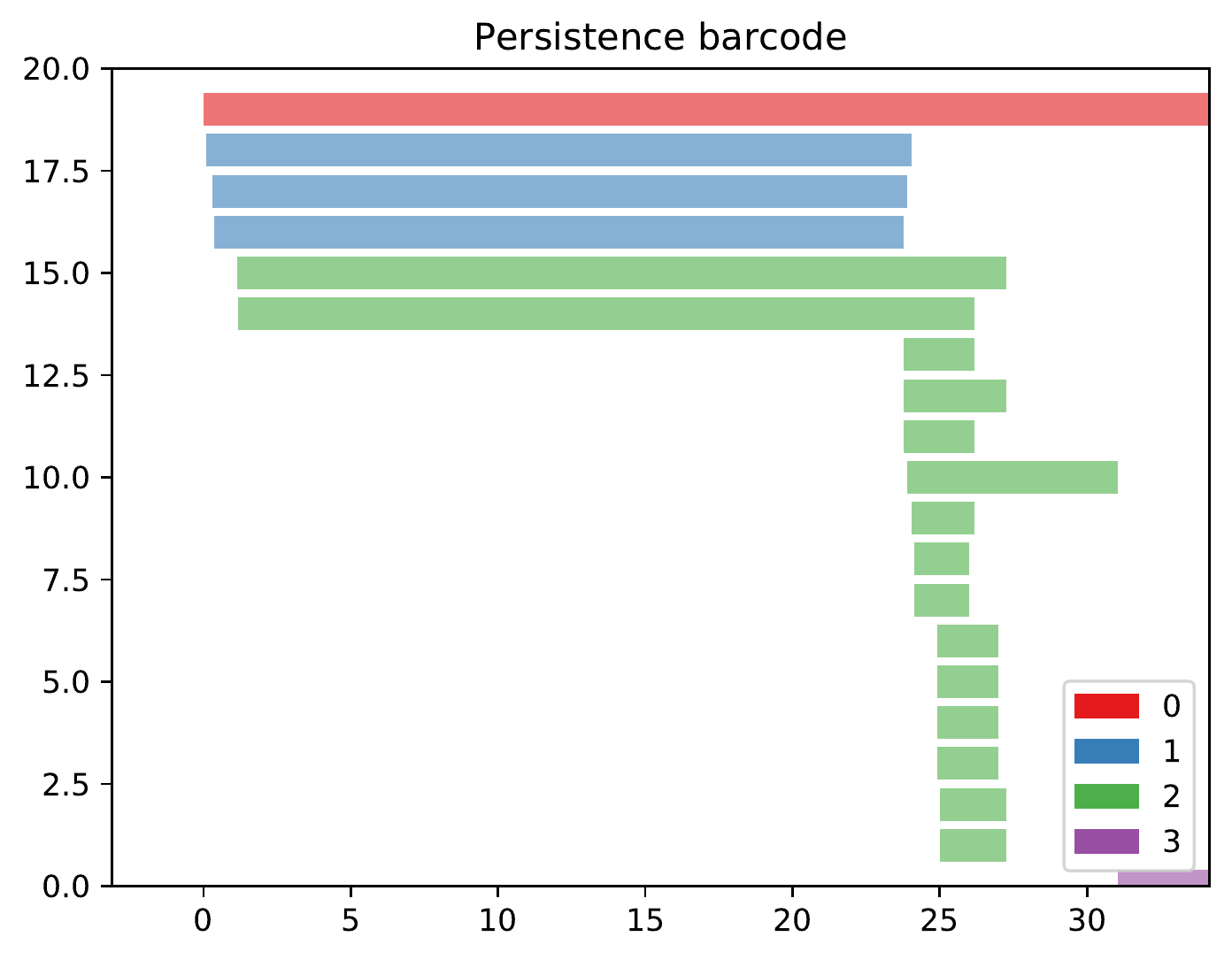}
	\caption{Persistent barcode of $\Alpha(\widetilde{X})$}
	\label{fig:barcode_B3}
\end{minipage}
\begin{minipage}[c]{.46\linewidth}
	\includegraphics[scale=0.45]{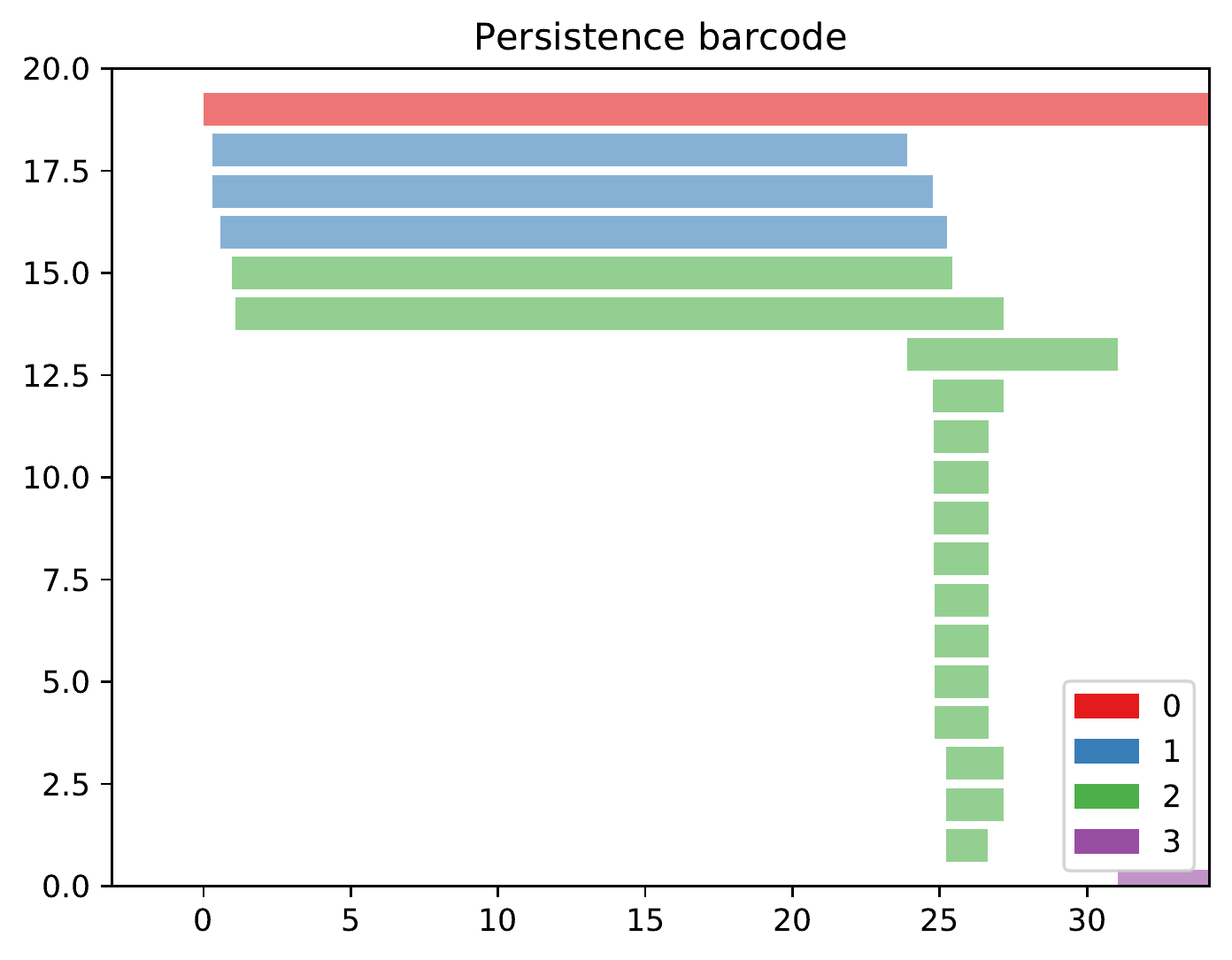}
	\caption{Persistent barcode of $\Alpha(\widetilde{Y})$}
	\label{fig:barcode_U3} 
\end{minipage}
\end{figure}

\begin{figure}[h!]
	\begin{minipage}[c]{.46\linewidth}
	\includegraphics[scale=0.45]{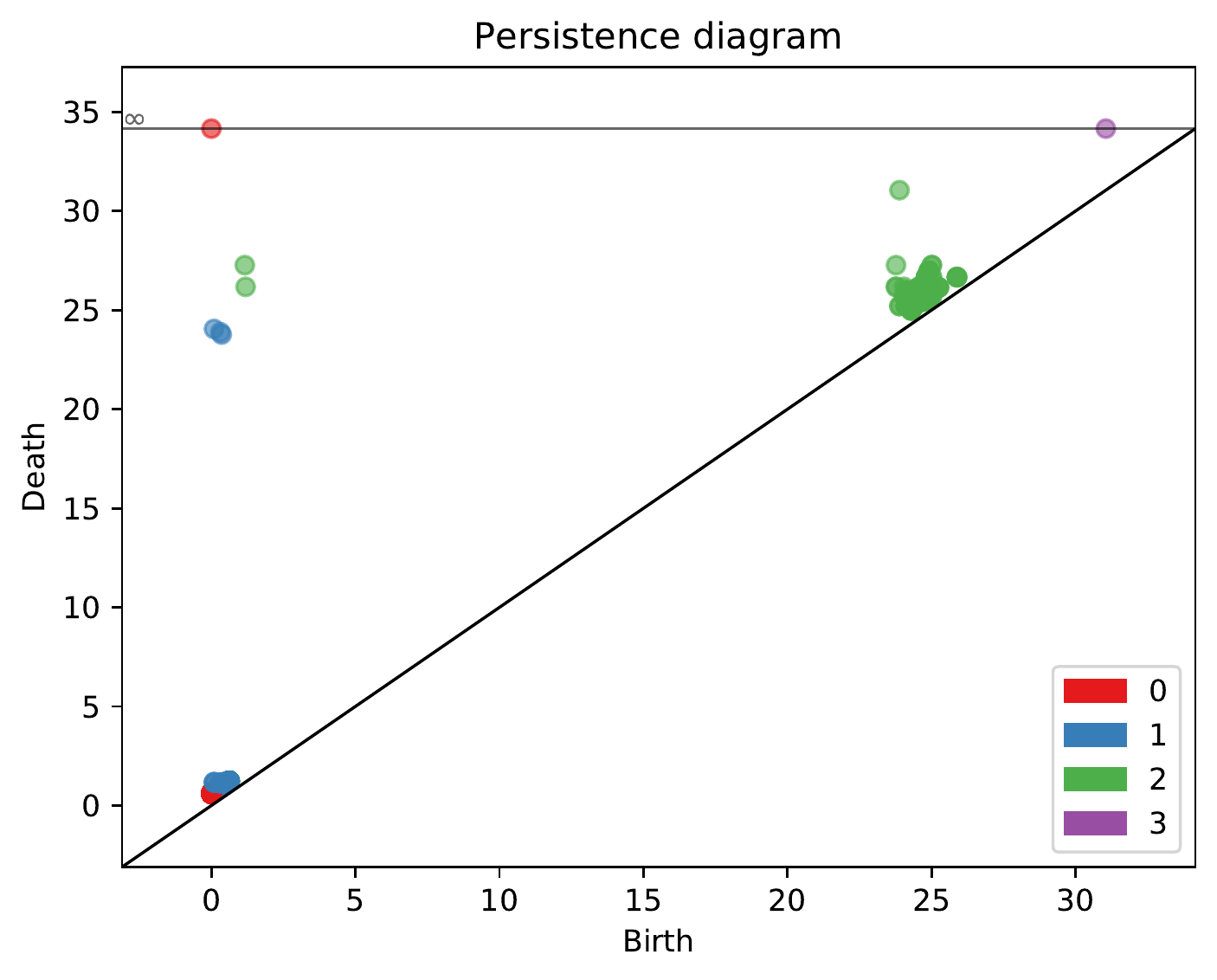}
	\caption{Persistent diagram of $\Alpha(\widetilde{X})$}
	\label{fig:persistent_B3}
\end{minipage}
\begin{minipage}[c]{.46\linewidth}
	\includegraphics[scale=0.45]{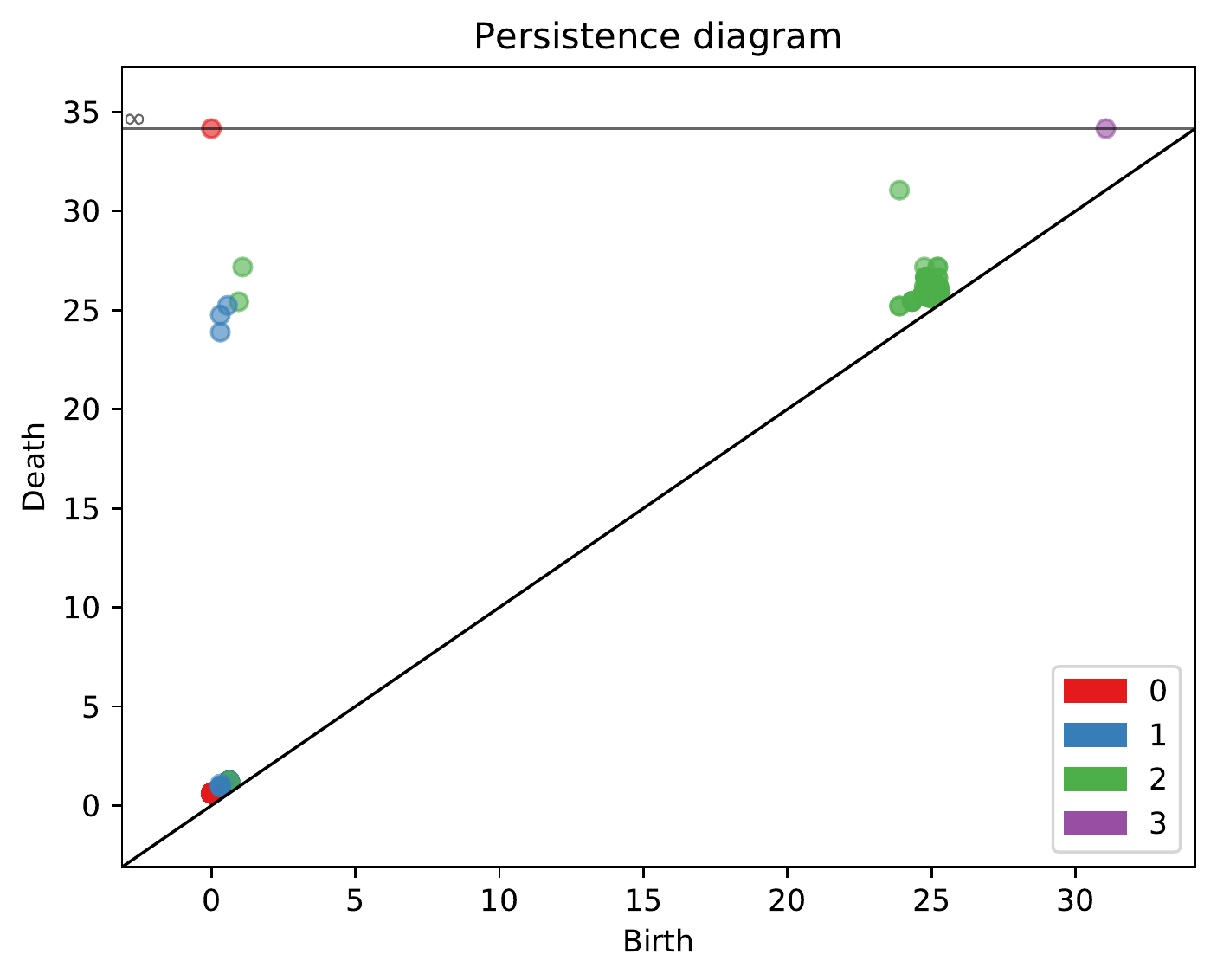}
	\caption{Persistent diagram of $\Alpha(\widetilde{Y})$}
	\label{fig:persistent_U3} 
\end{minipage}
\end{figure}
\end{example}
\begin{remark}[About computation of the  $\calA_\infty$-distance] \label{remark:computation_Ainfty_distance} 
One can give lower bounds of the $\calA_\infty$-distance in the same way 
as for the graded algebra ones in~\Cref{remark:computation_multiplicative_distance}.

Indeed, 
let $A^*,\, B^*$ be pointwise finite dimensional copersistent  $\calA_\infty$-algebras over the field $k=\mathbb{F}_2$.  Then every bar in the barcode correspond to (co)homology classes, and given $n$-many  of them, say $a_1,\dots, a_n$, we can compute their $m_n$-$\calA_\infty$-product  which is a sum $\sum_{k\in I} c_k \in A^{i_1+\dots i_n +2-n}$ where the $c_k$ are cohomology classes associated to bars of the graded barcode of $A^*$. Therefore, if one has an $\eps$-$\calA_\infty$-interleaving such that  the $a_i$'s are $\eps$-matched to bars $b_i\in B^*$, then, denoting $m_n(b_1, \dots,,b_n)= \sum_{\ell\in J} d_j$,  from diagrams similar to~\eqref{eq:diagBorrocomp}, we get that   
the bars $(c_k)_{k\in I}$ are $\eps$-interleaved (with the cohomology classes representing the) bars $(d_\ell)_{\ell\in J}$ as well as $\eps$-matched with them.

Thus, in order to give an effective lower bound on the $\calA_\infty$-distance, on can look for  $\eps$-matching between the positive degrees  bars of $A^*$ and $B^*$ such that for any matched pairs $(a_1, b_1)$,..., $(a_n, b_n)$, those  $\eps$-matching restricts to $\eps$-matching in between  the corresponding non zero bars $(c_k)$ and $(d_k)$ appearing in  $m_n(a_1, \dots, a_n)= \sum c_k$, and $m_n(b_1,\dots, b_n)= \sum d_k$.  Then the lower such $\eps$ is a lower bound for the $\calA_\infty$-distance. By finiteness, this is only a finite number of steps. 
\end{remark}

\section{A  stability theorem for multiplicative distances}

Consider $X$ and $Y$, two compact topological spaces. Assume $X$ and $Y$ are almost the same,  i.e. the Gromov-Hausdorff distance between $X$ and $Y$ is small (for example, if $Y$ is a small perturbation of $X$), then, we want  the interleaving distances between $\rmC^*(X)$ and $\rmC^*(Y)$ or $\rmH^*(X)$ and $\rmH^*(Y)$ (defined in \Cref{subsect:distance_As} and \Cref{subsec:distance_Ainfty}) to be small as well. We obtain this result (see \Cref{thm:stabilite}) by adapting the proof of its classical version (cf. \cite{chazal2014stability}). Note that in \cite{bubenik2017interleaving}, Bubenik \emph{et al.} have given a general framework to show stability theorems.

\subsection{Gromov-Hausdorff distance}
We first review the basic definitions of Gromov-Hausdorff distance. 

\begin{definition}[Multivalued map -- Correspondence]
	Let $X$ and $Y$ be two sets. A \emph{multivalued map} from $X$ to $Y$ is a subset $C$ of $X\times Y$ such that the canonical projection restrict to $C$ $\pi_X|_C:C \rightarrow X$ is surjective.  We denote a multivalued map $C$ from $X$ to $Y$ by $C:X \rightrightarrows Y$. 	
	The \emph{image} $C(\sigma)$ of a subset $\sigma$ of $X$ is the canonical  projection onto $Y$ of the preimage of $\sigma$ through $\pi_X$.
	A map $f:X\rightarrow Y$ is \emph{surbordinate} to $C$ if, for all $x$ in $X$, the pair $(x,f(x))$ is in $C$. In that case, we write $f:X \overset{C}{\rightarrow} Y$.	
	The \emph{composition} of two multivalued maps $C:X \rightrightarrows Y$ and $D:Y\rightrightarrows Z$ is the multivalued map $D\circ C :X\rightrightarrows Z$, defined by:
	\[
		(x,y) \in D\circ C \Longleftrightarrow \mbox{ there exists } y\in Y  \mbox{ such that } (x,y)\in C  \mbox{ and } (y,z)\in D.
	\]
	A multivalued map $C:X\rightrightarrows Y$ such that the canonical projection $C$ $\pi_Y|_C:C \rightarrow Y$ is surjective, is called \emph{a correspondence}. The \emph{transpose} of a correspondence $C$, denoted $C^T$, is the correspondence defined by the image of $C$ through the symmetry $(x,y)\mapsto (y,x)$.	
\end{definition}

\begin{remark}\label{rem:idetite_subordonnee}
	Consider a correspondence $C:X\rightrightarrows Y$. Then we have
	\[
		\begin{tikzcd}
				\id_X : X \arrow[r,"C^T\circ C"] &  X
		\end{tikzcd}
		\quad \mbox{and} \quad
		\begin{tikzcd}
			\id_Y : Y \arrow[r,"C\circ C^T"] &  Y
		\end{tikzcd} \ .
	\]
\end{remark}
To a correspondence $C:X \rightrightarrows Y$, we associate a quantity called the distortion metric, and we define the Gromov Hausdorff distance.

\begin{definition}[Distortion of a correspondence -- Gromov-Hausdorff distance]\label{def::metric_distortion}
	Let $(X,\d_X)$ and $(Y,\d_Y)$ be two metric spaces. The \emph{distortion} of a correspondence $C:X\rightrightarrows Y$ is defined as follows:
	\[
		\dist_m(C)\coloneqq \underset{(x,y),(x',y')\in C}{\sup} \ |\d_X(x,x')-\d_Y(y,y')|.
	\]
	The \emph{Gromov-Hausdorff  distance} between the metric spaces $X$ and $Y$ is defined as follows:
	\[
		\dGH(X,Y)\coloneqq \dfrac{1}{2}\underset{C:X\rightrightarrows Y}{\inf} \ \dist_m(C).
	\]
\end{definition}

\begin{remark}
	Let $X$ and $Y$ be two metric spaces. The Gromov-Hausdorff distance between $X$ and $Y$ is equal to the following one:
	\[
		\dGH(X,Y)= \inf_{(\gamma,\eta)\in \Gamma} \min\{\eps\geqslant 0 \ | \ \gamma(X)\subset \eta(Y)_\eps \ \mbox{and} \ \eta(X)\subset \gamma(X)_\eps \}
	\]
	where $\Gamma=\{ X \overset{\gamma}{\rightarrow} Z \overset{\eta}{\leftarrow} Y\ | \  \mbox{$(Z,d_Z)$ metric space and $\gamma$ and $\eta$ are isometrical embeddings} \}$, and 
	\[
	(\gamma(X))_\eps \coloneqq \bigcup_{x \in  \gamma(X)} B_Z(x,\eps)	.
	\]
\end{remark}

\subsection{Simplicial multivalued map} 

\begin{definition}[$\eps$-simplicial multivalued map]
	Let $\calS$ and $\calT$ be two persistent  delta complexes such that, for all $r$ in $\sfR$, the vertex sets of $\calS_r$ and $\calT$ are $X$ and $Y$ respectively. A multivalued map $C:X\rightrightarrows Y$ is \emph{$\eps$-simplicial for $\calS$ and $\calT$}  if, for any $r$ in $\sfR$ and any simplex $\sigma$ in $\calS_r$, every finite subset of $C(\sigma)$ is a simplex of $\calT_{r+\eps}$.
\end{definition}

\begin{definition}[Contiguous maps]
	Let $K$ and $L$ be two simplicial complexes. Two simplicial maps $f,g :K\rightarrow L$ are \emph{contiguous} if, for each simplex $v_0,\ldots,v_n$ of $K$, the points 
	\[
	f(v_0),\ldots,f(v_n),g(v_0),\ldots,g(v_n)
	\]
	span a simplex $\tau$ of $L$.
\end{definition}

\begin{lemma}[{\cite[Proposition 3.3]{chazal2014stability}}] \label{lemma:subordinate_are_homotopic}
	Let $\calS,\calT:\sfR \rightarrow \Delta\Cpx$ with vertex sets $X$ and $Y$ respectively and let $C:X\rightrightarrows Y$ be a $\eps$-simplicial multivalued map from $\calS$ to $\calT$. Then, any two subordinate maps $f_1,f_2:X \overset{C}{\rightarrow} Y$ induce simplicial maps $\calS_a\rightarrow \calT_{a+\eps}$ which are contiguous.
	Also, the maps $|f_1|$ and $|f_2|$ are homotopic.
\end{lemma}

\begin{proof}
	Let $\calS,\calT:\sfR \rightarrow \Delta\Cpx$ with vertex sets $X$ and $Y$ respectively and let $C:X\rightrightarrows Y$ be a $\eps$-simplicial multivalued map from $\calS$ to $\calT$. Any subordinate map $f:X \overset{C}{\rightarrow} Y$ induces a simplicial map $\calS_r \to \calT_{r+\eps}$ for each $r$ in $\sfR$ by definition of an $\eps$-simplicial multivalued map.
	
	Consider two subordinate maps $f_1,f_2:X \overset{C}{\rightarrow} Y$, and let $\sigma=[v_0,\ldots,v_n]$ be a simplex in $\calS_r$. As $C$ is a $\eps$-simplicial multivalued map, then, by definition, every subset of $C(\sigma)$ is a simplicial set of $\calT_{r+\eps}$. Therefore, since every $f_i(v_j)$ is in $C(\sigma)$,  the set  $f_1(v_0),\ldots,f_1(v_n),f_2(v_0),\ldots,f_2(v_n)$ is a simplex of $\calT_{r+\eps}$. Consequently, the simplical maps induced by $f_1$ and $f_2$ are contiguous. By \cite[Proposition 10.20]{mccleary06}, contiguous simplicial maps have homotopic realisations; hence, the realisations of $f_1$ and $f_2$ are homotopic.
\end{proof}

\begin{proposition}[{\cite[Chapter 16, Theorem 3.8.]{james1995handbook}}]\label{prop:homotopy_of_dga} 
	The functor $\rmC^*:\Top^\op\rightarrow \AsAlg$ converts weak homotopy equivalences to quasi-isomorphisms and homotopy classes of maps to homotopy classes of morphisms of dg-associative algebras.
\end{proposition}

\subsection{Case of \v{C}ech and Vietoris-Rips complexes}
In this section, we prove the stability of the distances that we introduced before for the  \v{C}ech and Vietoris-Rips complexes.

\begin{lemma}[{\cite[Lemmas 4.3 and 4.4]{chazal2014stability}}]\label{lem:C_eps_simplicial}
	Let $(X,\d_X)$ and $(Y,\d_Y)$ be two metric spaces, and let $C:X \rightrightarrows Y$ be a correspondence with distortion at most $\eps$. Then
	\begin{enumerate}
		\item the correspondence $C$ is $\eps$-simplicial from $\calR(X,\d_X)$ to $\calR(Y,\d_Y)$\ ;
		\item the correspondence $C$ is $\eps$-simplicial from $\Cech(X,\d_X)$ to $\Cech(Y,\d_Y)$\ .
	\end{enumerate}	
\end{lemma}

\begin{proof} Let $C:X \rightrightarrows Y$ be a correspondence with distortion at most $\eps$.
	\begin{enumerate}
		\item If $\sigma$ is a simplex of $\calR(X,\d_X)_r$, then $\d_X(x,x')\leqslant r$ for all $x,x'$ in $\sigma$. Let $\tau$ be any subset of $C(\sigma)$: for any $y,y'$ in $\tau$, there exist $x$ and $x'$ in $\sigma$ such that $y\in  C(x)$ and $y'\in C(x')$, and therefore:
		\[
			\d_Y(y,y') \leqslant \d_X(x,x') +\eps \leqslant r+\eps \ .
		\]
		Therefore $\tau$ is a simplex of $\calR(Y,\d_Y)_{r+\eps}$. We have thus shown that $C$ is $\eps$-simplicial from $\calR(X,\d_X)$ to $\calR(Y,\d_Y)$.
		
		\item Let $\sigma$ be a simplex of $\Cech(X,\d_X)_r$, and let $\bar{x}$ be an $r$-center of $\sigma$, so, for all $x$ in $\sigma$, we have  $\d_X(x,\bar{x})\leqslant r$. Take an element $\bar{y}$ in $C(\bar{x})$. For any $y$ in $C(\sigma)$, we have $y$ in $C(x)$ for some $x$ in $\sigma$, and therefore
		\[
			\d_Y(\bar{y},y) \leqslant \d_X(\bar{x},x) +\eps \leqslant r+\eps \ ,
		\]
		Let $\tau$ be a subset of $C(\sigma)$; $\bar{y}$ is an $(r+\eps)$-center for $\tau$ and hence $\tau$ is a simplex of $\calC(Y,\d_Y)_{r+\eps}$. We have thus shown that $C$ is $\eps$-simplicial from $\Cech(X,\d_X)$ to $\Cech(Y,\d_Y)$.
	\end{enumerate}
\end{proof}

\begin{theorem}[Stability theorem - Associative version]\label{thm:stabilite}
	Let $(X,\d_X)$ and $(Y,\d_Y)$ be two metric spaces. We have the following inequalities:
	\begin{enumerate}
		\item in the homotopy category $\ho(\AsAlg)$:
			\begin{align*}
				\d_{\ho(\AsAlg)}(\rmC^*(\calR(X,\d_X)),\rmC^*(\calR(Y,\d_Y))) \leqslant 2\dGH\left((X,\d_X),(Y,\d_Y)\right) \ ; \\
				\d_{\ho(\AsAlg)}(\rmC^*(\Cech(X,\d_X)),\rmC^*(\Cech(Y,\d_Y))) \leqslant 2\dGH\left((X,\d_X),(Y,\d_Y)\right) \ ;
			\end{align*}
			
		\item in the strict category $\AsAlg$ :
		\begin{align*}
			\d_{\As}(\calR(X,\d_X),\calR(Y,\d_Y)) \leqslant 2\dGH\left((X,\d_X),(Y,\d_Y)\right) \ ; \\
			\d_{\As}(\Cech(X,\d_X),\Cech(Y,\d_Y)) \leqslant 2\dGH\left((X,\d_X),(Y,\d_Y)\right) \ .
		\end{align*}
	\end{enumerate}
\end{theorem}

Recall (\Cref{rem:notation_distance}) that in the inequalities (2), we are considering the associative algebra structures given by the cup product on the \emph{cohomology} algebras $\rmH^*(\calR(X,\d_X))$, $\rmH^*(\calR(Y,\d_Y))$ (and not algebra structures at the cochain level).

\begin{proof}
	Let $X$ and $Y$ be two metric spaces and let $C:X \rightrightarrows Y$ be a correspondence with distortion at most $\eps$. By \Cref{lem:C_eps_simplicial},  $C$ is $\eps$-simplicial from  $\calR(X,\d_X)$ to $\calR(Y,\d_Y)$. Thus, by \Cref{lemma:subordinate_are_homotopic},  any two subordinate maps $f,g: X \overset{C}{\to} Y$ induce $\eps$-morphisms of copersistent differential graded algebras $\rmC^*(\calR(X,\d_X)) \to \rmC^*(\calR(Y,\d_Y) )$ which are homotopic in the category $\AsAlg$. Therefor, the correspondence $C$ induces a $\eps$-morphism $\phi:\rmC^*(\calR(X,\d_X)) \to \rmC^*(\calR(Y,\d_Y) )$ in the homotopy category $\ho(\AsAlg)$ thanks to \Cref{prop:homotopy_of_dga}.

	By the same argument, the correspondence $C^T:Y \rightrightarrows X$ gives us an  $\eps$-morphism $\psi : \rmC^*(\calR(Y,\d_Y)) \to \rmC^*(\calR(X,\d_X) )$. The correspondence $C^T\circ C$ (respectively $C\circ C^T$) gives us the $2\eps$-morphism $\psi\circ\phi$ (resp. $\phi\circ\psi$), which is the canonical $2\eps$-endomorphism of $\rmC^*(\calR(X,\d_X))$ (resp. $\rmC^*(\calR(X,\d_X))$) given by the structural morphisms of $\rmC^*(\calR(X,\d_X))$ (resp. $\rmC^*(\calR(X,\d_X))$). Therefore the $\eps$-morphisms $\phi$ and $\psi$ define a $\eps$-interleaving between $\rmC^*(\calR(X,\d_X))$  and $\rmC^*(\calR(Y,\d_Y))$ . 
	
	The same argument applies \emph{verbatim} for all other cases.
\end{proof}

\begin{corollary}[Stability theorem - $\calA_\infty$ version]
	Let $X$ and $Y$ be two finite set of points of $\R^n$. We have the following inequalities:
	\begin{align*}
	\d_{\calA_{\infty}}(\calR(X),\calR(Y)) \leqslant 2\dGH\left(X,Y\right) \ ; \\
	\d_{\calA_{\infty}}(\Cech(X),\Cech(Y)) \leqslant 2\dGH\left(X,Y\right) \ .
	\end{align*}
\end{corollary}

\begin{proof}
    It follows from \Cref{prop:inequalities_Ainfty} (2) and \Cref{thm:stabilite}.
\end{proof}

\section{Homotopy commutativity preserving distances}

In this section we also consider the homotopy \emph{commutativity} of the cup-product in cohomology. The cochain algebra is not commutative on the nose though the cohomology is commutative (in the graded sense). Indeed, there are structures (encoded in higher homotopies) yielding the commutativity of the cup-product after passing to cohomology. Thus  commutativity is \emph{additional structures} on the cochains (or their cohomology) which can be used to distinguish homotopy types. We start by studying distance associated to the most homotopical additional structure in \Cref{sec:Einfty} before moving to more tractable but   useful ones in \Cref{sec:Steenrod_distance}. 

In this section, whenever needed, we adopt the operadic language (see \cite[Section 5.2]{LV2012} for the definition of an (algebraic symmetric) operad and the definition of an algebra over an operad). However, we try to make the statement and constructions understandable without  knowledge of operadic methods as much as we can.

\subsection{A theorical construction: $\calE_\infty$-structures and $\calE_\infty$-distances}\label{sec:Einfty}
In this section, we introduce a new distance, based on the $\calE_\infty$-algebra structure of  cochain complexes, which dominate all the other ones one can build on persistence cohomology associated to a space (or data set); see \Cref{rem:mandell}. 
The $\calE_\infty$-algebra structures are (differential graded) \emph{homotopy  commutative} and associative structures which are \emph{functorially} carried by cochain complexes associated to spaces or more generally simplicial sets or complexes. We first need the following standard construction.

\begin{definition}[Normalized (co)chain complex]\label{def:normalizedcomplex}
	Let $X$ be a simplicial set. The \emph{normalized chain complex} $\rmN_*(X)$ is the quotient of the dg-module $\rmC_*(X)$ by the degeneracies:
	\[
	\rmN_d(X)=\dfrac{\rmC_d(X)}{s_0\rmC_{d-1}(X)+\ldots s_{d-1}\rmC_{d-1}(X)}.
	\]
	We consider also the dual cochain complex $\rmN^*(X)=\Hom_{k}(\rmN_*(X),k)$.
	If $Y$ is a topological space, we define $\rmN_*(Y)\coloneqq  \rmN_*(\Sing_\bullet(Y))$ to be the normalized complex of the singular simplicial set associated to $Y$ (see \Cref{sec:CechRips}).
\end{definition}

\begin{remark}
	The normalized (co)chain complexes $\rmN^*$ and standard simplicial (co)chain complexes $\rmC^*$ functors are canonically quasi-isomorphic~\cite{weibelhomological}. In particular, for any topological space the canonical map 	$\rmN^*(Y) \to \rmC^*(Y)$ to the singular cochains of $Y$ is a natural quasi-isomorphism, i.e., induces an isomorphism in cohomology.
	Further, if $X$ is a simplicial complex, then the natural cochain complex associated to $X$ is isomorphic to $\rmN^*(X)$, the normalized complex of the simplicial sets associated to $X$ viewed as a simplicial sets, that is where we have added all degeneracies freely. Said in simpler terms, the normalized cochain complex precisely computes the cochains of a simplicial complex.
\end{remark}

The data of an $\calE_\infty$-algebra involves infinitely many homotopies and there are several equivalent models (meaning models which yields the same homotopy categories of $\calE_\infty$-algebras) for them. We refer to~\cite{Mandell2002, mandell2006cochains} for details on their homotopy theories. A nice  explicit and combinatorial model\footnote{Other popular models are given by the algebras over the Barrat-Eccles operad or the linear isometry operad} for $\calE_\infty$-algebras was given in  \cite{bergerfresse}. It is an explicit operad, called the \emph{surjection operad}, which we denote by $\calE_{\infty}$ in this paper,   which is a cofibrant resolution of the commutative operad $\Com$. The important point is that this operad encodes the structure of associative commutative product up to homotopy and in particular the  category of algebras over the surjection operads models $\calE_\infty$-algebras.

\begin{theorem}[see {\cite[Theorem 2.1.1.]{bergerfresse}}]\label{thm:bergerfresse}
	For any simplicial set $X$, we have evaluation products $\calE_\infty(r) \otimes \rmN^*(X)^{\otimes r} \to \rmN^*(X)$, functorial in $X$, which give the normalized cochain complex $\rmN^*(X)$ the structure of an $\calE_\infty$-algebra. In particular, the classical cup-product of cochains is an operation $\mu_0:\rmN^*(X)^{\otimes 2} \to \rmN^*(X)$  associated to an element $\mu_0$ in $\calE_\infty(2)_0$.
\end{theorem}

\begin{remark}\label{rem:coefficientforcochainsEinfty} The above theorem (and \Cref{def:normalizedcomplex}) are valid with any coefficient ring; in particular over $\mathbb{Z}$, $\mathbb{Q}$ or $\mathbb{F}_p$. 
\end{remark}

\begin{remark} \label{rem:mandell}
By \cite[Main Theorem]{mandell2006cochains}, we know that the $\calE_\infty$-structure on the cochain complex of a topological space $X$ (under some finiteness and nilpotence assumptions) is a faithfull invariant of the homotopy type of $X$ and essentially encodes it.
\end{remark}

\begin{remark}[see {\cite[Section 1.1.1.]{bergerfresse}}]\label{rem:forget_functor}
	As we have a factorisation of operad morphisms
	\[
	\begin{tikzcd} 
	\As \ar[r] &\calE_\infty \ar[r] &\Com \ ,
	\end{tikzcd}
	\]
	we have a forgetful functor 
	$
		\begin{tikzcd} 
		\Alg_{\calE_\infty} \ar[r,"\mathrm{forget}"]&
		\Alg_{\As}\ .
		\end{tikzcd}
	$ 
	Applied to the normalized cochain complex, this  functor recover the cup-product structure, that is, the usual differential graded algebra structure on cochain.
\end{remark}

\begin{definition}\label{def:NEinfty}
	 We denote $\rmN^*_{\calE_\infty}: \Top^{\op}\to \Alg_{\calE_\infty}$ the functor induced by \Cref{thm:bergerfresse} and call it the \emph{cochain $\calE_\infty$-algebra} functor. We denote in the same way its composition with the canonical functor 
	 $\Alg_{\calE_\infty}\to \ho(\Alg_{\calE_\infty}) $.
\end{definition}

As for dg-associative algebras in \Cref{sect:distance_As}, we will only consider the homotopy category of $\calE_\infty$-algebras. 

\begin{remark}\label{rem:forget_functor_Alg}
	The \Cref{rem:forget_functor} implies that the composition of functors 
	\[
		\begin{tikzcd} 
		 \Top^{\op}\ar[r,"\rmN^*_{\calE_\infty}"] &
		 \Alg_{\calE_\infty} \ar[r,"\mathrm{forget}"] & 
		 \Alg_{\As} 
		\end{tikzcd} 
	\] 
	is equal to $\rmN^*_{\As}$.
\end{remark}

The functoriality of the $\calE_\infty$-structure on the normalized cochain complex given by \Cref{thm:bergerfresse} justifies the following refined interleaving distance.

\begin{definition}[$\calE_\infty$-interleaving distance]
Let $X_\bullet,Y_\bullet:\sfR \to \Top$ be  two persistent spaces. The \emph{$\calE_\infty$ interleaving distance} is defined  by 
\[
	\d_{\calE_\infty}(X,Y) \coloneqq \d_{\ho(\Alg_{\calE_\infty })}(\rmN^*_{\calE_\infty}(X),\rmN^*_{\calE_\infty}(Y)) \ 
\]
where the right hand side is the interleaving distance in the \emph{homotopy category of $\calE_\infty$-algebras}.
\end{definition}  

\begin{remark}
The $\calE_\infty$-interleaving distance depends on the choice of the ground field $k$. In particular, as we will see, it behaves very differently in characteristic $0$ than in characteristic $p$. When we need to be explicit on the ground ring, we will use the notation 
 \[ 
 	\d_{\calE_\infty, k}(X,Y) \coloneqq \d_{\ho(\Alg_{\calE_\infty})} (\rmN^*_{\calE_\infty}(X,k),\rmN^*_{\calE_\infty}(Y,k))
 \]
for the  $\calE_\infty$-interleaving distance computed with coefficient in $k$.
\end{remark}

The $\calE_\infty$-interleaving distance is the more refined distance we can put on the persistence cochain complex of a space (or simplicial set). Indeed every other ones we consider are smaller, see~\Cref{thm:resume}. 

\begin{theorem}[Stability theorem - $\calE_\infty$ version] \label{thm:stability_Einfty_version}
	Let $(X,\d_X)$ and $(Y,\d_Y)$ be two metric spaces. We have the following inequalities:
	\[
	\begin{aligned}
	\d_{\calE_\infty }(\calR(X,\d_X),\calR(Y,\d_Y)) \leqslant 2\dGH\left((X,\d_X),(Y,\d_Y)\right)\ ; \\
	\d_{\calE_\infty }(\Cech(X,\d_X),\Cech(Y,\d_Y)) \leqslant 2\dGH\left((X,\d_X),(Y,\d_Y)\right) \ .
	\end{aligned}
	\]
\end{theorem}

\begin{proof}
	It is the same proof that for \Cref{thm:stabilite} using the analogue (see \Cref{prop:homotopy_Einfty}) of \Cref{prop:homotopy_of_dga}.
	In positive characteristic the proposition follows from~\cite[Proposition 4.2 and 4.3]{Mandell2002} and this result is extended over $\mathbb{Z}$ (and therefore any coefficient ring) in~\cite[Section~1]{mandell2006cochains}; the case of $\rmN^*$ follows from that and the main construction of~\cite{bergerfresse}. Here notions of homotopy of algebras are with respect to the standard model structures, see~\cite{mandell2006cochains} for instance.
\end{proof} 

\begin{proposition}\label{prop:homotopy_Einfty}
	The functor  $\rmC^*:\Top^\op\rightarrow \Alg_{\calE_\infty}$ (respectively $\rmN^*_{\calE_\infty}: \sSet^\op  \rightarrow \Alg_{\calE_\infty}$) converts weak homotopy equivalences of spaces (resp. simplicial sets) to quasi-isomorphisms and homotopy classes of maps (resp. simplicial sets morphisms) to homotopy classes of morphisms of $\calE_\infty$-algebras.
\end{proposition}

Despite being the most interesting distance from a purely theoretical point of view, the $\calE_\infty$-interleaving distance is not really easily computable  for the moment, therefore we now introduce coarsest ones, which are more computer friendly.

\subsection{Positive characteristic and the Steenrod interleaving distance}\label{sec:Steenrod_distance} In this subsection, we fix $k=\bbF_p$, where $p$ is a prime.

\begin{definition}[Steenrod algebra $\calA_p$ (see {\cite[Section 1.1]{baues2006algebra}})]\label{def:Steenrod}
	Let $p$ be a prime number. The \emph{mod $p$-Steenrod algebra}, denoted by $\calA_p$, is the graded commutative algebra over $\bbF_p$ which is
	
	\begin{itemize}
		\item for $p=2$, generated by elements denoted $\Sq^n$ and called the \emph{Steenrod squares}, for $n\geqslant 1$, with cohomological degree $n$;
		\item for $p>2$, generated by elements denoted $\beta$, called the \emph{Bockstein}, of degree $1$, and $\rmP^n$ for $n\geqslant 1$ of degree $2n(p-1)$;
	\end{itemize}

	whose product satisfy the following relations, called the \emph{\'Adem relations}:

	\begin{itemize}
		\item for $p=2$ and for $0<h<2k$, then
			\[
				\Sq^h\Sq^k = \sum_{i=0}^{[\frac{h}{2}]} \binom{k-i-1}{h-2i}\Sq^{h+k-i}\Sq^i \ ,
			\]
		\item for $p>2$ and for $0<h<pk$, then
			\[
				\rmP^h\rmP^k = \sum_{i=0}^{[\frac{h}{p}]} (-1)^{h+i} \binom{(p-1)(k-i)-1}{h-pi}\rmP^{h+k-i}\rmP^i \ ,
			\]
			and 
			\begin{align*}
				\rmP^h\beta\rmP^k =  
				& \sum_{i=0}^{[\frac{h}{p}]} (-1)^{h+i} \binom{(p-1)(k-i)}{h-pi}\beta\rmP^{h+k-i}\rmP^i 
				  \ + \sum_{i=0}^{[\frac{h-1}{p}]} (-1)^{h+i-1}  \binom{(p-1)(k-i)-1}{h-pi-1}\rmP^{h+k-i}\beta\rmP^i 
			\end{align*}
	\end{itemize}
	We denote by $\calA_p\mbox{-}\Alg$, the category of commutative algebras over $\calA_p$. 
\end{definition}

\begin{remark}\label{R:Steenrodascup}
	The Steenrod algebra  also have a structure of Hopf algebra. Further, an important formula is given, for $x$ in $\rmH^*(x)$, by 
	\[ 
		\Sq^{|x|}(x)= x\cup x \ .
	\]
\end{remark}

The Steenrod algebra is the algebra of cohomological operations, i.e. all the natural transformations of degree $d$ for all $d$ in $\bbN$ :
\[
	\rmH^*(-,\bbF_p) \longrightarrow \rmH^{*+d}(-,\bbF_p) \ ,
\]	
as illustrated by the following theorem.

\begin{theorem}[Steenrod, \'Adem]
	Let $X$ be a topological space. The singular cohomology with coefficient in $\bbF_p$ is a commutative algebra over the $\calA_p$ Steenrod algebra, so we have the functor
	\[
		\rmH_{\calA_p}^*(-,\bbF_p) : \Top^\op \longrightarrow \calA_p\mbox{-}\Alg.
	\]
\end{theorem}

\begin{remark} The Steenrod algebra operations are the trace on the cohomology of the non-commutativity of the cup-product. In particular, they are determined by the $\calE_\infty$-structure on the cochains.
\end{remark}

Let us now define a new interleaving distance in positive characteristic.

\begin{definition}
	Let $X$ and $Y$ be two persistent spaces and let $p$ be a prime. The \emph{$\calA_p$-interleaving distance} is defined by
	\[
		\d_{\calA_p-\As}(X,Y)\coloneqq \d_{\calA_p\mbox{-}\Alg}(\rmH^*(X,\bbF_p),\rmH^*(Y,\bbF_p)) \ 
	\]  
	that it is the interleaving distance (see \Cref{def:interleaving}) computed in the category of $\calA_p$-algebras. Then we do not wish to specify a particular $p$, we will simply refer to this distance as \emph{Steenrod interleaving distance}.
\end{definition}

We have the following commutative diagram of functors
\begin{equation}\label{eq:diagmineq_ccrac_p}
\begin{tikzcd}
	&\Top^\op \ar[dl,"\rmN^*_{\calE_{\infty}}"'] \ar[dd,"\rmH^*_{\Com}" description] \ar[rd,"\rmN^*_{\As}"] 
	 & & \\
	\Alg_{\calE_\infty } \ar[rd,"\rmH^*"'] \ar[d,"\rmH_{\calA_p}^*"']
	\ar[rr,dotted, bend right=15]
	&&	\AsAlg \ar[r,"\For"] \ar[d,"\rmH"] 
	& \Ch_k \ar[d,"\rmH"] \\
	\calA_p\mbox{-}\Alg \ar[r,"\For"'] &\ComAlg \ar[r,"\For"'] &	\AsAlg \ar[r,"\For"'] & \grVect \;
\end{tikzcd}
\end{equation}
Using \Cref{lem::apply_functor} in characteristic $p$, we will deduce the following inequalities between the distances.

\begin{proposition} \label{prop::ineq_carac_p}
Let $X_\bullet,Y_\bullet:\sfR \to \Top$ be two persistent spaces. We have
\begin{align*}
	\d_{\ComAlg}(\rmH^*(X,\bbF_p),\rmH^*(Y,\bbF_p))  
	\leqslant\d_{\calA_p-\As}(X,Y)  
	\leqslant\d_{\calE_\infty,\bbF_p}(X,Y) \ 
\end{align*}
and for a commutative ring $k$ in \emph{any} characteristic 
\begin{align*}
	\d_{\ComAlg}(\rmH^*(X,k),\rmH^*(Y,k))  
	\leqslant \d_{\ho(\AsAlg)}(\rmC^*(X,k),\rmC^*(Y,k))
	\leqslant\d_{\calE_\infty, k}(X,Y) \ .
\end{align*}
\end{proposition}

\begin{proof}
	Propositions~\ref{prop:homotopy_of_dga} and~\ref{prop:homotopy_Einfty} implies that the functors $\rmN^*_{\calE_\infty}$ and $\rmN^*_{\As}$ passes to the homotopy category and thus diagram~\eqref{eq:diagmineq_ccrac_p} induces a commutative diagram:
	\begin{equation}\label{eq:diagmineq_ccrac_p_ho}
	\begin{tikzcd}
		&\Top^\op \ar[dl,"\rmN^*_{\calE_{\infty}}"'] \ar[dd,"\rmH^*_{\Com}" description] \ar[rd,"\rmN^*_{\As}"] 
		 & & \\
		\ho(\Alg_{\calE_\infty }) \ar[rd,"\rmH^*"'] \ar[d,"\rmH_{\calA_p}^*"']
		\ar[rr,dotted, bend right=15]
		&&	\ho(\AsAlg) \ar[r,"\For"] \ar[d,"\rmH"] 
		& \ho(\Ch_k) \ar[d,"\rmH"] \\
		\calA_p\mbox{-}\Alg \ar[r,"\For"'] &\ComAlg \ar[r,"\For"'] &	\AsAlg \ar[r,"\For"'] & \grVect \ ;
	\end{tikzcd}
	\end{equation}
	The two claimed string of inequalities then follows from \Cref{lem::apply_functor} applied to  diagram~\eqref{eq:diagmineq_ccrac_p_ho}.
\end{proof}

\begin{remark}[Effective computability of Steenrod distance] 
An important practical fact that makes the Steenrod distance appealing is that	there exists  \emph{algorithms} to compute the persistent Steenrod squares with coefficient in $\bbF_2$ (see \cite{aubrey2011persistent,medina2018persistence}).  
\end{remark}

We also have the following theorem of stability:

\begin{theorem}[Stability theorem - Commutative mod $p$ version]
	Let $(X,\d_X)$ and $(Y,\d_Y)$ be two metric spaces. We have the following inequality:
		\begin{align*}
		\d_{\calA_p-\As}(\calR(X,\d_X),\calR(Y,\d_Y)) \leqslant 2\dGH\left((X,\d_X),(Y,\d_Y)\right) \ ; \\
		\d_{\calA_p-\As}(\Cech(X,\d_X),\Cech(Y,\d_Y)) \leqslant 2\dGH\left((X,\d_X),(Y,\d_Y)\right) \ .
		\end{align*}
\end{theorem}

\begin{proof}
	It follows from \Cref{thm:stability_Einfty_version} and the first inequality in~\Cref{prop::ineq_carac_p}.
\end{proof}

\begin{example}\label{ex:Einfty_example}\label{ex55}
	Consider the compact spaces $X=S^3\vee S^5$ and $Y=\Sigma \bbC P^2$, the suspension of $\bbC P^2$. The spaces $X$ and $Y$ have the same cohomology. As $\rmH^*(X)$ is a sub-algebra of $\rmH^*(S^3)\times\rmH^*(S^5)$, and $Y$ is  a suspension, both spaces have trivial cup products in cohomology. In particular, their cohomology are the same as associative algebras.
	But, if we consider their cohomology with coefficient in $\bbF_2$ as module under the Steenrod algebra, then $\rmH^*(X)$ has only trivial Steenrod square while $\rmH^*(Y)$ has non-trivial ones. Note also that this difference is also detected by the $\calE_{\infty}$-structure on the cochain level. Therefore, we can proceed as in \Cref{ex:torus_sphere} or \Cref{ex:Borromeandiscret}. We can embed $X$ and $Y$ in $\bbR^8$ and take discretisations of $X$ and $Y$, depending of a small parameter $\alpha \ll 1$. We denote by $\widetilde{X}$ and $\widetilde{Y}$ respectively those discretisation, which are thus finite sets. We then get 
	\[
		\d_{\ComAlg}(\rmH^*(\calR(\widetilde{X})),\rmH^*(\calR(\widetilde{Y}))) <
		\d_{\calA_p-\As}(\calR(\widetilde{X}),\calR(\widetilde{Y})).
	\] 
	And the difference grows bigger and bigger when $\alpha$ goes to $0$ as the graded algebra distance will converge to $0$ while the Steenrod one will converge to a non-zero value, depending on the precise way they are embedded in some eucliden space $\bbR^n$, for instance $\bbR^8.$
	Similarly, \Cref{prop::ineq_carac_p} also implies 
	\[
		\d_{\ho\AsAlg}(\rmN^*(\calR(\widetilde{X})),\rmN^*(\calR(\widetilde{Y}))) <
		\d_{\calE_\infty}(\calR(\widetilde{X}),\calR(\widetilde{Y})) \ .
	\]
\end{example}
 
\subsection{Steenrod and $\calA_\infty$-transferred structure}

\begin{definition}[The interleaving distance $\d_{\calA_{p_\infty}}$]
	Let $p$ be a prime. We consider the subcategory of the homotopy category of $\calA_\infty$-algebras over the field $\bbF_p$ such that the objects are $\calA_\infty$-algebras over $\bbF_p$ which are also modules on the Steenrod algebra $\calA_p$, and morphisms are $\infty$-morphisms $f:A \rightsquigarrow B$ such that $f_1:A \rightarrow B$ are also a morphism of $\calA_p$-modules. We denote this category by $\ho(\calA_p-\AinfAlg)$. We denote by $\d_{\calA_{p_\infty}}$, the interleaving distance associated to this category.
\end{definition} 

\begin{remark}
	As $\calA_\infty$ is $\mathfrak{S}_n$-free, then this operad always encodes the notion of associativity up to homotopy  in characteristic $p>0$.
\end{remark}

\begin{definition}
	We define the \emph{$\calA_{p_\infty}$-algebra homology functor} as the composition
	\[
	\rmH_* \circ \rmT\rmC^*: \mathsf{fData}^\sfR \longrightarrow \ho\left(\calA_p-\AinfAlg\right)^{\sfR^\op}.
	\]
\end{definition}

\begin{example}	 \label{ex:80}
	Consider the same compact spaces $X$ and $Y$ as in \Cref{ex55}. The spaces $X$ and $Y$ have the same cohomology as a graded vector space, which is $k1\oplus kx \oplus ky$, where $1$ is in degree $0$, $x$ in degree $3$ and $y$ in degree $5$. We consider the transferred structure: all the higher products $m_i$ are trivial for degree reasons, only except if some of the variables are $1$. For instance, consider
	\[
	m_3: \rmH^p(S)\otimes \rmH^q(S) \otimes \rmH^r(S) \longrightarrow \rmH^{p+q+r-1}(S)
	\]
	(where $p,q$ and $r$ have to be in $\{0,3,5\}$ to have non-zero entries) and $S$ is the space $X$ or $Y$. Then, an imediate degree argument shows that  the product $m_3$ is trivial unless maybe $m_3(1,x,x),$ $m_3(x,1,x)$ or $m_3(x,x,1)=0$ (which are degree $5$). By \cite[Proposition 3.2.4.1]{lefevre2002categories}, as the cochain complex $\rmC^*(S)$ is a strictly unital $\calA_\infty$-algebra, then its homology is equivalent to a minimal model $A$ which is strictly unital and such that the $\calA_\infty$-morphism $i_\infty: A \rightsquigarrow\rmC^*(S)$ is strictly unital, so $m_3(1,a,b)=m_3(a,1,b)=m_3(a,b,1)=0$ for all $a,b$ in $\rmH^*(S)$ and similarly, all higher $m_k$ are null if at least one of the variable is $1$. So, the $A_\infty$-transferred structures on $\rmH^*(X)$ and $\rmH^*(Y)$ are trivial.
	
	As in \Cref{ex55}, we can  denote by $\widetilde{X}$ and $\widetilde{Y}$ respectively discretisations of $X$ and $Y$, which are thus finite sets. We then get 
	\[
	\d_{\calA_\infty,\bbF_p}(\calR(\widetilde{X}),\calR(\widetilde{Y})) <
	\d_{\calA_{p_\infty}}(\calR(\widetilde{X}),\calR(\widetilde{Y})).
	\] 
\end{example}

\begin{example}\label{ex:borromeanp} 
	Let us return again to \Cref{ex32}, that is the complements of $\beta$-thickenings of the borromean links and trivial entanglements of three circles in $S^3$. 
	Since the cohomology of these spaces is concentrated in degree less than 2, we obtain that the only non-zero Steenrod squares are given by $\Sq^1: \rmH^1(X)\to \rmH^2(X)$. The latter  is given by the self cup-product $\Sq^1(x)= x\cup x$, see \Cref{def:Steenrod} and \Cref{R:Steenrodascup}. In other words, the Steenrod squares of $X$ and $Y$ are the same. Similarly, the other Steenrod powers coincide for $X$ and $Y$. In particular, for the discretisation $\widetilde{X}$ and $\widetilde{Y}$, we obtain
	\[
			\d_{\calA_2-\As}(\calR(\widetilde{X}),\calR(\widetilde{Y}))<
			\d_{\calA_{2_\infty}}(\calR(\widetilde{X}),\calR(
			\widetilde{Y})).
	\]
\end{example}

\subsection{Distances in characteristic zero}
In this section, we fix $k$ a field of characteristic $0$. Numerous efficient tools have been developped in algebraic topology to deal with homotopy invariants specific to this case and we review here how they are related to our general formalism. In the characteristic zero case (and only this one), the homotopy theory of $\calE_\infty$-algebras is equivalent to that of algebras over another operad which encodes the structure of commutative associative algebra up to homotopy, called $\Com_\infty$, given by given by  Koszul duality theory (see \cite[Section 13.1.8]{LV2012} for $\Com_\infty$-algebras and \cite[Chapter 7]{LV2012} for Koszul duality theory). This operad is smaller than the operad of surjection (see \Cref{thm:bergerfresse}). One of the advantage of  $\Com_\infty$ is that we have a homotopy transfer theorem as for $\calA_\infty$ in~\Cref{sect:Ainfty}.
Therefore, as in \Cref{sub:contraction}, we have the functor 
\[
	\begin{array}{rcc}
		\rmH_* : \fTrans_{\coCh,\Com}^{\sfR^\op} & \longrightarrow & \ho\left(\infty\textsf{-}\Alg_{\Com_\infty}\right)^{\sfR^\op} \\
		(A,H,i,p,h)^\bullet & \longmapsto & (H,\{\mu_i\}_{i\in\N })^\bullet
	\end{array} \ .
\]
In particular, given $X_\bullet,Y_\bullet:\sfR \to \DCpx$ two finite filtered data, we can define their  $\Com_\infty$-interleaving distance: 
\[ 
	\d_{\ho(\infty\textsf{-}\Alg_{\Com_\infty})}(\rmH^*(X),\rmH^*(Y))\ .
\] 
The structure of $\Com_\infty$-algebra encodes more homotopy structure than that of $\calA_\infty$-algebra. However, by  \Cref{lem:Cinfty_inutile}, this new distance does not necessarily permit to differentiate more persistent spaces than the $A_\infty$ one.

\begin{lemma}\label{lem:Cinfty_inutile}
	Let $X_\bullet,Y_\bullet:\sfR \to \DCpx$ be two finite filtered data. We have 
	\[
		\d_{\ho(\infty\textsf{-}\Alg_{\Com_\infty})}(\rmH^*(X),\rmH^*(Y))=0
		\Longleftrightarrow
		\d_{\calA_\infty}(X,Y)=0 \ .
	\]
\end{lemma}

\begin{proof}
	Let $X_\bullet,Y_\bullet:\sfR \to \DCpx$ be two finite filtered data such that $\d_{\ho(\infty\textsf{-}\Alg_{\Com_\infty})}(\rmH^*(X),\rmH^*(Y))=0$. As in \Cref{thm:equivalence_homotopy_category}, by \cite[Theorem 11.4.8]{LV2012}, the homotopy category of differential graded commutative associative algebras and the homotopy category of $\Com_\infty$-algebras with the $\infty$-morphisms are equivalent. Therefore, by  \Cref{lem:distance_nulle}, $\rmH^*(X)$ and $\rmH^*(Y)$ are quasi-isomorphic as homotopy commutative algebras. By \cite[Theorem A]{campos2019lie}, this is equivalent to the fact that $\rmH^*(X)$ and $\rmH^*(Y)$ are quasi-isomorphic in the homotopy category of dg-associative algebras. Then by \Cref{thm:equivalence_homotopy_category} and \Cref{lem:distance_nulle}
	\[
		\d_{\ho(\infty\textsf{-}\Alg_{\Com_\infty})}(\rmH^*(X),\rmH^*(Y))=0
		\Longleftrightarrow
		\d_{\calA_\infty}(X,Y)=0 \ .
	\]
\end{proof}

\begin{remark}
	Since a $\Com_\infty$-algebra structure has a canonical underlying $\calA_\infty$-algebra structure, we automatically have an inequality 
 	\[
 		\d_{\ho(\infty\textsf{-}\Alg_{\Com_\infty})}(\rmH^*(X),\rmH^*(Y)) \geqslant \d_{\calA_\infty}(X,Y)
 	\]
	where the left hand side is the interleaving distance in the category of $\Com_\infty$-algebras. The \Cref{lem:Cinfty_inutile} suggests that the inequality might actually be an equality in many practical cases.
\end{remark}

\begin{remark}[The functor $\APL$]
	In characteristic zero, there is a functor
	\[
		\APL\colon \Top \longrightarrow \ComAlg \ ,
	\]
	(see \cite[Section~10]{FHT12}) such that, for any topological space $X$, there exists two natural quasi-isomorphisms of differential graded associative algebras
	\[
		\rmC^*(X) \overset{\sim}{\longrightarrow} \bullet \overset{\sim}{\longleftarrow} \APL(X) \ 
	\]
	(see \cite[Corollary 10.10]{FHT12}). As the functor $\APL$ (or the equivalent combinatorial model given by Felix \emph{et al.} in \cite{felix2009combinatorial}) is more computable, we expect that these constructions can be used to compute the $\calA_\infty$-interleaving distance for finite filtered data (see \Cref{prop:inequalities_Ainfty}).
\end{remark}

\begin{proposition}
	We have the following inequalities 
	\[
	\begin{tikzcd}[row sep = small]
	\d_{\calA_\infty,\Q}
	\ar[r,phantom, "=",sloped, description]
	& 
	\d_{\ho\AsAlg }
	\ar[r,phantom, "\geqslant",sloped, description]
	\ar[r,phantom, "_{(2)}", shift right=2.2ex]
	&
	\d_{\As,\Q} 
	\ar[r,phantom, "\geqslant" sloped, description] 
	\ar[r,phantom, "_{(1)}", shift right=2ex]
	& \d_{\grVect,\Q} 
	\end{tikzcd}
	\]
	for  for \emph{finite filtered data} (see \Cref{def:ffdata}). None of these inequalities are equalities in general. 
\end{proposition}
\begin{proof}
	We just use
	\begin{enumerate}
		\item \Cref{prop:inequalities} and \Cref{ex:torus_sphere};
		\item \Cref{prop:inequalities_Ainfty} and \Cref{ex32}.
	\end{enumerate}
\end{proof}

\subsection{The best of both worlds}
We have seen in~\Cref{prop::ineq_carac_p} that the $\calE_\infty$-interleaving distance is one of the finer distances that we can define, but this distance seems difficult to calculate because of the intricate structure of an algebra of the operad of surjection and its homotopy category. As the operad $\calE_\infty$ encodes in particular the cup-product and the Steenrod operations we can restrict to the following distances.

\begin{notation}
	We denote by $\mathbb{P}$, the set containing all prime numbers and $0$. By an abuse of notation, we denote by $\d_{\calA_{0_\infty}}$ (resp. $\d_{\calA_{0}-\As}$) the interleaving distance $\d_{\calA_\infty,\Q}$ (resp. $\d_{\As,\Q}$) and $\bbF_0=\Q$.
\end{notation}

\begin{defiprop}\label{def:distance_max}
	Let $X_\bullet,Y_\bullet:\sfR \to \DCpx$ be two finite filtered data and let $p$ and $q$  be two numbers in $\mathbb{P}$. We define two distances given by
	\begin{align*}
		\d_{p_\infty ,q_\infty}(X,Y)
		& \coloneqq 
		\max\big(
		\d_{\calA_{p_\infty}}(X,Y),
		\d_{\calA_{q_\infty}}(X,Y)
		\big) \ , \\
		\d_{\mathbb{P}}(X,Y)
		& \coloneqq  
		\underset{p\in \mathbb{P}}{\sup}~	(\d_{\calA_{p_\infty}}(X,Y)) \ .
	\end{align*}
	More generally, for any persistent spaces $Z$, $T$, we define
	\begin{align*}
		\d_{p_\infty ,q_\infty}(Z,T)
		& \coloneqq 
		\max\big(
		\d_{\calA_{p}-\ho\Alg_{\As}}(Z,T),
		\d_{\calA_{q}-\ho\Alg_{\As}}(Z,T)
		\big) \ , \\
		\d_{\mathbb{P}}(Z,T)
		& \coloneqq  
		\underset{p\in \mathbb{P}}{\sup}~	(\d_{\calA_p-\ho\Alg_{\As}}(Z,T) )\ .
	\end{align*}
	In both cases, for all $p$ and $q$ in $\mathbb{P}$, we have
	\[
		\d_{p_\infty ,q_\infty}(X,Y) \leqslant \d_{\mathbb{P}}(X,Y) \ ,
	\]
	which is not an equality in general (see \Cref{ex:distance_max}).
\end{defiprop}

\begin{example}\label{ex:distance_max} 
	Fix $p$ an odd prime and consider the real projective plane $\RP $: it has the following cohomology
	\[
		\rmH^0(\RP ,\Z) = \Z \ , \ \rmH^1(\RP ,\Z) = 0 \ , \ \rmH^2(\RP ,\Z) =  \Z/2\Z \ ,
	\]
	and
	\[
		\rmH^0(\RP ,\Z/2\Z) = \Z/2\Z \ , \ \rmH^1(\RP ,\Z/2\Z) = \Z/2\Z \ , \ \rmH^2(\RP ,\Z/2\Z) =  \Z/2\Z \ .
	\]	
	Let $X=\Sigma\RP$ be the suspension of $\RP $ and $Y$ be a closed ball in $\R^3$. The spaces $X$ and $Y$ have trivial cup product, but the real projective space has non trivial Steenrod square, and its suspension too. As in \Cref{ex:torus_sphere}, we can construct two discretisations $\widetilde{X}$ and $\widetilde{Y}$ of $X$ and $Y$ respectively, depending of a parameter $\alpha$ such that
	\begin{align*}
		&\d_{0,p}\big(\Cech(\widetilde{X}),\Cech(  \widetilde{Y})\big) < \d_{\mathbb{P}}\big(\Cech(\widetilde{X}),\Cech( \widetilde{Y})\big) \ ; \\
		&\d_{0,p}\big(\calR(\widetilde{X}),\calR(  \widetilde{Y})\big) < \d_{\mathbb{P}}\big(\calR(\widetilde{X}),\calR(  \widetilde{Y})\big) \ .
	\end{align*}
	Similar examples can be obtained in other characteristic using lens spaces in place of projective spaces.
\end{example}

\begin{proposition}
	Let $p$ and $q$ be in $\mathbb{P}$. We have the following inequalities
	\[
	\begin{tikzcd}[row sep = small,  column sep = large]
	\d_{p_\infty ,q_\infty} 
	\ar[r,phantom, "\geqslant" sloped, description] 
	&
	\d_{\ho\Alg_{\As},\bbF_p}
	\end{tikzcd}
	\]
	which are not equalities in general
\end{proposition}
\begin{proof}
	The inequalities follows by definition of $\sup$, and it can be proved to be strict in general by using \Cref{ex:distance_max} with $p=0$ or $p$ odd and $q=2$. By attaching a disk $D^{n+1}$ on a sphere along a map of degree $p$ to get a space $Y_p$ allows to distinguish similarly this space to the disk in characteristic $p$ but not in characteristic $q\neq p$. 
\end{proof}

\begin{remark}	
	From the same way as in~\ref{def:distance_max}, we can define similar distances which take into account of  three or more characteristics: such distances are always smaller than the distance $\d_{\mathbb{P}}$ and we can construct spaces as in \Cref{ex:distance_max} for which $\d_{\mathbb{P}}$ is strictly bigger.
\end{remark}

\section{R\'esum\'e and proof of Theorem \ref{thm:maintheoremB}}\label{sec::resume}
In this paper we have defined several distances between finite filtered data, using the cohomology of their associated persistent spaces. In this section we finish to compare them all and in particular prove~\Cref{thm:maintheoremB}. We fix $p$ a prime and we recall that the $\calA_\infty$ interleaving distance is defined for the cohomology of persistent data with coefficient in a field of characteristic zero. Further we consider the $\calE_\infty$-interleaving distance of~\Cref{sec:Einfty} with value in $\mathbb{Z}$ for coefficient.
We can summarize all the interleaving distances constructed in this paper in the following diagram of distances. 

\begin{theorem}\label{thm:resume} 
	Consider  the various interleaving distances introduced in the paper and let $p$ be in $\mathbb{P}$.
\begin{enumerate} \item\label{th93item1} There is a string of inequalities 
\[
\begin{tikzcd}[row sep = small,  column sep = large]
&&&
\d_{\calA_{p}-\As}
\ar[rd,phantom, "\geqslant",sloped, description]
\ar[rd,phantom, "_{(\ref{item:4})}", shift right=2.2ex]
&& \\
\d_{\calE_\infty}
\ar[r,phantom, "\geqslant" sloped, description] 
&
\d_{\mathbb{P}}
\ar[r,phantom, "\geqslant" sloped, description] 
&
\d_{{p_\infty}} 
\ar[rd,phantom, "\geqslant" sloped, description] 
\ar[rd,phantom, "_{(\ref{item:7})}", shift right=2ex]
\ar[ru,phantom, "\geqslant" sloped, description] 
\ar[ru,phantom, "_{(\ref{item:5})}", shift right=2ex]
&~&
\d_{\As,\bbF_p} 
\ar[r,phantom, "\geqslant" sloped, description] 
\ar[r,phantom, "_{(\ref{item:1})}", shift right=2ex]
& \d_{\grVect,\bbF_p} \\
&&&
\d_{\calA_\infty,\bbF_p} 
\ar[ru,phantom, "\geqslant" sloped, description] 
\ar[ru,phantom, "_{(\ref{item:5})}", shift right=2ex]
&&
\end{tikzcd} \ .
\]

for \emph{finite filtered data} (\Cref{def:ffdata}). None of these inequalities are equalities in general. 

\item More generally, for arbitrary \emph{persistent spaces}, there is a string of inequalities  

\[
\begin{tikzcd}[row sep = small,  column sep = large]
&&&
\d_{\calA_{p}-\As}
\ar[rd,phantom, "\geqslant",sloped, description]
\ar[rd,phantom, "_{(\ref{item:4})}", shift right=2.2ex]
&& \\
&
\d_{\calE_\infty}
\ar[r,phantom, "\geqslant" sloped, description] 
&
\d_{\calA_p-\ho\Alg_{\As}} 
\ar[rd,phantom, "\geqslant" sloped, description] 
\ar[rd,phantom, "_{(\ref{item:7})}", shift right=2ex]
\ar[ru,phantom, "\geqslant" sloped, description] 
\ar[ru,phantom, "_{(\ref{item:5})}", shift right=2ex]
&~&
\d_{\As,\bbF_p} 
\ar[r,phantom, "\geqslant" sloped, description] 
\ar[r,phantom, "_{(\ref{item:1})}", shift right=2ex]
& \d_{\grVect,\bbF_p} \\
&&&
\d_{\ho\Alg_{\As},\bbF_p} 
\ar[ru,phantom, "\geqslant" sloped, description] 
\ar[ru,phantom, "_{(\ref{item:5})}", shift right=2ex]
&&
\end{tikzcd} \ .
\]
Further,  these inequalities are not equalities in general. 
\end{enumerate}
\end{theorem}

In particular, those distances are not equal for Rips or  \v{C}ech complex associated to  discretisations of  spaces.

\begin{remark}
	Note that the Main \Cref{thm:maintheoremB} is nothing but a special case of \Cref{thm:resume}.
\end{remark}

\begin{remark}\label{rem:homotopyinterleaving}
	In~\cite{BL17} Blumberg and Lesnick have studied the homotopy interleaving distance $d_{HI}$ for persistant (nice) topological spaces. Roughly speaking, it amounts to studying interleaving in the (weak) homotopy category of functors from $\R$ to $\Top$. The  $\calE_\infty$-functor $\rmN^*_{\calE_\infty}\colon \Top^{\op}\to \Alg_{\calE_\infty}$~\Cref{def:NEinfty} maps persistent topological spaces to persistent $\calE_\infty$-algebras. The homotopy invariance \Cref{prop:homotopy_Einfty} and stability \Cref{thm:stability_Einfty_version}, together with \cite[Theorem 1.7]{BL17} implies that 
	\[\d_{\calE_\infty} \,\leqslant \, \d_{HI}. \] 
	In view of~\Cref{rem:mandell}, one can expect those two distances to be close from each other in many cases.
\end{remark}

\begin{remark}
	If $p=0$, then, the diagram in \Cref{thm:resume} \eqref{th93item1} can be rewritten as follows
	\[
	\begin{tikzcd}[row sep = small,  column sep = large]
	\d_{\calE_\infty}
	\ar[r,phantom, "\geqslant" sloped, description] 
	&
	\d_{\mathbb{P}}
	\ar[r,phantom, "\geqslant" sloped, description] 
	&
	\d_{\calA_\infty,\Q} 
	\ar[r,phantom, "\geqslant" sloped, description] 
	&
	\d_{\As,\Q} 
	\ar[r,phantom, "\geqslant" sloped, description] 
	& \d_{\grVect,\Q} 
	\end{tikzcd}\ .
	\]
\end{remark}
\begin{proof}[Proof of \Cref{thm:resume}] 
The first string of inequalities as well as the fact that they are strict in general for finite data sets follows from 
\begin{enumerate}
	\item \label{item:1} \Cref{prop:inequalities} and \Cref{ex:torus_sphere};
	\item \label{item:4} forgetting the Steenrod power operations and \Cref{ex:Einfty_example};
	\item \label{item:7} forgetting the Steenrod power operations and \Cref{ex:80};
	\item \label{item:5} forgetting the $\calA_\infty$-structure  and  \Cref{ex:borromeanp}.
\end{enumerate}	
The same argument and part (2) of \Cref{prop:inequalities_Ainfty} yields the inequalities of the second part as well as the fact that they are strict in general. By the universal coefficient theorem~\cite{weibelhomological}, any  $\varepsilon$-interleaving between normalized chain complex $\rmN^*(X,\mathbb{Z})$ and $\rmN^*(Y,\mathbb{Z})$ induces a $\varepsilon$-interleaving between $\rmN^*(X,k)$ and $\rmN^*(Y,k)$ for any field $k$. This proves the inequality $\d_{\calE_\infty} \geqslant \d_{\mathbb{P}}$. 

Note that for general persistent spaces the counter-examples are  easier to produce than for those arising from Rips or \v{C}ech complexes. Indeed, it suffices to take $X$ and $Y$ two topological spaces such that $\rmN^*(X) $ and $\rmN^*(Y)$ are equivalent in the category $\sfC$ but distinct in the category $\sfD$ (here the categories are any of the one we consider for interleavings) to obtain two persistent spaces such that $\d_\sfC < \d_\sfD$. Indeed one can consider the constant functions $f:X\to \R$ and $g:Y\to \R$ that sends the spaces to the point $0$ in $\R$. Then the persistent spaces associated to the function satisfy the strict inequalities.
\end{proof}

\begin{remark}
	For many practical applications, it seems that the $\calA_\infty$, the $\calA_2$ and the $\calA_3$ interleaving distances will be useful: in fact, the spaces which are not differentiated by these distances but which are differentiated by other refined ones will be complicated to compute algorithmically, at least for the moment. Since it is possible to compute algorithmically Steenrod squares as well as the $\calA_\infty$-structure in characteristic 2 for finite data, we believe that the distance 
	$d_{{2_\infty}}$ is a promissing and reachable  lift of the classical interleaving distance judging from \Cref{Prop:torusvswedge}, \Cref{ex:torus_sphere_comput}, \Cref{ex:Borromeandiscret}, \Cref{Prop:Borromean}.
\end{remark}

\bibliographystyle{alpha}
\bibliography{biblio_persistence}

\newcommand{\etalchar}[1]{$^{#1}$}
\begin{thebibliography}{CPRNW19}

\bibitem[Aub11]{aubrey2011persistent}
HB~Aubrey.
\newblock {\em Persistent Cohomology Operations}.
\newblock Duke University, 2011.

\bibitem[Bau06]{baues2006algebra}
Hans-Joachim Baues.
\newblock {\em The algebra of secondary cohomology operations}.
\newblock Springer, 2006.

\bibitem[BDSS17]{bubenik2017interleaving}
Peter Bubenik, Vin De~Silva, and Jonathan Scott.
\newblock Interleaving and gromov-hausdorff distance.
\newblock {\em arXiv preprint arXiv:1707.06288}, 2017.

\bibitem[{Bel}19]{belchi2017}
{Belch{\'i}, Francisco}.
\newblock Optimising the topological information of the
  {$A_\infty$-persistence} groups.
\newblock {\em Discrete {\&} Computational Geometry}, 62(1):29--54, 2019.

\bibitem[BF04]{bergerfresse}
Clemens Berger and Benoit Fresse.
\newblock Combinatorial operad actions on cochains.
\newblock In {\em Mathematical Proceedings of the Cambridge Philosophical
  Society}, volume 137, pages 135--174. Cambridge University Press, 2004.

\bibitem[BG18]{Berk18}
Nicolas Berkouk and Gr{\'e}gory Ginot.
\newblock A derived isometry theorem for constructible sheaves on $\mathbb{R}$.
\newblock {\em arXiv preprint arXiv:1805.09694}, 2018.

\bibitem[BL13]{bauer2013induced}
Ulrich Bauer and Michael Lesnick.
\newblock Induced matchings and the algebraic stability of persistence
  barcodes.
\newblock {\em arXiv preprint arXiv:1311.3681}, 2013.

\bibitem[BL17]{BL17}
Andrew~J. Blumberg and Michael Lesnick.
\newblock Universality of the homotopy interleaving distance, 2017.

\bibitem[BM15]{belchi2015}
Francisco Belch\'{i} and Aniceto Murillo.
\newblock ${A}_\infty$-persistence.
\newblock {\em Appl. Algebra Engrg. Comm. Comput.}, 26(1-2):121--139, 2015.

\bibitem[BMM20]{BMM20}
Urtzi {Buijs}, Jos\'e~M. {Moreno-Fern\'andez}, and Aniceto {Murillo}.
\newblock {\(A_\infty\) structures and Massey products}.
\newblock {\em {Mediterr. J. Math.}}, 17(1):15, 2020.
\newblock Id/No 31.

\bibitem[BP19]{BerPetit}
Nicolas Berkouk and Fran\c{c}ois Petit.
\newblock Ephemeral persistence modules and distance comparison.
\newblock {\em arXiv preprint arXiv:1902.09933}, 2019.

\bibitem[BS14]{bubenik2014categorification}
Peter Bubenik and Jonathan~A Scott.
\newblock Categorification of persistent homology.
\newblock {\em Discrete \& Computational Geometry}, 51(3):600--627, 2014.

\bibitem[BS19]{belchi2019}
Francisco Belch{\'i} and Anastasios Stefanou.
\newblock ${A}_\infty$ persistent homology estimates the topology from
  pointcloud datasets.
\newblock {\em arXiv preprint arXiv:1902.09138}, 2019.

\bibitem[CDSO14]{chazal2014stability}
Fr{\'e}d{\'e}ric Chazal, Vin De~Silva, and Steve Oudot.
\newblock Persistence stability for geometric complexes.
\newblock {\em Geometriae Dedicata}, 173(1):193--214, 2014.

\bibitem[CO08]{Towards}
Fr{\'e}d{\'e}ric Chazal and Steve~Y. Oudot.
\newblock Towards persistence-based reconstruction in euclidean spaces.
\newblock In {\em Proceedings of the twenty-fourth annual symposium on
  Computational geometry}, pages 232--241. ACM, 2008.

\bibitem[CPRNW19]{campos2019lie}
Ricardo Campos, Dan Petersen, Daniel Robert-Nicoud, and Felix Wierstra.
\newblock Lie, associative and commutative quasi-isomorphism.
\newblock {\em arXiv preprint arXiv:1904.03585}, 2019.

\bibitem[CSEH07]{cohen2007stability}
David Cohen-Steiner, Herbert Edelsbrunner, and John Harer.
\newblock Stability of persistence diagrams.
\newblock {\em Discrete \& Computational Geometry}, 37(1):103--120, 2007.

\bibitem[CZCG05]{Shape}
Gunnar Carlsson, Afra Zomorodian, Anne Collins, and Leonidas~J. Guibas.
\newblock Peristent barcodes for shapes.
\newblock {\em International Journal of Shape Modeling}, 11(02):149--187, 2005.

\bibitem[EH08]{Edel10}
Herbert Edelsbrunner and John Harer.
\newblock Persistent homology-a survey.
\newblock {\em Contemporary mathematics}, 453:257--282, 2008.

\bibitem[EH10]{edelsbrunner2010computational}
Herbert Edelsbrunner and John Harer.
\newblock {\em Computational topology: an introduction}.
\newblock American Mathematical Soc., 2010.

\bibitem[FHT12]{FHT12}
Yves F{\'e}lix, Stephen Halperin, and J-C Thomas.
\newblock {\em Rational homotopy theory}, volume 205.
\newblock Springer Science \& Business Media, 2012.

\bibitem[FJP09]{felix2009combinatorial}
Yves Felix, Barry Jessup, and Paul-Eug{\`e}ne Parent.
\newblock The combinatorial model for the {S}ullivan functor on simplicial
  sets.
\newblock {\em Journal of Pure and Applied Algebra}, 213(2):231--240, 2009.

\bibitem[Hat02]{Hatcher}
Allen Hatcher.
\newblock Algebraic topology. 2002.
\newblock {\em Cambridge UP, Cambridge}, 606(9), 2002.

\bibitem[Her18]{Herscovich2018}
Estanislao Herscovich.
\newblock A higher homotopic extension of persistent (co)homology.
\newblock {\em J. Homotopy Relat. Struct.}, 13(3):599--633, 2018.

\bibitem[Jam95]{james1995handbook}
Ioan~Mackenzie James.
\newblock {\em Handbook of algebraic topology}.
\newblock Elsevier, 1995.

\bibitem[JCR{\etalchar{+}}17]{jeitziner2017two}
Rachel Jeitziner, Mathieu Carri{\`e}re, Jacques Rougemont, Steve~Y. Oudot,
  Kathryn Hess, and Cathrin Brisken.
\newblock {T}wo-{T}ier {M}apper: a user-independent clustering method for
  global gene expression analysis based on topology.
\newblock {\em arXiv preprint arXiv:1801.01841}, 2017.

\bibitem[KDS{\etalchar{+}}18]{kanari2018topological}
Lida Kanari, Pawe{\l} D{\l}otko, Martina Scolamiero, Ran Levi, Julian
  Shillcock, Kathryn Hess, and Henry Markram.
\newblock A topological representation of branching neuronal morphologies.
\newblock {\em Neuroinformatics}, 16(1):3--13, 2018.

\bibitem[KS18]{Kash18}
Masaki Kashiwara and Pierre Schapira.
\newblock Persistent homology and microlocal sheaf theory.
\newblock {\em Journal of Applied and Computational Topology}, 2018.

\bibitem[LH02]{lefevre2002categories}
Kenji Lefevre-Hasegawa.
\newblock {\em Sur les {{A}$_\infty$}-cat{\'e}gories}.
\newblock PhD thesis, Ph. D. thesis, Universit{\'e} Paris 7, UFR de
  Math{\'e}matiques, 2003, math. CT/0310337, 2002.

\bibitem[LV12]{LV2012}
Jean-Louis Loday and Bruno Vallette.
\newblock {\em Algebraic operads}, volume 346.
\newblock Springer Science \& Business Media, 2012.

\bibitem[Man02]{Mandell2002}
Michael~A. Mandell.
\newblock Equivariant {$p$}-adic homotopy theory.
\newblock {\em Topology Appl.}, 122(3):637--651, 2002.

\bibitem[Man06]{mandell2006cochains}
Michael~A. Mandell.
\newblock Cochains and homotopy type.
\newblock {\em Publications Math{\'e}matiques de l'Institut des Hautes
  {\'E}tudes Scientifiques}, 103(1):213--246, 2006.

\bibitem[McC06]{mccleary06}
John McCleary.
\newblock {\em A first course in topology: Continuity and dimension},
  volume~31.
\newblock American Mathematical Soc., 2006.

\bibitem[MM18]{medina2018persistence}
Anibal~M. Medina-Mardones.
\newblock Persistence steenrod modules.
\newblock {\em arXiv preprint arXiv:1812.05031}, 2018.

\bibitem[MS17]{manetti2017projective}
Marco Manetti and Chiara Spagnoli.
\newblock The projective model structure on contractions.
\newblock {\em arXiv preprint arXiv:1703.03231}, 2017.

\bibitem[Oud15a]{Oudo15}
Steve~Y. Oudot.
\newblock {\em Persistence theory: from quiver representations to data
  analysis}, volume 209.
\newblock American Mathematical Society Providence, RI, 2015.

\bibitem[Oud15b]{Oudot}
Steve~Y. Oudot.
\newblock {\em Persistence theory: from quiver representations to data
  analysis}, volume 209.
\newblock American Mathematical Society Providence, RI, 2015.

\bibitem[RMA09]{real2009cell}
Pedro Real and Helena Molina-Abril.
\newblock Cell {AT}-models for digital volumes.
\newblock In {\em International Workshop on Graph-Based Representations in
  Pattern Recognition}, pages 314--323. Springer, 2009.

\bibitem[RT00]{rudyak2000thom}
Yuli Rudyak and Aleksy Tralle.
\newblock On {T}hom spaces, {M}assey products, and nonformal symplectic
  manifolds.
\newblock {\em International Mathematics Research Notices}, 2000(10):495--513,
  2000.

\bibitem[{The}21]{gudhiweb}
{The GUDHI Project}.
\newblock {\em {GUDHI} User and Reference Manual}.
\newblock {GUDHI Editorial Board}, {3.4.1} edition, 2021.

\bibitem[TO06]{tralle2006symplectic}
Alesky Tralle and John Oprea.
\newblock {\em Symplectic manifolds with no K{\"a}hler structure}.
\newblock Number 1661. Springer, 2006.

\bibitem[Wei94]{weibelhomological}
Charles~A. Weibel.
\newblock {\em An introduction to homological algebra}, volume~38 of {\em
  Cambridge Studies in Advanced Mathematics}.
\newblock Cambridge University Press, Cambridge, 1994.

\end{thebibliography}

\end{document}